\documentclass[a4paper, 12pt]{amsart}
 \usepackage{amsmath, amsthm, amssymb, amsfonts}
 \usepackage[foot]{amsaddr}
\usepackage[utf8]{inputenc}
\usepackage[normalem]{ulem}
\usepackage[colorlinks = true,
            linkcolor = black,
            urlcolor  = blue,
            citecolor = red,
            anchorcolor = red,
            hidelinks]{hyperref}
\usepackage{url}
\usepackage{longtable} 
\usepackage{enumitem}
\usepackage{relsize}
\usepackage{bm}
\usepackage[dvipsnames]{xcolor}
\usepackage{graphicx}
\usepackage{mathabx}
\usepackage{mathdots}
\usepackage{mathtools}
\usepackage{color}
\usepackage{tikz-cd,tikz}
\usepackage{braket}
\usepackage{stmaryrd}
\usepackage{amssymb}
\usepackage{tipa}
\usepackage[margin=1.2 in]{geometry}
\usepackage{appendix}
\usepackage{physics}
\usepackage{mathrsfs}  

\usetikzlibrary{arrows, decorations.markings}
\usetikzlibrary{positioning}
\tikzstyle{vecArrow} = [thick, decoration={markings,mark=at position
   1 with {\arrow[semithick]{open triangle 60}}},
   double distance=1.4pt, shorten >= 5.5pt,
   preaction = {decorate},
   postaction = {draw,line width=1.4pt, white,shorten >= 4.5pt}]
\tikzstyle{innerWhite} = [semithick, white,line width=1.4pt, shorten >= 4.5pt]
\tikzstyle{vecEq} = [thick,
   double distance=1.4pt,
   preaction = {decorate},
   postaction = {draw,line width=1.4pt, white,shorten >= 4.5pt}]
\usetikzlibrary{matrix}
\usetikzlibrary{shapes}
\usetikzlibrary{arrows,decorations.markings, tikzmark}
\usepackage{stmaryrd}

\pgfarrowsdeclare{bad to}{bad to}
{
  \pgfarrowsleftextend{-2\pgflinewidth}
  \pgfarrowsrightextend{\pgflinewidth}
}
{
  \pgfsetlinewidth{0.8\pgflinewidth}
  \pgfsetdash{}{0pt}
  \pgfsetroundcap
  \pgfsetroundjoin
  \pgfpathmoveto{\pgfpoint{-3\pgflinewidth}{4\pgflinewidth}}
  \pgfpathcurveto
  {\pgfpoint{-2.75\pgflinewidth}{2.5\pgflinewidth}}
  {\pgfpoint{0pt}{0.25\pgflinewidth}}
  {\pgfpoint{0.75\pgflinewidth}{0pt}}
  \pgfpathcurveto
  {\pgfpoint{0pt}{-0.25\pgflinewidth}}
  {\pgfpoint{-2.75\pgflinewidth}{-2.5\pgflinewidth}}
  {\pgfpoint{-3\pgflinewidth}{-4\pgflinewidth}}
  \pgfusepathqstroke
}

\def\M{\mathcal{M}}
\def\PP{\mathbb{P}}

\theoremstyle{definition}
\newtheorem{definition}{Definition}[section]
\newtheorem{assumption}[definition]{Assumption}
\newtheorem{theorem}[definition]{Theorem}
\newtheorem{example}[definition]{Example}
\newtheorem{proposition}[definition]{Proposition}

\newtheorem{corollary}[definition]{Corollary}

\newtheorem{lemma}[definition]{Lemma}
\newtheorem{conjecture}[definition]{Conjecture}

\newtheorem{remark}[definition]{Remark}

\newtheorem*{claim}{Claim}

\newtheorem{thm}{Theorem}

\newcommand\numberthis{\addtocounter{equation}{1}\tag{\theequation}}

\newcommand{\BC}{{\mathbb{C}}}
\newcommand{\BD}{{\mathbb{D}}}
\newcommand{\BE}{{\mathbb{E}}}
\newcommand{\BF}{{\mathbb{F}}}
\newcommand{\BG}{{\mathbb{G}}}

\newcommand{\BL}{{\mathbb{L}}}

\newcommand{\BP}{{\mathbb{P}}}
\newcommand{\BQ}{{\mathbb{Q}}}
\newcommand{\BR}{{\mathbb{R}}}

\newcommand{\BZ}{{\mathbb{Z}}}

\newcommand{\CA}{{\mathcal A}}

\newcommand{\CE}{{\mathcal E}}
\newcommand{\CF}{{\mathcal F}}
\newcommand{\CG}{{\mathcal G}}

\newcommand{\CM}{{\mathcal M}}
\newcommand{\CN}{{\mathcal N}}
\newcommand{\CO}{{\mathcal O}}
\newcommand{\CP}{{\mathcal P}}
\newcommand{\CQ}{{\mathcal Q}}

\newcommand{\CS}{{\mathcal S}}

\newcommand{\CU}{{\mathcal U}}
\newcommand{\CV}{{\mathcal V}}

\newcommand{\CZ}{{\mathcal Z}}

\newcommand{\1}{\mathbf 1}

\DeclareMathOperator{\Hilb}{Hilb}

\DeclareMathOperator{\id}{id}
\DeclareMathOperator{\coker}{coker}

\newcommand{\Ext}{\mathcal{E}\text{xt}}

\newcommand{\Pic}{\mathop{\rm Pic}\nolimits}

\newcommand{\JS}{0+}

\newcommand{\Jac}{\mathop{\rm Jac}\nolimits}

\newcommand{\ovZ}{\overline{0}}
\newcommand{\ovO}{\overline{1}}

\newcommand{\End}{{\rm End}}

\newcommand{\pa}{\mathop{\rm pa}\nolimits}

\newcommand{\rel}{\rm rel}

\newcommand{\bF}{{\mathsf{F}}}

\newcommand{\bL}{{\mathsf{L}}}

\newcommand{\Linv}{\mathsf{L}_{\inv}}
\newcommand{\bE}{{\mathsf{E}}}
\newcommand{\bI}{{\mathsf{I}}}
\newcommand{\bS}{{\mathsf{S}}}
\newcommand{\bT}{{\mathsf{T}}}

\newcommand{\bR}{\mathsf{R}}

\newcommand{\pt}{{\mathsf{pt}}}

\newcommand{\congpf}{\xymatrix@1@=15pt{\ar[r]^-\sim&}}

\newcommand{\MX}{\mathcal{M}_X}

\newcommand{\Malpha}{\mathcal{M}_{\alpha}}

\newcommand{\ch}{\mathrm{ch}}
\newcommand{\chh}{\mathrm{ch}^{\mathrm{H}}}
\newcommand{\td}{\mathrm{td}}
\newcommand{\Coh}{\mathrm{Coh}}

\newcommand{\vir}{\mathrm{vir}}

\newcommand{\sst}{\text{sst}}

\newcommand{\SSym}{\text{SSym}}
\newcommand{\sym}{\text{sym}}
\newcommand{\ssym}{\text{ssym}}

\newcommand{\RHom}{\mathrm{RHom}}
\newcommand{\HH}{\mathrm{H}}
\newcommand{\rk}{\mathrm{rk}}

\newcommand{\rig}{\mathrm{rig}}
\newcommand{\inv}{{\mathrm{wt}_0}}
\newcommand{\inva}{\mathrm{inv}}

\newcommand{\tp}{\mathrm{top}}
\newcommand{\tCH}{\text{CH}}
\newcommand{\T}{\mathrm{T}}

\newcommand{\Quot}{\text{Quot}}

\usepackage{tikz}
\usepackage{tikz-cd}
\usetikzlibrary{decorations.pathmorphing}
\AtBeginDocument{%
	\def\MR#1{}
}

\DeclarePairedDelimiter{\floor}{\lfloor}{\rfloor}

\renewcommand{\Ext}{\textup{Ext}}
\newcommand{\Dpa}{\BD^{X, \pa}}

\begin{document}

\baselineskip=16pt
\parskip=5pt

    \title{Virasoro constraints for moduli of sheaves and vertex algebras}

\author{Arkadij Bojko}
\email{arkadij.bojko@math.ethz.ch}

\author{Woonam Lim}

\email{woonam.lim@math.ethz.ch}

\author{Miguel Moreira}
\email{miguel.moreira@math.ethz.ch}

\address{ETH Z\"urich, Department of Mathematics, Switzerland}

\date{\today}
\maketitle
\begin{abstract}
  In enumerative geometry, Virasoro constraints were first conjectured in Gromov-Witten theory with many new recent developments in the sheaf theoretic context. In this paper, we rephrase the sheaf-theoretic Virasoro constraints in terms of primary states coming from a natural conformal vector in Joyce’s vertex algebra. This shows that Virasoro constraints are preserved under wall-crossing. As an application, we prove the conjectural Virasoro constraints for moduli spaces of semistable torsion-free sheaves on any curve and on surfaces with only $(p,p)$ cohomology classes by reducing the statements to the rank 1 case.
\end{abstract}
\setcounter{tocdepth}{1} 
\tableofcontents

\section{Introduction}
This paper concerns the Virasoro constraints on sheaf counting theories. Given a moduli space of sheaves $M$ with a virtual fundamental class $[M]^\vir$ we may produce numerical invariants by integrating natural cohomology classes -- called descendents -- against the virtual fundamental class. The Virasoro conjecture predicts that these numerical invariants are constrained by some explicit and universal relations. The main result of this paper is proof of those constraints in the following cases: 

\begin{thm}[=Theorem \ref{thm: main result of the paper}]\label{thm: main} 
The Virasoro constraints (Conjecture \ref{conj: virasorosheaves}) hold for the following moduli spaces when there are no strictly semistable sheaves:
\begin{enumerate}
\item The moduli spaces $M_C(r,d)$ of slope semistable bundles on any smooth projective curve $C$.
\item The moduli spaces $M_S^H(r,c_1, c_2)$ of slope or Gieseker semistable torsion-free sheaves on any smooth projective surface $S$ with $h^{1,0}(S)=h^{2,0}(S)=0$.
\item  Assuming that a necessary wall-crossing formula holds (see Assumption \ref{ass:WCpair}), the moduli spaces $M_S^H(\beta, n)$ of slope semistable one-dimensional sheaves on any smooth projective surface $S$ with $h^{1,0}(S)=h^{2,0}(S)=0$.\footnote{In \cite{Pi22}, the ring structure of the cohomology of the moduli spaces of one-dimensional sheaves
on $S = \BP^2$ is studied in terms of descendents. It would be interesting to find out whether the
Virasoro constraints are implied by the relations they prove in \cite[Proposition 2.6]{Pi22}.}
\end{enumerate}

\end{thm}

Theorem \ref{thm: main result of the paper} also proves an analog of (1), (2) and (3) in the presence of strictly semistable sheaves as stated in Conjecture \ref{conj:actualinvvir}. The case (2) of the theorem solves the conjecture of D. van Bree in \cite{bree} (see Remark \ref{rem: normalizedpt} for a comparison of our formulation with van Bree's). For $S=\PP^2$ and $\PP^1\times\PP^1$, an alternative proof of (2) with Gieseker stability was given by the first author in \cite{BoVir} without relying on the previous result for $\Hilb^n(S)$ in \cite{moop}. The formulation of the other two cases is new; indeed, we provide a very general conjecture that includes many other interesting cases. The proof of Assumption \ref{ass:WCpair} making (3) self-contained will be addressed in the future.

Our work relies in a fundamental way on the vertex algebra that D. Joyce recently introduced \cite{Jo17,GJT, Jo21} to study the wall-crossing of moduli spaces of sheaves. We explain how the Virasoro constraints can be naturally formulated in this language.

\begin{thm}[=Theorem \ref{thm: duality} and Theorem \ref{thm: virasorophysical}]\label{thm: B}
Let $X$ be a curve or a surface with $p_g=0$. There is a natural conformal element $\omega$ in the vertex algebra $V_\bullet^{\pa}$ for which the corresponding Virasoro operators $L_n^{\pa}$ are dual to the Virasoro operators $\bL_n^{\pa}\colon \BD^{X,\pa}\to \BD^{X,\pa}$ for $n\geq -1$ on the pair descendent algebra. A moduli space of sheaves (or pairs) satisfies the Virasoro constraints if and only if its class in $\widecheck V_\bullet^{\pa}$ (or $V_\bullet^{\pa}$) is a primary state.
\end{thm}

As a consequence of this description, the wall-crossing machinery developed by Joyce can be used to prove a compatibility between wall-crossing and the Virasoro constraints.
\begin{thm}[=Propositions \ref{prop: wallcrossingcompatibility1} and \ref{prop: wallcrossingcompatibility2}]\label{thm: C}
The Virasoro constraints are compatible with wall-crossing.
\end{thm}
Theorem \ref{thm: C} also holds in the setting of quivers as was shown by the first author in \cite[Theorem 1.2]{BoVir}. This was used to conclude Virasoro constraints for moduli spaces of semistable (framed) representations of quivers in loc. cit.

Our proof of Theorem \ref{thm: main} uses this wall-crossing compatibility to reduce the statement to the rank 1 case; for curves, the rank 1 case can be shown directly or via Joyce-Song wall-crossing, and for surfaces, it was already proven in \cite{moop, moreira}. We remark that the hypotheses $h^{1,0}(S)=h^{2,0}(S)=0$ are currently necessary: $h^{2,0}(S)=0$ is needed to have a vertex algebraic interpretation of the Virasoro constraints (cf. Theorem \ref{thm: B}), and hence wall-crossing compatibility, and $h^{1,0}(S)=0$ is necessary in the proof of the constraints for the Hilbert scheme of points in \cite{moreira}; we hope these restrictions can be lifted in the future.

\subsection{History}
\label{sec: history of Virasoro constraints }\hfill

\noindent\textbf{Virasoro constraints} \
The study of Virasoro constraints on curve counts traces back to the origins of Gromov-Witten (GW) theory and intersection theory on the moduli space of stable curves, in Witten's foundational paper \cite{witten}. Witten conjectured that integrals of products of descendents -- certain natural classes in $H^\bullet(\overline \M_{g,n})$ -- obeyed some explicit relations. The relations he proposed were equivalent to certain differential operators,  which satisfy the Virasoro bracket relation, annihilating the partition function which encodes all the integrals of descendents on $\overline \M_{g,n}$. Witten's conjecture was proven by Kontsevich \cite{kontsevich} and new proofs were obtained by Okounkov-Pandharipande \cite{OPwitten} and Mirzakhani \cite{mirzakhani}.

In \cite{ehx}, the authors extended the Virasoro conjecture to the GW invariants of a variety $X$. Since then, a lot of effort has been put into proving the result for some target varieties $X$; most notably, the conjecture is now known when $X$ is a toric variety (or, more generally, when $X$ has semisimple quantum cohomology) by work of Givental and Teleman \cite{givental, teleman} and when $X$ is a curve by work of Okounkov-Pandharipande \cite{OP}. The general case, however, is still out of reach. 

In \cite{mnop1, mnop2}, Maulik-Nekrasov-Okounkov-Pandharipande propose a \allowbreak deep connection between Gromov-Witten invariants and Donaldson-Thomas (DT) invariants of 3-folds. Such correspondence suggested that the DT descendent invariants should as well be constrained by some sort of Virasoro operators. Almost 15 years ago, not long after the proposal of the MNOP conjecture, Oblomkov, Okounkov and Pandharipande were able to predict the precise form for the DT Virasoro operators (at least for $X=\PP^3$, see \cite[Conjecture 8]{survey}) from experimental data with $X$ toric. The understanding of the MNOP correspondence at the time, however, was not sufficiently explicit to be used effectively to relate the conjectures on the GW and on the DT sides. Recently, the GW/PT descendent correspondence has been made more effective and this allowed a proof of the DT Virasoro conjecture when $X$ is a toric 3-fold in the stationary regime \cite{moop}.\footnote{The results in \cite{moop} are formulated entirely in the theory of stable pairs, also known as Pandharipande-Thomas (PT) invariants, and not DT invariants. In the stationary regime for toric 3-folds, however, the two formulations of Virasoro are known to be equivalent. } 

Taking a surface $S$ and $X=S\times \PP^1$ it is possible to deduce some Virasoro constraints for the Hilbert scheme of points on $S$ from the PT Virasoro constraints on $X$. The third author used a universality argument in \cite{moreira} to prove such constraints for every surface with $H^1(S)=0$ by starting with the toric results in \cite{moop}. Subsequently, van Bree proposed a generalization of the Hilbert scheme constraints to the moduli spaces of torsion-free stable sheaves on a surface $S$ and made several non-trivial checks for toric $S$ using localization \cite{bree}.

While the Virasoro constraints on sheaf-counting theories come historically from Gromov-Witten theory, they form a rich theory themselves as indicated by the examples where they can be studied -- DT, PT, Hilbert scheme, stable torsion-free sheaves on surfaces. We show that they have an independent meaning and origin by connecting them to the geometric construction of the vertex algebras of Joyce \cite{Jo17} which were developed to study wall-crossing.

\noindent\textbf{Wall-crossing and vertex algebras} \
When Donaldson \cite{Do83} introduced his invariants counting anti-self-dual instantons, they were intrinsically dependent on the choice of a metric $g$ of the underlying four-manifold. Varying $g$ leads to discontinuous jumps of the invariants along codimension one walls, a phenomenon called ``wall-crossing". The precise description of the wall-crossing contributions has been given in \cite{KoMo94} and many further studies have been conducted.

With the goal of treating wall-crossing phenomena uniformly, Joyce \cite{JO06I, JO06II, JO06III, JO06IV} developed a theory which could be applied in large generality to abelian categories. Here the metric was replaced by stability conditions and instantons by semistable objects. Using a Lie algebra structure, he was able to define motivic invariants counting semistable objects and described how they change when varying the stability conditions. Further refinements to include DT theory of Calabi-Yau 3-folds were considered by Joyce-Song \cite{JS12} and Kontsevich-Soibelman \cite{KS08}.

The (virtual) fundamental classes of sheaves are however not motivic outside of the realm of Calabi-Yau 3-folds, so their theories were not sufficient for studying other geometries. In a more recent development, Joyce \cite{Jo17} introduced a sheaf-theoretic construction of vertex algebras (see \cite{Borcherds, Ka98, LLVA} for a gentle introduction to this topic).\footnote{See also H. Liu's work \cite{Liu} where he extends it to $K$-theory and multiplicative vertex algebras.} Vertex algebras are representation theoretic objects introduced by Borcherds \cite{Borcherds} and they give an axiomatization of conformal field theories in two dimensions \cite{BPZ}. The Lie bracket operation induced from the sheaf-theoretic vertex algebras was used to describe wall-crossing of virtual fundamental classes counting semistable objects, as conjectured in \cite{GJT} and proven in many cases by Joyce \cite{Jo21}. For surfaces, these wall-crossing formulae are related to the work of Mochizuki \cite{mochizuki} where the formulae are presented without vertex algebras. 

Wall-crossing has been used in \cite{Bo21.5, Bo21} to give explicit formulae for all descendent invariants of punctual Quot schemes on surfaces and Calabi-Yau fourfolds, and in \cite{Bu22} to study moduli spaces of vector bundles on curves. However, further structures coming from the vertex algebra remained a mystery. We fill this gap by giving a geometric interpretation of a natural conformal element in terms of Virasoro constraints. The conformal element  induces a representation of the Virasoro algebra on Joyce's vertex algebra \cite{Jo17}. The Virasoro operators $\{L_n\}_{n\in \BZ}$ act on the homology of the stack where wall-crossing takes place, defining a smaller Lie algebra of \textit{primary/physical states} \cite{BPZ,Borcherds}. We show that the Virasoro constraints are precisely the statement that (virtual) fundamental classes of moduli of semistable sheaves are primary states, and thus are preserved by wall-crossing. We use this new technique, together with a rank reduction to the case of rank 1, to prove existing and new conjectures about Virasoro constraints.

\subsection{Moduli of sheaves and pairs}
\label{sec: modulisheavespairs}

Let $X$ be a smooth projective variety over the complex numbers. Typically, we will restrict ourselves to small dimension $X$ (up to dimension 3) so that the moduli spaces of sheaves that we consider have a virtual fundamental class in the sense of Behrend-Fantechi \cite{BF}.\footnote{Virasoro constraints are also expected for moduli spaces of sheaves on Calabi-Yau 4-folds, admitting a virtual fundamental class in the sense of Oh-Thomas \cite{OT}. However, we will not consider this case here and it will be the subject of future work.} 

The main objects in this paper are moduli spaces $M$ which parameterize semistable sheaves on $X$ and their cohomology ring. Throughout this Introduction and Section \ref{sec: virasoro constraints}, when we refer to a general moduli space $M$ we assume that:
\begin{enumerate}
\item $M$ is a projective scheme of finite type. 
\item There are no $\BC$-points of  $M$ corresponding to strictly semistable sheaves.
\item Deformation theory at $[G]\in M$ is given by
\begin{align*}
\textup{Tan}&=\Ext^1(G,G)\\
\textup{Obs}&=\Ext^2(G,G)\\
0&=\Ext^{>2}(G,G)\,.
\end{align*}
\end{enumerate}

In such conditions, $M$ admits a virtual fundamental class $[M]^\vir\in H_\bullet(M)$ by \cite{BF}. There are many examples of moduli of sheaves on curves, surfaces and Fano 3-folds where all the assumptions are satisfied. To facilitate the exposition, for the introduction we will also assume that there exists a universal sheaf $\BG$ in $M\times X$; we will explain why this is not necessary in Section \ref{sec: invariantdescendents}. Note that a universal sheaf, when it exists, is non-unique: given a universal sheaf $\BG$ and a line bundle $L$ on $M$, the sheaf $\BG\otimes p^\ast L$ is also a universal sheaf (where $p\colon M\times X\to M$ is the projection) parametrizing the same objects.

Later in the paper we will also formulate Virasoro constraints for moduli spaces $M$ that have strictly semistable sheaves, by using Joyce's invariant class 
\[[M]^\inva \in H_\bullet(\CM_X^\rig)\,,\]
which is defined by a wall-crossing formula as we will overview in Section \ref{sec: joyceclasses}. This generalization will play an important role in the inductive argument used to prove Theorem \ref{thm: main}.
\begin{remark}
\label{rem:nocoarsedirect}
    The relation between the coarse moduli spaces $M$ and their invariant classes $[M]^{\inva}$ is not obvious in the presence of strictly semistables because there need not be a map $M\to \CM_X^{\rig}$. So saying that $M$ satisfies Virasoro constraints is an abuse of terminology that we will repeatedly make throughout this work. 
\end{remark}

Apart from the moduli of sheaves, we also study those of pairs. We fix a sheaf $V$ on $X$ and we let $P$ be a moduli space parametrizing a sheaf $F$ together with a map $V\to F$ (with some stability condition). We make the following assumptions: 
\begin{enumerate}
\item $P$ is a projective scheme of finite type. 
\item There are no $\BC$-points of $P$ corresponding to strictly semistable pairs.
\item There is a unique universal pair $q^*V\rightarrow\BF$ in $P\times X$ where $q\colon P\times X\to X$ is the projection. 
\item Deformation theory at $[V\to F]\in P$ is given by
\begin{align*}
\textup{Tan}&=\Ext^0([V\to F],F)\\
\textup{Obs}&=\Ext^1([V\to F],F)\\
0&=\Ext^{>1}([V\to F],F)\,.
\end{align*}
\end{enumerate}
In such conditions, $P$ admits a virtual fundamental class $[P]^\vir\in H_\bullet(P)$. Various moduli of pairs on curves and surfaces satisfy these assumptions; Quot schemes with at most one-dimensional quotients and moduli of Bradlow pairs.

There are two important differences between moduli spaces of sheaves and pairs.
\begin{enumerate}
\item The first is the difference in obstruction theory. It is apparent from comparing Example \ref{Ex: Tvir} and the definition of the Virasoro operators in Section \ref{sec: Virarosorops} that obstruction theory dictates their form.
\item  The second is the uniqueness or non-uniqueness of the universal object. This difference will play a crucial role in our treatment of the Virasoro constraints for moduli of sheaves and for moduli of pairs.
\end{enumerate}

\begin{remark}\label{rem: tracelessobs}
When we refer to moduli of sheaves we are mostly thinking about moduli of sheaves without fixed determinant. This is implicit in the obstruction theory above since when the determinant is fixed the deformation theory should instead use traceless Ext groups:
\[\Ext^i(G,G)_0=\ker\left(\Ext^i(G,G)\to H^i(\CO_X)\right).\]
We explain how to obtain a fixed determinant version of the Virasoro constraints in Section \ref{sec: fixeddet} when $h^{1,0}\neq 0$ but $h^{p,0}=0$ for $p>1$. Although a conjecture for Hilbert schemes of points on surfaces with possibly $p_g=h^{2,0}>0$  (which have traceless deformation theory) appears in \cite{moreira}, our approach in this paper is currently not suitable to understand it. We hope to pursue this direction in the future. 
\end{remark}

\begin{remark}\label{rem: stablepairsPT}
Virasoro constraints that we study for moduli of sheaves naturally generalize to moduli of objects in a derived category $D^b(X)$. Indeed, moduli spaces of stable pairs on a 3-fold $X$ (with $H^i(\CO_X)=0$ for $i>0$) in the sense of Pandharipande-Thomas \cite{PT} are instances of such. We emphasize here that stable pairs on 3-folds are subject to Virasoro constraints of sheaf type rather than pair type, despite their name. This is because virtual classes are constructed using the obstruction theory governed by $\Ext^i(I^\bullet,I^\bullet)$ where $I^\bullet=[\CO_X\to F]\in D^b(X)$. 

\end{remark}

\subsection{Universal sheaves and descendents}
\label{sec: descendents}
Descendents on $M$ are defined using a slant product construction with a universal sheaf $\BG$ and the maps

\begin{center}\begin{tikzcd}
& M\times X\arrow[ld, swap, "p"]\arrow[rd, "q"] &\\
M & & X\,.
\end{tikzcd}
\end{center}

\begin{definition}\label{def: geometricrealization}
We let $\BD^X$ be the supercommutative algebra generated by symbols $\chh_i(\gamma)$ for $i\geq 0$, $\gamma\in H^\bullet(X)$ (see Definition \ref{def: descendentalgebra}). The geometric realization with respect to a universal sheaf $\BG$ in $M\times X$ is the algebra homomorphism 
\[\xi_\BG\colon \BD^X\to H^\bullet(M)\]
defined on generators $\chh_i(\gamma)$ with $\gamma\in H^{r,s}(X)$ by 
\[\xi_\BG\left(\ch_i^\HH(\gamma)\right)=p_\ast\left(\ch_{i+\dim(X)-r}(\BG)\,q^\ast \gamma\right).\]
\end{definition}
The shift in the index of the Chern character using the Hodge degree of $\gamma$ is non-standard, but useful for a cleaner formulation of the Virasoro operators. With this convention, we may think of $\xi_\BG(\ch_i^\HH(\gamma))$ as being in $H^{i, i-r+s}(M)$ (of course $M$ might be singular, so a Hodge decomposition may not exist). See also Remark \ref{rem: comparingnotation}. 

The main objects of study in this paper are descendent integrals, i.e., the enumerative invariants obtained by integrating descendents against the virtual fundamental classes
\[\int_{[M]^\vir}\xi_{\BG}(D)\,,\quad \int_{[P]^\vir}\xi_{\BF}(D)\quad \textnormal{for}\quad D\in \BD^X.\]
Note that the descendent invariants of $M$ depend in principle on the choice of universal sheaf $\BG$. For some $D$ in the descendent algebra, however, they do not depend on this choice; these $D$ form what we call the  \textit{weight 0 descendent} algebra $\BD^X_{\inv}$ (cf. Section~\ref{sec: invariantdescendents}). For $D\in \BD^X_{\inv}$ we will omit the geometric realization morphism and write
\[\int_{[M]^\vir}D=\int_{[M]^\vir}\xi_\BG(D)\]
for any universal sheaf $\BG$ since it does not depend on such choice.

The Virasoro constraints say that these numbers satisfy some explicit universal relations. These relations are stated using certain operators
\[\bL_{\inv}\colon \BD^X\to \BD^X_{\inv}\,,\quad \bL_k^V:\BD^X\to \BD^X\,,\quad k\geq -1\]
that we will introduce in Section \ref{sec: virasoro constraints}. 
\begin{conjecture}[Virasoro for sheaves]\label{conj: virasorosheaves}
Let $M$ be a moduli of sheaves as before. Then
\[\int_{[M]^\vir} \Linv(D)=0\quad \textup{ for any }D\in \BD^X.\]
\end{conjecture}

\begin{conjecture}[Virasoro for pairs]
\label{conj: virasoropairs}
Let $P$ be a moduli of pairs as before. Then
\[\int_{[P]^\vir} \xi_\BF\left(\bL_k^V(D)\right)=0\quad \textup{ for any }k\geq 0, \, D\in \BD^X.\]
\end{conjecture}

In Example \ref{ex: rank2} we illustrate very explicitly how the constraints look like in the case of rank 2 stable bundles over a curve. 

\begin{remark}
The previous Virasoro conjectures for sheaves in \cite{moop, moreira, bree} require a specific choice of a universal sheaf and $\bS_k$ operators.\footnote{The $\bS_k$ operators in \cite{moop, moreira, bree} do not satisfy Virasoro bracket relations, obscuring the meaning of Virasoro constraints.} Conjecture \ref{conj: virasorosheaves} improves the formulation by avoiding both of these, and  we prove that the two formulations are equivalent (see Proposition \ref{prop: normalizedvsinv}). Conjecture \ref{conj: virasoropairs} for pairs is new and we provide convincing evidences by proving it for various geometries in this paper. 
\end{remark}

\subsection{Joyce's vertex algebra}
\label{sec: introvertexalgebra}

D. Joyce recently introduced a vertex algebra and a closely related Lie algebra associated to the derived category $D^b(X)$ \cite{Jo17,GJT, Jo21}. Joyce proposes to use his Lie algebra to study wall-crossing formulae for moduli of sheaves (or, more generally, moduli of semistable objects in a $\BC$-linear abelian or triangulated category).

The vertex algebra is constructed using the homology of the (higher) moduli stack $\CM_X$ parametrizing objects in the triangulated category $D^b(X)$. He defines a vertex algebra structure on
\[V_\bullet=\widehat{H}_\bullet(\CM_X)\,,\]
where $\widehat{H}_\bullet$ is meant to denote an appropriate shift in the grading of the homology.
The two most important ingredients for a vertex algebra are a translation operator $T$ and a state-field correspondence $Y(-,z)$; we will recall the definition of vertex algebras in Section \ref{sec: voadefinition}. In our setting, the translation operator is obtained from the $B\BG_m$--action on $\CM_X$; the state-field correspondence $Y(-,z)$ is defined in terms of the map 
$\Sigma:\M_X\times \M_X\to \M_X$ induced by taking direct sums and a perfect complex $\Theta$ on $\M_X\times \M_X$. These arise as a consequence of the master space localization technique that is commonly used for the proof of wall-crossing formulae (for instance in Mochizuki's work \cite{mochizuki}); the complex $\Theta$ is closely related to the virtual normal bundle appearing in the localization formula and thus to the obstruction theory of $\CM_X$. Remark \ref{rem: obstructionvsvirasoro} uses this observation to explain the relation of Virasoro constraints to the obstruction theory which we eluded to earlier on. 

Associated to the vertex algebra $V_\bullet$ is the Lie algebra obtained as the quotient by the translation operator:
\[\widecheck V_{\bullet}= V_{\bullet+2}/T V_{\bullet}\,.\]
The Lie bracket on $\widecheck V_\bullet$ is a shadow of the vertex algebra structure on $V_\bullet$ and is obtained by a well-known construction due to Borcherds \cite{Borcherds}.
Alternatively, it can be constructed as the homology $H_\bullet(\CM^\rig)$ of the rigidification $\CM_X^\rig=\CM_X \mkern-7mu\fatslash B\BG_m$; the two definitions agree when restricted to complexes with non-trivial numerical class, see Lemma \ref{Lem: DXinv}. 

The Lie algebra $\widecheck V_\bullet$ is a natural place where we can compare virtual fundamental classes of moduli spaces of sheaves; given a moduli space $M$ of semistable sheaves (or more generally of objects in $D^b(X)$) containing no strictly semistable sheaves, there is an open embedding $M\hookrightarrow \CM_X^\rig$. If $M$ admits a virtual fundamental class, we may push it forward along this embedding to obtain a class
\[[M]^\vir\in \widecheck{H}_\bullet(\CM_X^\rig)=\widecheck V_\bullet\]
where the first appearance of $\widecheck{(-)}$ represents an appropriate degree shift. If we fix a choice of a universal sheaf $\BG$ in $M\times X$, by the universal property of $\CM_X$ we get a map $f_\BG\colon M\to \CM_X$ lifting $M\hookrightarrow \CM_X^\rig$, and thus a natural lift of the virtual fundamental class to the vertex algebra
\[[M]^{\vir}_{\BG}\coloneqq (f_\BG)_\ast [M]^\vir\in V_\bullet\,.\]
Crucially, Joyce defines more general classes
\[[M]^\inva\in \widecheck V_\bullet\] 
even when strictly semistable sheaves exist; when $[M]^\vir$ is defined, both classes agree. 

The classes $[M]^\vir\in \widecheck V_\bullet$ or $[M]^{\vir}_{\BG}\in V_\bullet$ contain essentially the information of the (invariant) descendent integrals on $M$. This is made precise by J. Gross' \cite{Gr19} explicit description of $V_\bullet$, which we recall in Section \ref{sec: grossiso}. The cohomologies $H^\bullet(\CM_X)$ and $H^\bullet(\CM_X^\rig)$ are closely related to the algebras of descendents $\BD^X$ and $\BD^X_{\inv}$, respectively; see Lemmas \ref{Lem: isomorphism} and \ref{Lem: DXinv} for the precise statements. The pairing between cohomology and homology then recovers the descendent integrals
$$H^\bullet(\CM_X)\otimes H_\bullet(\CM_X)\to \BC
\,,\quad (D,[M]^{\vir}_{\BG})\mapsto \int_{[M]^\vir}\xi_{\BG}(D)\,,
$$
and 
$$H^\bullet(\CM_X^\rig)\otimes H_\bullet(\CM_X^\rig)\to \BC\,,\quad
(D,[M]^\vir)\mapsto \int_{[M]^\vir}D \,,
$$
where the second integral is independent of the choice of $\BG$.

\subsection{Conformal element and Virasoro constraints}
\label{sec: introconformal}

A vertex operator algebra is a vertex algebra $V_\bullet$ equipped with a conformal element $\omega\in V_4$. The main property of a conformal element (see Section \ref{sec: voadefinition} for a precise definition) is that the operators $\{L_n\}_{n\in \BZ}$ on $V_\bullet$ induced from $\omega$ via the state-field correspondence satisfy the Virasoro bracket
\[\big[L_n,L_m\big] = (n-m)L_{m+n} + \frac{n^3-n}{12}\delta_{n+m,0}\cdot C\]
for some constant $C\in \BC$ called the central charge of $(V_\bullet, \omega)$. One of the goals of this paper is to explain the Virasoro operators in the descendent algebra previously studied in terms of a conformal element $\omega$ in Joyce's vertex algebra (or some slight variation, namely the pair vertex algebra). Due to the mysterious role that the Hodge degrees play in the Virasoro operators in \cite{moreira}, we do not know how to do so in complete generality, but only under the following assumption:

\begin{assumption}
\label{Ass: vanishing}
We assume that the Hodge cohomology groups $H^{p,q}(X)$ vanish whenever $|p-q|>1$. 
\end{assumption}
This assumption is satisfied for curves, surfaces with $p_g=0$ and Fano 3-folds, hence covering the majority of the target varieties in Donaldson-Thomas theory.

The result of J. Gross \cite{Gr19}\footnote{Some corrections to the statements in \cite{Gr19} were necessary for purposes of accuracy. We rederive all of the necessary results about the vertex algebra structure in Theorem \ref{thm: gross} to make sure that they are true.} shows that, under certain assumption (satisfied for curves, surfaces and rational 3-folds), Joyce's vertex algebra $V_\bullet$ is naturally isomorphic to a lattice vertex algebra from $(K^\bullet(X), K_\sst^0(X), \chi_\sym)$; here
\[K^\bullet(X)=K^0(X)\oplus K^1(X)\cong H^\bullet(X)\] is the topological $K$-theory of $X$ with $\BC$-coefficient, $K_\sst^0(X)$ is the semi-topological $K$-theory\footnote{The zeroth semi-topological $K$-theory $K^0_\sst(X)$ is defined as a Grothendieck's group of vector bundles modulo algebraic equivalence. By \cite[Theorem 4.21]{B16} and \cite[Theorem 2.3]{AH18}, we have $K^0_\sst(X)\simeq \pi_0(\mathcal{M}_X)$.} with $\BZ$-coefficient and $\chi_\sym$ is the symmetric pairing on $K^\bullet(X)$
\[\chi_\sym(v,w)=\int_{X}\ch(v^\vee)\ch(w)\td(X)+\int_{X}\ch(w^\vee)\ch(v)\td(X)\,.\]
The construction of a vertex algebra from such data is recalled and summarized in Theorem \ref{thm: voaconstruction} and follows Kac \cite{Ka98}; it uses Kac's bosonic vertex algebra construction in the even part $K^0(X)$ and the anti-fermionic vertex algebra construction in the odd part $K^1(X)$. 

Kac's construction produces a conformal element when the pairing $\chi_\sym$ is non-degenerate; unfortunately, due to the symmetrization this is not often the case. It turns out that this issue can be overcome by using the larger vertex algebra~$V_\bullet^{\pa}$. The vector space underlying $V_\bullet^{\pa}$ is the homology of the stack of pairs $\CP_X\simeq \CM_X\times \CM_X$:
\[V_\bullet^{\pa}=\widehat{H}_\bullet(\CP_X)\,.\]
The construction of the conformal element requires a choice of an isotropic decomposition of the fermionic part (see Theorem \ref{thm: voaconstruction} or \cite[Section 3.6]{Ka98}). This decomposition is where the Hodge degrees and Assumption \ref{Ass: vanishing} come into play, because $K^1(X)$ splits into the isotropic subspaces
$$K^1(X)=K^{\bullet,\bullet+1}\oplus K^{\bullet+1,\bullet}
$$
which via the Chern character isomorphism correspond to 
$$K^{\bullet,\bullet+1}\cong \bigoplus_{p\geq 0}H^{p, p+1}(X)\quad \textup{and}\quad K^{\bullet+1,\bullet}\cong \bigoplus_{p\geq 0}H^{p+1, p}(X)\,.
$$
It is for the construction of such conformal element that we use the vertex algebra over the complex numbers while the result of J. Gross \cite{Gr19} works over any field containing rational numbers, as it relies on the Hodge decomposition. We prove that the Virasoro operators induced by this conformal element $\omega$ are dual to the pair Virasoro operators defined in the algebra of descendents, see Section \ref{sec: virasorocomparison}.

One remarkable aspect of Theorem B is that, while the operators $\bL_n^{\pa}$ on the descendent algebra were previously only defined for $n\geq -1$, a conformal element provides fields $L_n^{\pa}$ for every $n\in \BZ$ and thus a complete representation of the Virasoro algebra. In particular, this representation now has a non-trivial central charge $2\chi(X)$ which the positive branch $\{L_n^{\pa}\}_{n\geq -1}$ does not detect. We note that the factor of 2 appears due to working with the pair vertex algebra $V_\bullet^{\pa}$; if the pairing $\chi_\sym$ were non-degenerate we would get a conformal element in $V_\bullet$ with central charge $\chi(X)$. Remarkably, the Virasoro operators on the Gromov-Witten theory of $X$ (at least if $\dim(X)$ is even) are also known to admit an extension to a full representation of the Virasoro algebra; the central charge in the Gromov-Witten case is $\chi(X)$ \cite[Section 2.10]{getzler}. 

This description of the Virasoro operators provides a beautiful formulation of the Virasoro constraints for sheaves and for pairs in terms of well-known notions in the theory of vertex operator algebras, namely subspace of primary states (also known as physical states):
$$\widecheck{P}_0\subset \widecheck{V}_\bullet,\quad P_0^{\pa} \subset V_\bullet^{\pa}.$$

\begin{theorem}[Section \ref{subsec: virasoroprimary}]\label{thm: virasorophysical}
Assume $X$ is in class D (see Remark \ref{rem: classD}) and satisfies Assumption \ref{Ass: vanishing}. Then 
\begin{enumerate}
\item Conjecture \ref{conj: virasorosheaves} is equivalent to the class 
\[
[M]^\vir\in \widecheck{V}_\bullet \subseteq \widecheck{V}^{\pa}_\bullet
\]
being a primary state in $\widecheck{P}_0$ (cf. Definition \ref{Def: physicalstates}). 
\item Conjecture \ref{conj: virasoropairs} is equivalent to the class 
\[
[P]_{(q^\ast V,\BF)}^\vir\in V^{\pa}_\bullet
\]
being a primary state in $P_0^{\pa}$ 
(cf. Definition \ref{Def: physicalstates}). Here, the class $[P]_{{(q^\ast V,\BF)}}^\vir$ denotes the lift of $[P]^\vir\in \widecheck{V}^{\pa}_\bullet$ to $V^{\pa}_\bullet$ induced by the universal pair $q^\ast V\to\BF$.
\end{enumerate}
\end{theorem}
While the part (1) in the above theorem was stated for moduli spaces of sheaves satisfying the assumptions (1), (2) and (3) from Section \ref{sec: modulisheavespairs}, the condition of being a primary state makes sense also for the invariant classes $[M]^{\inva}$ without the assumption (2). This motivates the following generalization of Conjecture \ref{conj: virasorosheaves}. 
\begin{conjecture}
    \label{conj:actualinvvir}Under the assumptions of Theorem \ref{thm: virasorophysical}, the moduli spaces $M$, possibly containing strictly semistable sheaves, satisfy Virasoro constraints in the sense that
  $$
  [M]^{\inva}\in \widecheck{P}_0\subset  \widecheck{V}^{\pa}_{0}\,.
  $$
\end{conjecture}

The proof of Theorem \ref{thm: virasorophysical} is given in Section \ref{subsec: virasoroprimary}. We prove in Proposition \ref{prop: wallcrossingcompatibility1}, \ref{prop: wallcrossingcompatibility2} that the space of physical states interact nicely with Lie bracket operations: $\widecheck{P}_0\subset \widecheck{V}_\bullet
$ is a Lie subalgebra and $P_0^{\pa} \subset V_\bullet^{\pa}$ is a Lie submodule over $\widecheck{P}_0$. Since wall-crossing formulae in \cite{Jo21} are always written using the Lie bracket, these Lie algebraic statements prove a compatibility between wall-crossing and the Virasoro constraints.

\subsection{Proof of Theorem \ref{thm: main} and other results}
\label{sec: results}

The main result of the paper is Theorem \ref{thm: main}, i.e., a proof of Conjecture \ref{conj: virasorosheaves} for semistable sheaves on curves and surfaces with $h^{0,1}=h^{0,2}=0$. We also denote the three cases (1), (2)  and (3) in Theorem \ref{thm: main} by $(m,d) = (1,1), (2,2)$ and $(2,1)$ where $m$ denotes the dimension of the underlying variety and $d$ the dimension of the support of sheaves we consider. 

The main ingredient in the proof is an inductive rank reduction argument via wall-crossing. This is the content of Section \ref{sec: rankreduction}. In each of the 3 cases, we consider the moduli spaces of Bradlow pairs $P^t_\alpha$ which depend on a stability parameter $t>0$. Assuming that $M_\alpha$ contains no strictly semistable sheaves, when $0<t\ll 1$ is small there is a map $P^{0+}_\alpha\coloneqq  P^{t\ll 1}_\alpha\to M_\alpha$ which is a (virtual) projective bundle. This geometric description is equivalent to wall-crossing at the Joyce-Song wall:
$$[P^{0+}_{\alpha}]^\vir = \big[[M_{\alpha}]^\vir,e^{(1,0)}\big]\,.$$
On the other hand, for large $t\gg 1$ the moduli spaces $P^{\infty}_\alpha\coloneqq  P^{t\gg 1}_\alpha$ are easier to understand (and sometimes empty). We prove not only that $M_\alpha$ satisfy the sheaf Virasoro constraints (i.e., Theorem  \ref{thm: main}) but also that the moduli spaces of Bradlow pairs $P_\alpha^t$ satisfy the pair analogue of the constraints:
\begin{theorem}[Theorem \ref{thm: rankreduction}, Section \ref{sec: limits}, Section \ref{sec:lowrank}]\label{thm: virasorobradlowpairs}
The moduli spaces of Bradlow pairs $P_\alpha^t$ (see Definition \ref{def: bradlow}) satisfy the pair Virasoro constraints (Conjecture \ref{conj: pairvirasoro}) for every $t>0$ in the 3 settings of Theorem \ref{thm: main}, i.e., $(m,d)=(2,2), (1,1)$ and $(m,d)=(2,1)$ provided Assumption \ref{ass:WCpair} holds.
\end{theorem}

To prove Theorems \ref{thm: main} and \ref{thm: virasorobradlowpairs} we need the following steps:
\begin{enumerate}
\item[(i)] We prove in Section \ref{sec:lowrank} that $P^{\infty}_\alpha$ satisfies Conjecture \ref{conj: virasoropairs}. In case $(m,d)=(1,1)$, we only need to prove the statement for symmetric powers of curves which we do by a direct computation in Proposition \ref{prop: symmetricpowers}. Cases $(m,d)=(2,1)$ and $(2,2)$ for slope stability can be reduced to Hilbert schemes of points, where they were shown to hold in \cite{moop, moreira}.
\item[(ii)] We use the wall-crossing formula \eqref{eq: rankreductionwallcrossing} between $P^{\infty}_\alpha$ and $P^{0+}_\alpha$ to show that $P^{0+}_\alpha$ satisfies Virasoro constraints for pairs as well. By induction on $\rk(\alpha)$, we know the Virasoro constraints on the wall-crossing terms. We then rely on the compatibility between wall-crossing and Virasoro constraints (Propositions \ref{prop: wallcrossingcompatibility1} and \ref{prop: wallcrossingcompatibility2}).
\item[(iii)]  Finally, we use a projective bundle compatibility for  $P^{0+}_\alpha\to M_\alpha$ proved in Theorem \ref{lem: projectivebundle} to show that the pair Virasoro constraints on $P^{0+}_\alpha$ imply the sheaf Virasoro constraints on $M_\alpha$.  Wall-crossing from slope stability to Gieseker stability in Corollary \ref{cor:slopetogies} concludes the proof of Theorem \ref{thm: main}.
\end{enumerate}
A crucial point in the argument is that we must include moduli spaces $M_\alpha$ admitting strictly semistable sheaves in the induction since they unavoidably appear as wall-crossing terms. That is, we must prove that $[M]^\inva$ is in the Lie algebra of primary states. Because of that, in step (iii) we do not exactly have a projective bundle. However, by the very definition of the invariant classes $[M]^\inva$, what we have to prove is essentially the same as in the projective bundle case. We do this in Theorem~\ref{lem: projectivebundle}.

In the appendix we explain Joyce-Song wall-crossing, using results of the first author in \cite{Bo21.5}, which provides an alternative for some of the arguments in Section \ref{sec:lowrank}. In particular, we prove that the pair Virasoro constraints hold for punctual Quot schemes.\footnote{It would be interesting to see the implication of Virasoro constraints to the generating series of descendent invariants of Quot schemes studied in \cite{JOP}.}

\begin{theorem}[=Theorem \ref{prop: quot}]\label{thm: punctualquot}
Let $X$ be a curve or a surface and let $V$ be a torsion-free sheaf on $X$. Then the punctual Quot scheme $\textup{Quot}_X(V, n)$ satisfies the pair Virasoro constraints (Conjecture \ref{conj: pairvirasoro}).
\end{theorem}

\subsection{Notation and conventions}
Except the semi-topological $K$-group $K_\sst^0(X)$ which we consider over $\BZ$-coefficients, all cohomology and $K$-theory groups are assumed to have coefficients in $\BC$ unless stated otherwise. We write 
\[K^\bullet(X)=K^0(X)\oplus K^1(X)\]
for the topological $K$-theory and we denote by
\[\ch\colon K^\bullet(X)\to H^\bullet(X)\]
the isomorphism between topological $K$-theory and cohomology. The total cohomology of a topological space is always understood to be the direct product of the cohomology groups in each degree, while homology is the direct sum. We will use the cap product with the cohomology acting on the left. This is the convention followed in \cite{dold}; it differs from the more usual convention with cohomology on the right by a sign, i.e., $\gamma \cap u=(-1)^{|\gamma||u|}u\cap \gamma$ for $\gamma\in H^\bullet(X)$, $u\in H_\bullet(X)$. 

By an obstruction theory on a scheme $M$ we mean a perfect complex $\BE$ of non-negative degree together with a map $\phi:\BE^\vee\to \BL_M$ such that $h^0(\phi)$ and $h^{-1}(\phi)$ are an isomorphism and a surjection, respectively. The complex $\BE$, or its $K$-theory class, is also called the virtual tangent bundle of $M$ (note that by \cite{siebert} the virtual fundamental class only depends on the $K$-theory class $\BE$, and in particular does not depend on the map to $\BL_M$). We will often abuse notation and just call $\BE$ (or its $K$-theory class) the obstruction theory, leaving the map to $\BL_M$ implicit. The construction of the morphism to $\BL_M$ in all the cases considered is standard, see \cite{huythomas} for sheaves/complexes and e.g. \cite[Section 5.3]{mochizuki} for pairs.

\begin{center}
    \begin{tabular}{p{4cm} p{10cm}}
    $\alpha,\beta$& Semi-topological $K$-theory classes in $K_\sst^0(X)$.\\
    $\gamma, \delta$&  Cohomology classes on $X$. \\
    $v,w$& Elements of $K^\bullet(X)$.\\
    $\deg(-)$ & Degree for any graded vector space. 
    \\
    $|-|$& Supergrading taking value in $\{0,1\}. $   \\
    $\chh_i(\gamma)$ & The holomorphic descendent in degree $2i-p+q$ depending on the Hodge degree of $\gamma\in H^{p,q}(X)$.\\
    $\ch_i(\gamma)$& The topological descendent in degree $2i-|\gamma|$. \\
     $L_n, T_n, R_n,$& Virasoro operators on homology and vertex algebra.\\
    $\mathsf{L}_n, \mathsf{T}_n,\mathsf{R}_n,$& Dual operator notation on cohomology and descendent algebra.\\
    $\mathsf{L}_{\inv}$& weight 0 Virasoro operator on descendent algebra.
        \end{tabular}
\end{center}

\subsection{Future directions} 
There are several open directions regarding the Virasoro constraints for sheaves. The first obvious direction is to try to improve Theorem \ref{thm: main} by removing the assumptions $h^{0,1}=h^{0,2}=0$. The arguments in this paper show that we can get the constraints for $h^{0,1}>0$ as long as we can prove them for the moduli of rank 1 sheaves (isomorphic to the Hilbert scheme of points times the Jacobian). Finding an argument that works in general for the Hilbert scheme of points and does not go through Gromov-Witten theory would be highly desirable. Removing the assumption $h^{0,2}=0$ requires a better understanding of the constraints in the setting of reduced virtual fundamental classes for fixed determinant theory (see Remark \ref{rem: fixed det}).

Sheaf-theoretic Virasoro constraints of Fano 3-folds are of particular interest because they are related to the original Virasoro constraints in Gromov-Witten theory. Since wall-crossing compatibility of Virasoro constraints also holds in this case, it would be interesting to develop wall-crossing techniques for Fano 3-folds that can be applied to Virasoro constraints.

\subsection{Acknowledgement} 
The authors would like to thank Y. Bae, D. van Bree, J. Gross, A. Henriques, D. Joyce, M. Kool, H. Liu, A. Mellit,  G. Oberdieck, A. Oblomkov, R. Pandharipande and W. Pi for many conversations about the Virasoro constraints on moduli spaces of sheaves.

The first and third authors were supported by
ERC-2017-AdG-786580-MACI. The second author is supported by the grant SNF-200020-182181. The project received funding from the European Research Council (ERC) under the European Union Horizon 2020 research and innovation programme (grant agreement 786580).

\section{Virasoro constraints}
\label{sec: virasoro constraints}

In this section, we formulate the Virasoro constraints for moduli spaces of sheaves and pairs, denoted by $M$ and $P$, respectively. These moduli spaces satisfy the numbered assumptions in Section \ref{sec: modulisheavespairs} unless otherwise mentioned. We use the notation $M_\alpha$ when the topological type $\alpha$ of the sheaves in the moduli space is relevant. 

\subsection{Supercommutative algebras}
\label{sec: supercommutative}
Before we move onto geometry, we note down some useful observations about freely generated \textit{supercommutative} algebras and derivations on them. Let $D_\bullet$ be a supercommutative $\BZ$-\textit{graded unital algebra} over~$\BC$ with degree $\deg(v) = i$ for any $v\in D_i$. Supercommutativity means that multiplication satisfies 
\[v\cdot w = (-1)^{|v||w|} w\cdot v\,,\]
where 
$$|v|\in \{0,1\}\,,\quad  \text{such that} \quad |v|\equiv \deg(v)\mod 2\,.$$
The unit of $D_\bullet$ is always going to be denoted by $1$ and in general, we will omit specifying it in the notation. 

A \textit{superderivation} of degree $r$ on $D_\bullet$ is a $\BZ$-graded linear map 
$$
R:D_\bullet\longrightarrow D_{\bullet+r}
$$
satisfying the \textit{graded Leibnitz rule}
$$R(v\cdot w) = R(v)\cdot w + (-1)^{r|v|}\,v\cdot R(w)\,.$$

\begin{definition}\label{def: supercommutativefree}
Let $C_\bullet$ be a $\BZ$-graded  $\BC$-vector space. We denote by $$D_\bullet = \SSym[C_\bullet]$$ the unital supercommutative algebra freely generated by $C_\bullet$. Denote by $C^\bullet$ the graded dual of $C_\bullet$. We define the dual of $D_\bullet$ as a completion of $\SSym[C^\bullet]$ with respect to the degree. More precisely, the dual is
\[D^\bullet\coloneqq \SSym\llbracket C^\bullet\rrbracket=\prod_{i\geq 0}\SSym[C^\bullet]^i\]
where $\SSym[C^\bullet]^i$ denotes the degree $i$ part of $\SSym[C^\bullet]$ with the degree induced by the one on $C^\bullet$.\footnote{We follow here the convention that the total homology is the direct sum of the homology groups in each degree, while cohomology is the product.} 
\end{definition}
Given a linear map $f\colon C_\bullet\to B_{\bullet+r}$ of degree $r$, there is a unique way to extend $f$ to an algebra homomorphism $\SSym\llbracket f\rrbracket$ and to a derivation $\text{Der}(f)$ of degree $r$.

The pairing between $C_\bullet$ and $C^\bullet$ can be promoted to a cap product between $D_\bullet$ and $D^\bullet$. 
\begin{definition}
\label{Def: pairingofV}
Fix $C_\bullet$ and $C^\bullet$ dual vector spaces and let $\langle-,-\rangle\colon C^\bullet\times C_\bullet\to \BC$ be the pairing. Let $D_\bullet$ and $D^\bullet$ be as in Definition \ref{def: supercommutativefree}.
We define a cap product
$$
\cap\colon C^\bullet\times D_\bullet\longrightarrow D_\bullet
$$
by letting $\nu\cap(-)$ for $\nu\in C^\bullet$ act as a superderivation of degree $-\deg(\nu)$ on $D_\bullet$  restricting to $\langle \nu,-\rangle: C_\bullet\to\BC$. The cap product extends uniquely to
$$
\cap\colon D^\bullet\times D_\bullet\longrightarrow D_\bullet\,,
$$
by requiring that $(\mu \nu) \cap u = \mu\cap(\nu\cap u)$. Notice that this makes $D_{\bullet}$ into a left $D^\bullet$-module.

Starting from a map $f: C_1^\bullet\to C_2^\bullet$, there is a dual map $f^{\dagger}:C_{2,\bullet}\to C_{1,\bullet}$. Constructing algebra homomorphisms and derivations commutes with taking duals: $$\SSym\llbracket f^\dagger\rrbracket = \SSym\llbracket f\rrbracket^ \dagger\,,\qquad \text{Der}(f^\dagger) = \text{Der}(f)^\dagger\,.$$

By composing with the projection $D_\bullet\to \BC$, we recover a non-degenerate pairings $$\langle -,-\rangle: D^\bullet\times D_\bullet\to \BC\,.$$
\end{definition}

An explicit description of the cap product can be obtained after fixing a basis $B\subset C_\bullet$ and observing that 
$$
\nu \cap (-) = \sum_{v\in B}\langle \nu,v\rangle\, \frac{\partial}{\partial v}\,
$$
where $\frac{\partial}{\partial v}$ is the superderivation of degree $-\deg(v)$ acting on the elements of the basis $B$ by $\frac{\partial}{\partial v}(v)=1$ and $\frac{\partial}{\partial v}(w)=0$ for $w\in B\setminus \{v\}$.

\subsection{Descendent algebra}
\label{sec: descendentalgebra}
Let $X$ be a smooth projective variety over $\BC$.

\begin{definition}\label{def: descendentalgebra}
Let $\text{CH}^X$ denote the infinite dimensional vector space over $\BC$  generated by symbols called \textit{holomorphic descendents} of the form
\[\chh_i(\gamma)\quad\textup{ for }\quad i\geq 0,\, \gamma\in H^\bullet(X)\]
subject to the linearity relations
\[\chh_i(\lambda_1\gamma_1+\lambda_2\gamma_2)=\lambda_1\chh_i(\gamma_1)+\lambda_2\chh_i(\gamma_2)\]
for $\lambda_1, \lambda_2\in \BC$. We define the \textit{cohomological} $\BZ$-grading on $\text{CH}^X$ by 
\begin{equation}
\label{eq: degchigamma}
\deg\chh_i(\gamma) = 2i-p+q \quad \text{ for }\quad \gamma\in H^{p,q}(X)\,.
\end{equation}
Finally, we let $\BD^X$ be the $\BZ$-graded \textit{algebra of holomorphic descendents}
\[\BD^X = \SSym\llbracket\text{CH}^X\rrbracket\,,\]
which is the completion of the supercommutative algebra generated by $\ch_i^\HH(\gamma)$. We will write $\chh_\bullet(\gamma)$ for the element
\[\chh_\bullet(\gamma)=\sum_{i\geq 0}\chh_i(\gamma)\in \BD^X\,.\]
\end{definition}

\begin{remark}\label{rem: comparingnotation}
This algebra of descendents is very similar to the one introduced in \cite{moop} with two small differences. Firstly, we now take a completion with respect to degree; this makes little difference in practice, but it is important in the comparison between $\BD^X$ and $H^\bullet(\CM_X)$ (cf. Lemma \ref{Lem: isomorphism}) since we follow the standard convention that the total cohomology is a product of the groups in each degree. The second difference is in the grading; in our notation, given $\gamma\in H^{p,q}(X)$, the symbol $\chh_i(\gamma)$ should be thought of as having Hodge degree $(i, i-p+q)$ (recall this from Definition \ref{def: geometricrealization}) so that the geometric realization is degree preserving. The superscript $H$ stands for holomorphic part of the Hodge degree and is used to indicate this degree convention. The original convention appearing in loc. cit. defined the descendents $\ch^{\text{old}}_i(\gamma)$ as \[\ch_i^{\textup{old}}(\gamma)=\chh_{i+p-\dim(X)}(\gamma)\]
for $\gamma\in H^{p,q}(X)$. The reason for introducing holomorphic descendents is to give a natural looking expression for the operator $\mathsf{T}_k$ in Section \ref{sec: Virarosorops}.
\end{remark}

It will sometimes be useful to also consider the shift
\[\ch_{i}(\gamma)\coloneqq \chh_{i+\lfloor\frac{p-q}{2}\rfloor}(\gamma)=\ch^{\textup{old}}_{i-\lceil\frac{\deg \gamma}{2}\rceil+\dim X}(\gamma)\]
so that $\deg\ch_i(\gamma)=2i-|\gamma|$; we recall that $|\gamma|\in \{0,1\}$ is the parity of $\gamma$ as in Section \ref{sec: supercommutative}. 
\begin{definition}\label{def: topdescendentalgebra}
Let $\alpha\in K_\sst^0(X)$ be a topological type. We define $\tCH^X_\alpha$ to be the graded vector space generated by symbols \[\ch_i(\gamma)\quad\textup{ for }\quad i\in \BZ_{>0},\,\gamma\in H^\bullet(X).\]
We let $\BD^X_\alpha$ be $\SSym\llbracket \tCH^X_\alpha\rrbracket$. The algebra $\BD^X_\alpha$ comes equipped with an algebra homomorphism
$p_\alpha\colon \BD^X\to \BD^X_\alpha$
sending
\[\ch_i^\HH(\gamma)\mapsto\begin{cases}
\ch_{i-\lfloor\frac{p-q}{2}\rfloor}(\gamma) &\textup{ if } \,\deg\ch_i^\HH(\gamma)>0\\
\int_X \gamma\cdot \ch(\alpha) &\textup{ if } \,\deg\ch_i^\HH(\gamma)=0\\
0&\textup{ otherwise.}
\end{cases}\]
\end{definition}

Note that abstractly the algebras $\BD^X_\alpha$ are independent of $\alpha$, but the morphisms $p_\alpha$ depend on $\alpha$ by their behavior on the descendents of degree 0. Let $M_\alpha$ be a moduli space parametrizing sheaves of topological type $\alpha$ with a universal sheaf $\BG$. Then the geometric realization factors through $p_\alpha$:
\begin{equation}
\label{Eq: descentofrealization}
\begin{tikzcd}
\BD^X\arrow[r,"p_{\alpha}"]\arrow[d,swap, "\xi_{\BG}"]& \BD^{X}_\alpha \arrow[dl,"\xi_{\BG}"]\\
H^\ast(M_\alpha)&
\end{tikzcd}
\end{equation} 
where we still denote by $\xi_\BG$ the factoring map, which is defined as
\begin{equation}\label{eq: geometricrealizationalpha}
\xi_{\BG}(\ch_i(\gamma))=p_\ast\left(\ch_{i-\lceil\frac{\deg \gamma}{2}\rceil+\dim X}(\BG)q^\ast \gamma\right)\in H^{2i-|\gamma|}(M_\alpha)\,.
\end{equation}
The map factors since 
$\xi_{\BG}(\chh_i(\gamma))=\int_X \gamma\cdot \ch(\alpha)\in H^0(M_\alpha)$ when $\deg(\chh_i(\gamma))=0$ and the degree of $\xi_{\BG}\big(\chh_i(\gamma)\big)$ is identical to \eqref{eq: degchigamma}. Given $D, D'\in \BD^X$ we say that ``$D=D'$ in $\BD^X_\alpha$'' if $p_\alpha(D)=p_\alpha(D')$.

\begin{example}
\label{Ex: Tvir}
   One may lift the Chern classes of the virtual tangent bundle of $M_\alpha$ to $\BD^X$ or $\BD^X_{\alpha}$. As a topological K-theory class, the virtual tangent bundle is defined as 
$$
T^\vir M=-Rp_\ast\RHom(\BG, \BG) + \CO_M\,.
$$
 Using $\sum_{t}\gamma_t^L\otimes \gamma_t^R$ to denote the Künneth decomposition of $\Delta_{\ast} \td(X)$, where $\Delta\colon X\to X\times X$ is the diagonal, and applying Grothendieck-Riemann-Roch, one computes that
\begin{align*}
\ch(TM^\vir)=-\xi_\BG\left(\sum_{i,j\geq 0}\sum_t(-1)^{i-p_{t}^L+\dim(X)}\chh_i(\gamma_t^L)\chh_j(\gamma_t^R)\right)+1\,,
\end{align*}
where $\gamma_t^L\in H^{p_t^L, q_t^L}(X)$. The reason for the existence of this lift will become apparent from Lemma \ref{Lem: isomorphism}. The similarity with Virasoro constraints below is a general phenomenon which can be used to guess their correct formulation.
\end{example}

\subsection{Virasoro operators}

\label{sec: Virarosorops}

In this section we define the Virasoro operators 
\[\bL_k\colon \BD^X\to \BD^X\quad\textnormal{for}\quad k\geq -1,\]which produce the Virasoro constraints. These operators have two terms, a derivation term $\bR_k$ and a linear term $\bT_k$ which is quadratic in $\chh_i$. The full Virasoro operators are $\bL_k=\bR_k+\bT_k$, where
\begin{enumerate}
\item 
$\bR_k \colon \BD^X\to \BD^X$ is an even (of degree $2k$) derivation extended from $\bR_k\colon \tCH^X\to\tCH^X$, where it is defined by
\[\bR_k \chh_i(\gamma)\coloneqq\left(\prod_{j=0}^k (i+j)\right)\chh_{i+k}(\gamma)\,.\]
We take the following conventions: the above product is $1$ if $k=-1$ and $\chh_{i+k}(\gamma)=0$ if $i+k<0$. 
\item $\bT_k\colon \BD^X\to \BD^X$ is the operator of multiplication by the element of $\BD^X$ given by
\[\bT_k\coloneqq\sum_{i+j=k}(-1)^{\dim X-p^L}i!j!\chh_i\chh_j(\td(X)).\]
In the formula above, $(-1)^{\dim X-p^L}\chh_i\chh_j(\td(X))$ is defined as follows: let $\Delta\colon X\to X\times X$ be the diagonal map and let 
\[\sum_{t}\gamma_t^L\otimes \gamma_t^R=\Delta_\ast \td(X)\] 
be a Kunneth decomposition of  $\Delta_\ast\td(X)$ such that $\gamma_t^L\in H^{p^L_t, q^L_t}(X)$ for some $p_t^L, q_t^L$. Then 
\[(-1)^{\dim X-p^L}\chh_i\chh_j(\td(X))\coloneqq \sum_{t}(-1)^{\dim X-p_t^L}\chh_i(\gamma_t^L)\chh_j(\gamma_t^R).\]
\end{enumerate}
\begin{remark}
The operator $\bL_{-1}=\bR_{-1}$ plays a special role and has a particularly nice geometric interpretation in terms of $\BG_m$--gerbes over $M_{\alpha}$ which we describe in Lemma \ref{lem: translationdualr-1}.
\end{remark}

The operators $\{\bL_k\}_{k\geq -1}$ satisfy the Virasoro bracket relations
\[[\bL_k, \bL_\ell]=(\ell-k)\bL_{k+\ell}\in \textup{End}(\BD^X)\]
where the bracket denotes the usual commutator of operators $[X,Y]=X\circ Y-Y\circ X$.
This was noted in \cite{moop}; for a detailed proof see \cite[Proposition 2.10]{bree}. The unusual constant factor $(\ell-k)$, instead of $(k-\ell)$, suggests that there might be another set of more natural Virasoro operators to which $\{\bL_k\}_{k\geq -1}$ are dual. This observation is made into a precise statement in Theorem \ref{thm: duality}. 

\subsection{Weight zero descendents}
\label{sec: invariantdescendents}

One of the issues that arise when dealing with descendent invariants is that a priori they depend on a choice of universal sheaf $\BG$ on $M\times X$. Given a universal sheaf $\BG$, the possible universal sheaves are of the form
$\BG'=\BG\otimes p^\ast L$ where $L$ is a line bundle on $M$.
\begin{lemma}
\label{lem: changeuniversalsheaf}
Let $\bE\colon \BD^X\to \BD^X\llbracket\zeta\rrbracket$ be the algebra homomorphism defined by
$$
\bE = e^{\zeta \bR_{-1}}\,,
$$
then given two universal sheaves $\BG$ and $\BG'=\BG\otimes p^\ast L$, the geometric realizations with respect to the two are related by
\[\xi_{\BG'}=\xi_{\BG}\circ \bE|_{\zeta = c_1(L)} \,.\]
\end{lemma}
\begin{proof}
We may write 
\[\bE\big(\chh_i(\gamma)\big) =\sum_{n=0}^i \frac{\zeta^n}{n!}\chh_{i-n}(\gamma)= [e^\zeta \chh_\bullet(\gamma)]_i\] 
which after comparing to
\[\xi_{\BG'}(\chh_\bullet(\gamma))=p_\ast\left(\ch(\BG \otimes p^\ast L)q^\ast \gamma\right)=e^{c_1(L)}\xi_{\BG}(\chh_\bullet(\gamma)).\qedhere\]
yields the result.
\end{proof}
In particular, it follows that if $D$ is such that $\bR_{-1}(D)=0$ then $\xi_\BG(D)=\xi_{\BG'}(D)$. This leads to the definition of the algebra of weight 0 descendents. Its geometric interpretation is summarized in Lemma \ref{Lem: DXinv} and is related to taking rigidification of moduli stacks. Roughly speaking, these classify families of sheaves on $S\times X$ 
up to twisting by line bundles on $L\to S$ and the above discussion formulates the precise interaction between the twisting and the descendents.

\begin{definition}\label{def: invariantdescendents}
For a topological type $\alpha\in K_\sst^0(X)$, we will also denote by 
\[\bR_{-1}\colon \BD^X_{ \alpha}\to \BD^X_{ \alpha}\] the derivation defined on generators by
$\bR_{-1}\ch_i(\gamma)=\ch_{i-1}(\gamma)$,
where $\ch_0(\gamma)$ is interpreted as $\int_X \gamma\cdot \ch(\alpha)$.
We then define
\begin{align*}\BD^X_{\inv}&=\{D\in \BD^X\colon \bR_{-1}(D)=0\}\,,\\
\BD^X_{\inv, \alpha}&=\{D\in \BD^X_{\alpha}\colon \bR_{-1}(D)=0\}\,.
\end{align*} 
\end{definition}

For any weight 0 descendent $D\in \BD^X_{\inv}$, by Lemma \ref{lem: changeuniversalsheaf}, the geometric realization $\xi_{\BG}(D)$ does not depend on the choice of $\BG$. Thus, when $D\in \BD^X_{\inv}$ we will often omit specifying the realization map and write
\[\int_{[M]^\vir} D=\int_{[M]^\vir}\xi_{\BG}(D)\]
for \emph{any} choice of universal sheaf $\BG$. 
The morphism $\BD^X_{\inv}\to H^\bullet(M)$ is defined even without assuming the existence of any universal sheaf $\BG$. This fact can be proven using Lemma \ref{Lem: DXinv}; if $M=M_\alpha$ is a moduli space of topological type $\alpha$, then it admits an open embedding $\iota\colon M\hookrightarrow \CM_\alpha^\rig$ and thus we get a map
\[\begin{tikzcd}
\BD^X_{\inv}\arrow[r, "p_\alpha"]&
\BD^X_{\inv, \alpha}\arrow[r, "\sim"]&
H^\bullet(\CM_\alpha^\rig)\arrow[r, "\iota^\ast"]&
H^\bullet(M).
\end{tikzcd}\]

\begin{example}\label{Ex: DXinvdescendents}
Given $\gamma_1,\gamma_2\in H^\bullet(X)$ we have
\[\chh_1(\gamma_1)\chh_0(\gamma_2)-\chh_0(\gamma_1)\chh_1(\gamma_2)\in \BD^X_{\inv}.\]
One can also check that the lift of $\ch(T^\vir M)$ to $\BD^X$ is in $\BD^X_{\inv}$ using the expression in Example \ref{Ex: Tvir}. A geometric reason for this is going to be given in Example \ref{ex: invdescendents}.
\end{example}

\subsection{Virasoro constraints for sheaves}
\label{sec: virasorosheaves}

To formulate the Virasoro constraints for moduli of sheaves, without a canonical choice of universal sheaf, we must produce relations among weight 0 descendents. This is achieved by combining the Virasoro operators previously introduced in the way which we now describe:
\begin{definition}\label{def: Linv}
The weight 0 Virasoro operator $\Linv \colon \BD^X\to \BD^X$ is defined by
\[\Linv=\sum_{j\geq -1}\frac{(-1)^j}{(j+1)!}\bL_j \bR_{-1}^{j+1}.\]
\end{definition}

The operator $\Linv$ maps $\BD^X$ to $\BD^X_{\inv}$. Indeed, using that $[\bR_{-1}, \bL_j]=(j+1)\bL_{j-1}$ we find
\begin{align*}
    \bR_{-1}\circ \Linv &=\sum_{j\geq -1}\frac{(-1)^j}{(j+1)!}\bL_j \bR_{-1}^{j+2}+\sum_{j\geq -1}\frac{(-1)^j}{(j+1)!}(j+1)\bL_{j-1} \bR_{-1}^{j+1}=0\,.
\end{align*}

In particular, for any $D\in \BD^X$ the integral $\int_M \xi_{\BG}(\Linv(D))$ does not depend on the universal sheaf $\BG$ so we omit the realization homomorphism $\xi_{\BG}$ from the notation. 

\begin{conjecture}\label{conj: invVirasoro}
Let $M$ be a moduli of sheaves as in Section \ref{sec: modulisheavespairs}. Then
\[\int_{[M]^\vir} \Linv(D)=0\quad \textup{ for any }D\in \BD^X.\]
\end{conjecture}

This formulation is different from the ones which appear in previous works, namely \cite{moop, moreira, bree}. There, Virasoro constraints are formulated using a choice of universal sheaf that is natural in each of the moduli spaces considered. 

\begin{definition}
Let $M=M_\alpha$ be a moduli space of sheaves of topological type $\alpha$ and let $\delta\in H^\bullet(X, \BZ)$ be an algebraic class such that $\int_X  \delta\cdot \ch(\alpha) \neq 0$. We say that a universal sheaf $\BG$ is $\delta$-normalized if
\begin{equation}\label{Eq:deltanorm}\xi_\BG(\chh_1(\delta))=0.\end{equation}
\end{definition}

Given $\delta$ as in the definition, we define the operators
\[\bL^\delta_k=\bR_k+\bT_k+\bS_k^\delta,\quad k\geq -1\,,\]
where
\[\bS_k^\delta=-\frac{(k+1)!}{\int_X \delta\cdot \ch(\alpha)}\bR_{-1}\circ \chh_{k+1}(\delta).\]

\begin{remark}\label{rem: rationaldeltanormalized}
If $\delta$ is such that $\int_X  \delta\cdot \ch(\alpha)\neq 0$ then a $\delta$-normalized universal sheaf always exists (and is unique) as an element of the rational $K$-theory of $M\times X$. Precisely, for any universal sheaf $\BG$ we have
\[\BG_\delta=\BG\otimes e^{-\xi_\BG(\chh_1(\delta))/\int_X \delta\cdot \ch(\alpha)}\]
in the rational $K$-theory of $M\times X$ that can be thought of as the unique $\delta$-normalized universal sheaf; here we use $e^{c}$ to denote a rational line bundle with first Chern class equal to the algebraic class $c\in H^2(M, \BQ)$. 

The geometric realization with respect to $\BG_\delta$ is given by $\xi_{\BG_\delta}=\xi_{\BG}\circ \eta$ where $\eta\colon \BD^X\to \BD^X_{\inv,\alpha}$ is 
\[\eta=\sum_{j\geq 0}\frac{1}{j!}\left(-\frac{\chh_1(\delta)}{\int_X \delta\cdot \ch(\alpha)}\right)^j\bR_{-1}^j\,.\]
Thus, we can talk about the geometric realization with respect to a $\delta$-normalized sheaf even if such a sheaf does not exist in the usual sense. Conjecture \ref{conj: normalizedVirasoro} still makes sense in this setting and the proof of Proposition \ref{prop: normalizedvsinv} goes through as well.
\end{remark}

\begin{conjecture}\label{conj: normalizedVirasoro}
Let $M=M_\alpha$ be a moduli of sheaves as in Section \ref{sec: modulisheavespairs} and let $\BG$ be a $\delta$-normalized universal sheaf. Then
\[\int_{[M]^\vir} \xi_\BG(\bL_k^\delta(D))=0  \quad \textup{for any }k\geq -1, \, D\in \BD^X.\]
\end{conjecture}

\begin{proposition}\label{prop: normalizedvsinv}
Conjectures \ref{conj: invVirasoro} and \ref{conj: normalizedVirasoro} are equivalent.
\end{proposition}
\begin{proof}
We begin by observing that we have the identity
\begin{equation}\label{eq: invLnormalized}
\Linv=\sum_{j\geq -1}\frac{(-1)^j}{(j+1)!}\bL_j^\delta \bR_{-1}^{j+1}\,,
\end{equation}
that follows from
\begin{align*}
\sum_{j\geq -1}&\frac{(-1)^j}{(j+1)!}\bS_j^\delta \bR_{-1}^{j+1}\\
&=-\frac{1}{\int_X \delta\cdot \ch(\alpha)}\sum_{j\geq -1}(-1)^j\left( \chh_j(\delta) \bR_{-1}^{j+1}+\chh_{j+1}(\delta)\bR_{-1}^{j+2}\right)=0.
\end{align*}
By \eqref{eq: invLnormalized}, Conjecture \ref{conj: normalizedVirasoro} clearly implies \ref{conj: invVirasoro}.

For the reverse implication we use (backward) induction on $k$. For every $k>\mathrm{virdim}(M)$ the statement of Conjecture $\ref{conj: normalizedVirasoro}$ is clear by degree reasons. Assume now that the result holds for every $k'>k$, and let $F=\chh_1(\delta)\in \BD^X$ satisfying $\xi_\BG(F)~=~0$ by \eqref{Eq:deltanorm}. The weight 0 Virasoro operator applied to $F^{k+1}D$ gives
\begin{equation}\label{eq: invariantimpliesnormalized}
0=\int_{[M]^\vir} \xi_{\BG}\left(\Linv(F^{k+1}D)\right)=\sum_{j\geq -1}\frac{(-1)^j}{(j+1)!}\int_{[M]^\vir}\xi_{\BG}\left(\bL_j^\delta \bR_{-1}^{j+1}(F^{k+1}D)\right).
\end{equation}
Since $\bR_{-1}$ is a derivation and $\bR_{-1}(F)=\chh_0(\delta)=\int_X \delta\cdot \ch(\alpha)$ in $\BD^{X}_\alpha$, we have
\begin{align}\label{eq: invariantimpliesnormalized2}
\nonumber \bR_{-1}^{j+1}(F^{k+1}D)&=\sum_{s=0}^{j+1} \binom{j+1}{s}\bR_{-1}^s(F^{k+1})\bR_{-1}^{j+1-s}(D)\\
&=\sum_{s=0}^{\min(k+1, j+1)}\binom{j+1}{s}\frac{(k+1)!}{(k+1-s)!}r^s F^{k+1-s}R_{-1}^{j+1-s}(D)
\end{align}
in $\BD^{X}_{\alpha}$, where we denote $r=\int_X \delta\cdot \ch(\alpha)$. Note also that the operators $\bL_j^\delta$ satisfy the property that $\bL_{j}^\delta(F D)=F \,\bL_{j}^\delta(D)$ in $\BD^{X}_\alpha$ for every $D$. Since by \eqref{Eq:deltanorm} the geometric realization of $F$ with respect to $\BG$ is zero, the only terms of \eqref{eq: invariantimpliesnormalized2} contributing to \eqref{eq: invariantimpliesnormalized} are the ones with $s=k+1$, thus  \eqref{eq: invariantimpliesnormalized} becomes
\[0=\sum_{j\geq k}\frac{(-1)^j}{(j-k)!} r^{k+1}\int_{[M]^\vir}\xi_{\BG}\left(\bL_j^\delta \bR_{-1}^{j-k}(D)\right)\]
By the induction hypothesis, all the terms with $j>k$ vanish and thus the term with $j=k$ also vanishes, showing that 
\[0=\int_{[M]^\vir} \xi_{\BG}(\bL_k^\delta(D)),\]
which concludes the induction step.\qedhere
\end{proof}

\begin{remark}\label{rem: normalizedpt}
For surfaces or 3-folds $X$ with $H^i(\CO_X)=0$ for $i>0$, the Hilbert scheme of points and the moduli of stable pairs, respectively, are moduli of sheaves in the sense of Section \ref{sec: modulisheavespairs}. The canonical universal sheaves in each of these cases are precisely the $\pt$-normalized universal sheaves and the formulations in \cite{moop, moreira} coincide with Conjecture \ref{conj: normalizedVirasoro}. In \cite{bree}, the author considers moduli spaces of torsion free stable sheaves on surfaces with $h^{1,0}=h^{2,0}=0$; the sheaves in such moduli spaces have fixed determinant $\Delta$. The author uses a geometric realization with respect to the sheaf
$\BG\otimes \det(\BG)^{-1/r}\,.$
This is not a universal sheaf for $M$ in the sense we use in this paper since its restriction to $\{G\}\times X$ is $G\otimes \det(G)^{-1/r}=G\otimes \Delta^{-1/r}$ rather than $G$; however,
\[\BG\otimes \det(\BG)^{-1/r}\otimes q^\ast\Delta^{1/r}\]
recovers the $\pt$-normalized universal sheaf of Remark \ref{rem: rationaldeltanormalized}. The equivalence between van Bree's formulation of Virasoro for the moduli of stable sheaves of positive rank and Conjecture \ref{conj: normalizedVirasoro} is explained by Lemma \ref{lem: twist}.
\end{remark}

\subsection{Virasoro constraints for pairs}
\label{sec: virasoropairs}

Since for moduli spaces of pairs we have a uniquely defined universal object $q^\ast V\to \BF$, there is no need to use the weight 0 Virasoro operator for pairs. We need, however, to slightly modify the operators to account for the different obstruction theory.

Let $V$ be a fixed sheaf on $X$ and let $P$ as in  Section \ref{sec: modulisheavespairs} be a moduli space of pairs parametrizing sheaves $F$ together with a map $V\to F$. The moduli $P$ comes equipped with a (unique) universal sheaf $\BF$ on $P\times X$ and a universal map $q^\ast V\to \BF$. We conjecture that descendent invariants obtained by integration on moduli of pairs are constrained by Virasoro operators which are similar to the ones introduced in Section \ref{sec: Virarosorops}. Define the operators 
\[\bL_k^V\colon \BD^X\to \BD^X\]
by $\bL_k^V=\bR_k+\bT_k^V$ where
\[\bT_k^V=\bT_k-k!\chh_k\left(\ch(V)^\vee\td(X)\right).\]
The operator $\bL_k^V$ can be described in an alternative way that should make its definition more natural and that will be closer to the vertex algebra language that we will introduce later. Let 
\[\Dpa=\BD^X\otimes\BD^X\]
be the algebra of pair descendents. We denote the generators of the first copy of $\BD^X$ by  $\ch_i^{\HH, \CV}(\gamma)$ and the generators of the second copy by  $\ch_i^{\HH, \CF}(\gamma)$. Given the universal pair $q^\ast V\to \BF$ on $P\times X$, we have a geometric realization map
\[\xi_{(q^\ast V,\BF)}\colon \Dpa\to H^\bullet(P)\]
that is defined by
\begin{align*}\xi_{(q^\ast V, \BF)}\left(\ch_i^{\HH, \CV}(\gamma)\right)&=\xi_{q^\ast V}\left(\chh_i(\gamma)\right)=\begin{cases}\int_X \gamma\cdot \ch(V)&\textup{ if }i=0\\
0 & \textup{ otherwise}\end{cases}\\
\xi_{( q^\ast V, \BF)}\left(\ch_i^{\HH, \CF}(\gamma)\right)&=\xi_{\BF}\left(\chh_i(\gamma)\right).
\end{align*}
This geometric realization map factors through
\begin{center}
\begin{tikzcd}
\Dpa\arrow[d, swap, "\xi_V"]\arrow[rd, "\xi_{(q^\ast V,\BF)}"]& \\
\BD^X\arrow[r, "\xi_{\BF}"]&
H^\bullet(P)
\end{tikzcd}
\end{center}
where $\xi_{V}$ is defined to send 
\[\ch_i^{\HH, \CV}(\gamma)\mapsto \delta_i \int_X \gamma\cdot \ch(V)\textup{ and }\ch_i^{\HH, \CF}(\gamma)\mapsto \ch_i^{\HH}(\gamma)\,;\]
here $\delta_i=\delta_{i,0}$ is equal to $1$ if $i=0$ and $0$ otherwise.
We define the pair Virasoro operators $\bL_k^{\pa}\colon \Dpa\to \Dpa$, for $k\geq -1$, as the sum $\bR_k^{\pa}+\bT_k^{\pa}$ where
\begin{enumerate}
\item $\bR_k^{\pa}$ is a derivation defined on generators in the same way as $\bR_k$; in other words,
\[\bR_k^{\pa}=\bR_k\otimes \id+\id\otimes \bR_k.\]
\item $\bT_k^{\pa}$ is the operator of multiplication by the element
\[\bT_k^{\pa}=\sum_{i+j=k} (-1)^{\dim X- p^L}i!j!\ch_i^{\HH, \CF-\CV}\ch_j^{\HH, \CF}(\td(X))\in \Dpa\]
where
\[\ch_i^{\HH, \CF-\CV}\coloneqq \ch_i^{\HH, \CF}-\ch_i^{\HH, \CV}.\]
\end{enumerate}

The definition of $\bT_k^{\pa}$ suggests an intimate relation between the linear part of the Virasoro operators and the obstruction theory of the moduli spaces we consider; recall that the deformation-obstruction theory of the pair moduli space $P$ at $V\to F$ is governed by
\[\RHom\left([V\to F], F\right)\]
and use Example \ref{Ex: Tvir}.

The operators $\bL_k^V$ are obtained from $\bL_k^{\pa}$ by the following commutative diagram:
\begin{center}
\begin{tikzcd}
\Dpa\arrow[r, "\bL_k^{\pa}"]\arrow[d, "\xi_V"]&
\Dpa\arrow[d, "\xi_V"]\\
\BD^X\arrow[r, "\bL_k^V"]&
\BD^X.
\end{tikzcd}
\end{center}
This holds due to the identity
\begin{align*}
\sum_{t}(-1)^{\dim X-p_t^L}\left(\int_X \gamma_t^L \cdot \ch(V) \right)\chh_k(\gamma_t^R)&=\sum_t \left(\int_X \gamma_t^L \cdot \ch(V)^\vee \right)\chh_k(\gamma_t^R)\\
&=\chh_k(\ch(V)^\vee \td(X)).
\end{align*}



It can be shown that the operators $\{\bL_k^V\}_{k\geq -1}$ and $\{\bL_k^{\pa}\}_{k\geq -1}$ satisfy the Virasoro bracket relations.

\begin{conjecture}\label{conj: pairvirasoro}
Let $P$ be a moduli of pairs as in Section \ref{sec: modulisheavespairs} with universal pair $q^\ast V\to \BF$ on $P\times X$. Then
\[\int_{[P]^\vir}\xi_\BF\big(\bL_k^V(D)\big)=0\quad \textup{ for any }k\geq 0, \, D\in \BD^X.\]
\end{conjecture}
Equivalently, the pair Virasoro constraints can be formulated as
\[\int_{[P]^\vir} \xi_{(q^\ast V,\BF)}\big(\bL_k^{\pa}(D)\big)=0\quad \textup{ for any }k\geq 0, \, D\in \Dpa.\]

\subsection{Invariance under twist}
\label{sec: invariancetwist}

Suppose that $M=M_\alpha$ is the moduli space of slope semistable sheaves with respect to a polarization $H$ with topological type $\alpha$. Then, the moduli space $M'=M_{\alpha(H)}$ is isomorphic to $M$ by sending a sheaf $[F]\in M$ to $[F']=[F\otimes H]$. As expected, the Virasoro constraints on $M$ and $M'$ are equivalent as we now proceed to verify. 

Suppose that $\BG$ is a universal sheaf on $M\times X$. The universal sheaf on $M'\times X$ is identified with $\BG'=\BG\otimes q^\ast H$ via the isomorphism $M'\times X\cong M\times X$. Define an algebra isomorphism $\mathsf F\colon \BD^X\to \BD^X$ by
\[\bF(\chh_i(\gamma))=\chh_i(e^{c_1(H)} \gamma)\,.\]
Then the following diagram commutes:
\begin{center}
\begin{tikzcd}
\BD^X\arrow[r, "\xi_{\BG'}"] \arrow[d, "\bF"]& H^\ast(M')\arrow[d, "\sim" {rotate=90, below}]\\
\BD^X\arrow[r, "\xi_{\BG}"] & H^\ast(M)
\end{tikzcd}
\end{center}

\begin{lemma}\label{lem: twist}
The isomorphism $\bF$ commutes with the Virasoro operators, i.e.\[\bL_k\circ \bF=\bF\circ \bL_k\textup{ for $k\geq -1\quad$   and }\quad\Linv\circ \bF=\bF\circ \Linv.\]
Thus, Conjecture \ref{conj: invVirasoro} holds for $M$ if and only if it holds for $M'$.	
\end{lemma}
\begin{proof}
It is enough to show that $\bL_k$ and $\bF$ commute. Commutativity with $\Linv$ follows immediately from its definition; the equivalence of Virasoro constraints follows from the commutativity and the diagram before the Lemma.

The commutativity with the derivation part, i.e., $\bR_k\circ \bF=\bF\circ \bR_k$, is straightforward. To show commutativity with the operator $\bT_k$ we need $\bF(\bT_k)=\bT_k$ (recall that $\bT_k$ denotes both an element in $\BD^X$ and the operator which is multiplication by that element). 
\begin{align*}\bF(\bT_k)&=\sum_{i+j=k}\sum_t(-1)^{\dim X-p^L_t}i!j!\chh_i(e^{c_1(H)}\gamma_t^L)\chh_j(e^{c_1(H)}\gamma_t^R)\\
&=\sum_{i+j=k}\sum_t \sum_{a, b\geq 0}(-1)^{\dim X-p^L_t}i!j!\frac{1}{a!b!}\chh_i(c_1(H)^a\gamma_t^L)\chh_j(c_1(H)^b\gamma_t^R)\\
&=\sum_{i+j=k} \sum_{a, b\geq 0}(-1)^{\dim X-p^L-a}i!j!\frac{1}{a!b!}\chh_i\chh_j(c_1(H)^{a+b}\td(X))\\
&=\sum_{i+j=k} (-1)^{\dim X-p^L}i!j!\chh_i\chh_j(\td(X))=\bT_k.
\end{align*}
The last line uses the fact that for fixed $c>0$ the sum $\sum\limits_{a+b=c}\frac{(-1)^a}{a!b!}$ vanishes.\qedhere
\end{proof}

\subsection{Variants for the fixed determinant theory}
\label{sec: fixeddet}

As we pointed out in Remark \ref{rem: tracelessobs}, this paper is mostly concerned with moduli spaces of sheaves without fixed determinant; in other words, our obstruction theories use full $\textup{Ext}$ groups instead of traceless $\textup{Ext}$ groups. This contrasts with the situation for Pandharipande-Thomas stable pairs or Hilbert schemes of points studied in \cite{moreira}. There, we see a new term appearing in the Virasoro operators which we now recall.

\begin{definition}
Given a class $\gamma\in H^\bullet(X)$ denote by $\bR_{-1}[\gamma]$ the superderivation (of degree $\deg(\gamma)-2)$ acting on generators by
\[\bR_{-1}[\gamma]\chh_i(\gamma')=\chh_{i-1}(\gamma \gamma').\]
For $k\geq -1$ we define the operator $\bS_k\colon \BD^X\to \BD^X$ by
\[\bS_k=(k+1)!\sum_{p_t^L=0}\bR_{-1}[\gamma_t^L]\circ\chh_{k+1}(\gamma_t^R)\]
where the sum runs over the terms $\gamma_t^L\otimes\gamma_t^R$ in the K\"unneth decomposition of $\Delta_\ast 1$ such that $p_t^L=0$.
\end{definition}
\begin{remark}
The part of $\bS_k$ corresponding to the K\"unneth component $1\otimes \pt$ is equal to $-r\bS_k^\pt$ where $r=\ch_0(\alpha)=\rk(\alpha)$. When $h^{0, q}(X)=0$ for $q>0$ the appearance of the operator $\bS_k$ is already explained in Remark \ref{rem: normalizedpt} as being related to the $\pt$-normalized universal sheaf. 
\end{remark}

We now proceed to explain the appearance of the operator $\bS_k$ in the more general case in which $h^{0, 1}\neq 0$ but $h^{0, q}=0$ for $q>1$. Let $\alpha$ be such that $r=\rk(\alpha)>0$ and $M=M_\alpha$ be a moduli space parametrizing semistable sheaves with topological type~$\alpha$. Let $\Delta\in \Pic(X)$ be a fixed line bundle  such that $c_1(\Delta)=c_1(\alpha)$. We let $M_\Delta\subseteq M$ be the moduli space of sheaves on $F\in M$ with fixed determiant $\det(F)=\Delta$; i.e., $M_\Delta$ is the pullback
\begin{equation}\label{eq: cartesianpic}
\begin{tikzcd}
M_\Delta\arrow[r, hookrightarrow]\arrow[d]& M\arrow[d, "\det"]\\
\{\Delta\} \arrow[r, hookrightarrow]&
\Pic^{c_1(\alpha)}(X).
\end{tikzcd}
\end{equation}
The moduli space $M_\Delta$ has a 2-term obstruction theory given by 
\begin{align*}
\textup{Tan}&=\Ext^1(G,G)_0\\
\textup{Obs}&=\Ext^2(G,G)_0=\Ext^2(G, G)\\
0&=\Ext^{>2}(G,G).
\end{align*}
Given such a moduli space, there is a unique universal sheaf (possibly rational, in the same sense of Remark \ref{rem: rationaldeltanormalized}) $\BG_\Delta$ on $M_\Delta\times X$ such that $\det(\BG_\Delta)=q^\ast \Delta$. 
We suppose also that the Jacobian $\Pic^0(X)$ acts on $M$ in the natural way; that is,
\[[F]\in M\Rightarrow [F\otimes L]\in M\]
for $L\in \Pic^0(X)$. The two main examples to keep in mind are
\begin{enumerate}
\item $M=M_C(r, d)$ a moduli of stable bundles on a curve $C$ and $\Delta$ a line bundle of degree $d$;
\item $M=M_S^H(r, c_1,c_2)$ a moduli of stable sheaves on a surface $S$ with $p_g(S)=0$ and $\Delta$ a line bundle with $c_1(\Delta)=c_1$. Note that if we take $r=1, c_1=0$ and $\Delta=\CO_S$ we recover the Hilbert scheme of points on $S$
\[M_S^H(1, \CO_S, n)= S^{[n]}\,.\]
\end{enumerate}

\begin{proposition}\label{prop: fixeddet}
Suppose that $X$ is such that $h^{0, q}=0$ for $q>1$ and $M, M_\Delta$ are as described before. Suppose also that the Virasoro constraints (Conjecture \ref{conj: invVirasoro}) hold for $M$.  Then, we have
\[\int_{[M_\Delta]^\vir}\xi_{\BG_\Delta} \big(\mathcal L_k(D)\big)=0\quad \textup{ for any }k\geq -1, \, D\in \BD^X\]
where 
\[\mathcal L_k=\bL_k-\frac{1}{r}\bS_k\,.\footnote{The $-$ sign in the $\bS_k$ operator does not appear in \cite{moreira} due to the fact that in \cite{moreira} the geometric realization is taken with respect to $-\BG_\Delta$ instead of $\BG_\Delta$.}\]
\end{proposition}
\begin{proof}
We use the tensoring morphism
\[\pi\colon \widetilde M=M_\Delta\times \Pic^0(X)\to M,\quad (F,L)\mapsto F\otimes L.\]
This morphism is a part of the Cartesian diagram 
\begin{center}
\begin{tikzcd}
M_\Delta\times \Pic^0(X) \arrow[r, "\pi"] \arrow[d, "i\times \id"', hook] & M \arrow[d, "{(\id,\det)}", hook] \\
M\times \Pic^0(X) \arrow[r, "f"]                       & M\times \Pic^{c_1(\alpha)}(X)                      
\end{tikzcd}
\end{center}
where $f(G,L)\coloneqq(G\otimes L,\Delta\otimes L^r)$ with $r=\rk(\alpha)$. The morphism $f$ is essentially the multiplication by $r$ map of the abelian variety $\Pic^0(X)$. This implies that $\pi$ is an \'etale morphism of degree $r^{2g}$ where $g=h^{0,1}(X)=\dim(\Pic^0(X))$. 

On the other hand, $f$ respects the obstruction theory because the obstruction theory of $M$ is invariant under tensoring by line bundles. This implies the compatibility between virtual classes
\[\pi^*[M]^\vir=[M_\Delta]^\vir\times [\Pic^0(X)]=:[\widetilde M]^\vir
\]
and $\pi_*[\widetilde M]^\vir=r^{2g}\cdot[M]^\vir$.

We now apply the above discussion to translate the Virasoro constraints of $M$ to that of $M_\Delta$. Let $\CP$ be the Poincaré bundle on $\Pic^0(X)\times X$ (i.e., the $\pt$-normalized universal bundle of $\Pic^0(X)$) and let $\BG$ be the $\pt$-normalized universal sheaf on $M\times X$. Since $\widetilde{\BG}=\BG_\Delta\boxtimes \CP$ is also $\pt$-normalized it follows that $(\pi\times \id)^\ast \BG=\widetilde \BG$ (in rational $K$-theory) and thus $\pi^\ast\circ \xi_{\BG}=\xi_{\widetilde \BG}$. Using the compatibilities between virtual classes and Proposition \ref{prop: normalizedvsinv} we get
\begin{equation}\label{eq: virasoromtilde}
\int_{[\widetilde M]^\vir}\xi_{\widetilde \BG}\big(\big(\bL_k+\bS_k^\pt\big)(D)\big)=0\,
\end{equation}
where $$\bS_k^\pt=-\frac{1}{r}\bR_{-1}\ch_{k+1}(\pt)\,.$$
Recall that $g=h^{0,1}(X)=h^{1,0}(X)$ and let $\{e_1, \ldots, e_g\}\subseteq H^{0,1}(X)$ and $\{f_1, \ldots, f_g\}\subseteq H^{1,0}(X)$ be basis; let $\{\hat e_1, \ldots, \hat e_g\}\subseteq H^{m, m-1}(X)$, $\{\hat f_1, \ldots, \hat f_g\}\subseteq H^{m-1, m}(X)$ be their dual basis, meaning that 
\[\int_X e_i \hat e_j=\delta_{ij}=\int_X f_i \hat f_j\,.\]
Let 
\[\tau_j=\xi_{\CP}(\chh_1(\hat e_j))\in H^{1,0}(\Pic^0(X))\quad \textup{and}\quad \rho_j=\xi_{\CP}(\chh_0(\hat f_j))\in H^{0,1}(\Pic^0(X))\]
so that
\[c_1(\CP)=\sum_{j=1}^g \tau_j\otimes e_j+\sum_{j=1}^g \rho_j\otimes f_j\in H^2(\Pic^0(X)\times X).\]
The Jacobian is topologically a real torus of dimension $2g$ and its cohomology is the exterior algebra generated by the classes $\{\tau_j, \rho_j\}_{j=1}^g$.
Let
\[\omega=\sum_{j=1}^g \rho_j\tau_j\in H^2(\Pic^0(X)).\]
By rescaling the elements of the basis, we may assume that $\int_{\Pic^0(X)}\prod_{j=1}^g \rho_j\tau_j=\frac{1}{g!}$ so that $\omega^g\in H^{2g}(\Pic^0(X))$ is the class Poincaré dual to a point in $\Pic^0(X)$.

 Since $\ch_0(\BG_\Delta)=r$ and $\ch_1(\BG_\Delta)=q^\ast c_1(\Delta)$, we have
\[\xi_{\widetilde \BG}(\chh_1(\hat e_j))=r\tau_j \quad \textup{and}\quad \xi_{\widetilde \BG}(\chh_0(\hat f_j))=r\rho_j\]
in $H^\bullet\big(\widetilde M\big)=H^\bullet(M_\Delta\times \Pic^0(X))$ where we omit the obvious pullback. Let
\[W=\sum_{j=1}^g \chh_0(\hat f_j)\chh_1(\hat e_j)\in \BD^X\quad \textup{so that}\quad \xi_{\widetilde \BG}(W)=r^2\omega.\]
We will apply equation \eqref{eq: virasoromtilde} to descendents of the form $W^g D$ for some $D\in \BD^X$ and use that to deduce the Proposition. We have
\begin{align}\label{eq: virasoroMtilde2}
0&=\int_{[\widetilde M]^\vir}\xi_{\widetilde \BG}\big((\bL_k+\bS_k^\pt)(W^g D)\big)\\
&=r^{2g}\int_{[\widetilde M]^\vir}\omega^g \xi_{\widetilde \BG}\big((\bL_k+\bS_k^\pt)(D)\big)+\int_{[\widetilde M]^\vir}\xi_{\widetilde \BG}\big(\bR_k(W^g) D)\big)\nonumber\\
&=r^{2g}\int_{[M_\Delta]^\vir}\xi_{\BG_\Delta}\big((\bL_k+\bS_k^\pt)(D)\big)+gr^{2g-2}\int_{[\widetilde M]^\vir}\omega^{g-1}\xi_{\widetilde \BG}\big( \bR_k(W) D)\big).\nonumber
\end{align}
We now rewrite the second integral in terms of an integral in $M_\Delta$. We have
\[\xi_{\widetilde\BG}\big(\bR_k(W)\big)=(k+1)!\sum_{j=1}^g \xi_{\widetilde\BG}\big(\chh_0(\hat f_j)\chh_{k+1}(\hat e_j)\big)=r(k+1)!\sum_{j=1}^g \rho_j\xi_{\widetilde\BG}\big(\chh_{k+1}(\hat e_j)\big).\]
\begin{claim}
Let $h\colon \widetilde M\to M_\Delta$ be the projection and let $D\in \BD^X$. Then we have
\[h_\ast\big(g\omega^{g-1}\rho_j\xi_{\widetilde \BG}(D)\big)=\xi_{\BG_\Delta}\big(\bR_{-1}[e_j]D\big).\]
\end{claim}\begin{proof}[Proof of claim]
Let $J\subseteq H^\bullet\big(\widetilde M\big)$ be the annihilator ideal of $\omega^{g-1}\rho_j$; in particular, $H^{\geq 2}(\Pic^0(X))\subseteq J$. By definition,
\begin{align*}
\xi_{\widetilde \BG}\big(\chh_\bullet(\gamma)\big)&=p_\ast\big(\ch(\BG_\Delta)\ch(\CP)q^\ast \gamma\big)=p_\ast\big(\ch(\BG_\Delta)\big(1+(p^\ast \tau_j) (q^\ast e_j)+\ldots\big)q^\ast \gamma\big)\\
&=\xi_{\BG_\Delta}\big(\chh_\bullet(\gamma)\big)+\tau_j\xi_{\BG_\Delta}\big(\chh_\bullet(e_j\gamma)\big)+\ldots
\end{align*}
where the terms in $\ldots$ are all in the ideal $J$ and where we omit all the pullbacks via the projections $\widetilde M\to M_\Delta$ and $\widetilde M\to \Pic^0(X)$. Thus,
\begin{align*}
\xi_{\widetilde \BG}(D)=\xi_{\BG_\Delta}(D)+\tau_j\xi_{\BG_\Delta}\big(\bR_{-1}[e_i]D\big)+\ldots
\end{align*}
Since 
\[\int_{\Pic^0(X)}\omega^{g-1}\rho_j\tau_j=(g-1)!\int_X \prod_{l=1}^g \rho_l\tau_l=\frac{1}{g}\]
the claim is proven.\qedhere
\end{proof}
Given the claim, \eqref{eq: virasoroMtilde2} becomes
\[
0=r^{2g}\int_{M_\Delta}\xi_{\BG_\Delta}\big((L_k+S_k^\pt)(D)\big)+r^{2g-1}(k+1)!\int_{M_\Delta}\sum_{j=1}^g\xi_{\BG_\Delta}\big(R_{-1}[e_j]\chh_{k+1}(\hat e_j) D)\big)
\]
Since $h^{0, q}=0$ for $q>1$, we have
\[\bS_k=-r\bS_k^\pt-\sum_{j=1}^g \bR_{-1}[e_j]\chh_{k+1}(\hat e_j) \]
so the Proposition follows. \qedhere
\end{proof}
\begin{remark}\label{rem: fixed det}
When $S$ is a surface with $p_g>0$ we do not know a rigorous interpretation for the terms of the $\bS_k$ with $\gamma_t^L\in H^{0,2}(S)$. Heuristically, these should be related to a diagram like \eqref{eq: cartesianpic} where the Picard group is interpreted as a derived stack, see \cite{STV}. 
\end{remark}

\begin{example}\label{ex: rank2}
Let $M=M_C(2, \Delta)$ be the moduli space of stable bundles on a curve $C$ of genus $g$ with rank $2$ and fixed determinant $\Delta\in \Pic(C)$ of odd degree; this is a smooth moduli space of dimension $3g-3$. This moduli space has been studied a lot in the past and in particular its ring structure and descendent integrals are completely understood. We use this example to illustrate what kind of information the Virasoro constraints provide on descendent integrals. 
Let $\{e_1, \ldots, e_g\}\subseteq H^{0,1}(C)$ and $\{f_1, \ldots, f_g\}\subseteq H^{1,0}(C)$ be dual basis. Newstead proved in \cite[Theorem 1]{newstead} that the cohomology $H^\bullet(M)$ is generated by the classes
\[\eta=-2\xi\big(\chh_1(\1)\big), \quad \theta=4\xi\big(\chh_2(\pt)\big), \quad \xi_j=-\xi\big(\chh_1(e_j)\big), \quad \psi_j=\xi\big(\chh_2(f_j)\big).\footnote{In \cite{thaddeus, newstead} the classes $\eta, \theta, \xi_j, \psi_j, \zeta$ are called, respectively, $\alpha, \beta, \psi_{j+g}, \psi_{j}, \gamma$.}\]

The geometric realization $\xi$ is taken with respect to the sheaf $\BG\otimes \det(\BG)^{-1/2}$ (see Remark \ref{rem: normalizedpt}). Every descendent can be written explicitly in terms of the classes $\eta, \theta, \xi_j, \psi_j$. By \cite[(24) Proposition]{thaddeus}, one can then write every descendent integral in terms of integrals of products of the classes $\eta, \theta$ and
\[\zeta=2\sum_{j=1}^g\psi_j \xi_j.\]
These integrals are fully determined in \cite[(30)]{thaddeus}: for $m, k, p$ such that $m+2k+3p=3g-3$ we have
\begin{equation}\label{eq: integralsM21}
\int_M \eta^m \theta^k\zeta^p=(-1)^{g-1-p}\frac{m!g!}{q!(g-p)!}2^{2g-2-p}(2^q-2)B_q
\end{equation}
where $q=m+p-g+1$ and $B_q$ is a Bernoulli number. A careful combinatorial analysis shows that the Virasoro constraints given by Proposition \ref{prop: fixeddet} for $M$ are equivalent to the relations
\[(g-p)\int_M \eta^m \theta^k\zeta^p=-2m\int_M \eta^{m-1} \theta^{k-1}\zeta^{p+1}\]
which of course follows from \eqref{eq: integralsM21}; it would be interesting to have a direct and simpler proof of this identity. Note that the Virasoro constraints do not capture the most interesting part of \eqref{eq: integralsM21} which is the Bernoulli number. In some sense this is to be expected: since the Virasoro constraints hold in great generality (in particular are invariant under wall-crossing), it should not be expected that they can capture special information about particular moduli spaces. 
\end{example}

\section{Vertex operator algebra}

To prove the conjectures from the previous sections, we will apply the wall-crossing machinery introduced by Joyce which relies on his geometric construction of vertex algebras. The resulting vertex algebras have been described in many cases as lattice vertex algebras and as such admit a natural family of conformal elements. This section focuses on developing the necessary vertex operator algebra language, including the definition of lattice vertex operator algebras in the generality that we need, Borcherds Lie algebra associated to a vertex algebra and the notion of primary states.

\subsection{Vertex operator algebra}
\label{sec: voadefinition}
 There are many equivalent formulations of vertex algebras. We will follow the definitions and notation in \cite{Ka98}. In particular, vertex algebra for us means $\BZ$-graded vertex superalgebra over $\BC$.
 
 \begin{definition}
 \label{Def: axioms}
 A \textit{vertex algebra} is a $\BZ$-graded vector space $V_{\bullet}$ over $\BC$ together with 
\begin{enumerate}
\item a \textit{vacuum vector} $\ket{0}\in V_0$, 
    \item a linear operator 
    $T\colon V_\bullet\to V_{\bullet+2}$
   called the \textit{translation operator},
    \item and a \textit{state-field correspondence} which is a degree 0 linear map
    $$
    Y\colon V_\bullet\longrightarrow \End(V_\bullet)\llbracket z,z^{-1} \rrbracket\,,
    $$
    denoted by 
    $$Y(a,z)\coloneqq \sum_{n\in\BZ}{a_{(n)}}z^{-n-1}\,,
    $$
   where $a_{(n)}:V_\bullet\rightarrow V_{\bullet+\deg(a)-2n-2}$ and $\deg (z) =-2$.
\end{enumerate}

They are required to satisfy conditions which can be formulated in many different ways. We choose our favorite version:
\begin{enumerate}
    \item \textit{(vacuum)} $T\ket{0} = 0$\,, $Y(|0\rangle, z)=\textnormal{id}$\,, $Y(a,z)|0\rangle\in a+zV_\bullet\llbracket z\rrbracket$\,,
    \item \textit{(translation covariance)} $[T,Y(a,z)] = \frac{d}{dz} Y(a,z)$ for any $a\in V_\bullet$\,,
    \item \textit{(locality)} for any $a,b\in V_{\bullet}$, there is an $N\gg 0$ such that 
    $$
    (z-w)^N[Y(a,z),Y(b,w)] = 0\,,
    $$
    where the supercommutator is defined on $\text{End}(V_\bullet)$ by
    $$
    [A,B] =A\circ B -  (-1)^{|A| |B|}B\circ A\,.
    $$
\end{enumerate}
\end{definition}
For later purposes it is useful to note that these axioms imply the two following identities which were used by Borcherds \cite[Section 4]{Borcherds} to originally define vertex algebras:
\begin{align}
\label{Eq: skewsymmetry}
a_{(n)}b &= \sum_{i\geq 0}(-1)^{|a||b|+i+n+1}\frac{T^i}{i!} b_{(n+i)}a\,,\\
\label{Eq: skewjacobi}\big(a_{(m)}b\big)\raisebox{0.2 pt}{\ensuremath{_{(n)}}} c
&=\sum_{i\geq 0}(-1)^i \binom{m}{i}\Big[a_{(m-i)}\big(b_{(n+i)}c\big) -(-1)^{|a||b|+m}b_{(m+n-i)}\big(a_{(i)}c\big)\Big]\,.
\end{align}
They are a refinement of skew-symmetry and the Jacobi identity to the setting of vertex algebras. 
Additionally, we will also use 
\begin{equation}
\label{Eq: translationu}
(Ta)_{(n)} = -n\cdot a_{(n-1)}
\end{equation}
which follows from the more general reconstruction result.
To understand it, one needs to make sense of a product of two fields $Y(v,z)$ and $Y(w,z)$ which naively could contain infinite sums for each coefficient. For this one uses the following trick.
\begin{definition}[{\cite[(2.3.5)]{Ka98}}]
 A \textit{normal ordering} $:-:$ is defined by 
$$
:v_{(k)} w_{(l)}: \,=\begin{cases}(-1)^{|v||w|}w_{(l)}v_{(k)}&\text{if} \quad k\leq 0,l>0\,,\\
v_{(k)}w_{(l)}& \text{otherwise}\,.\end{cases}
$$
In general, this can be extended to any monomial in $v_{(k)}$ by iterating the above operation on the neighboring terms until all terms with non-positive index $k$ are on the right. 
\end{definition}

\begin{theorem}[{\cite[Corollary 4.5]{Ka98}}]
    Let $a^1,\cdots, a^n\in V_{\bullet}$ be a finite collection of elements and $k_1,\dots,k_n\in \BZ_{\geq 0}$, then a general field can be described as
    $$
    Y\big(a^1_{(-k_1-1)}\cdots a^n_{(-k_n-1)}\ket{0},z\big)
= \frac{1}{k_1!k_2!\cdots k_n !}:\frac{d^{k_1}}{(dz)^{k_1}}Y(a^1,z)\cdots \frac{d^{k^n}}{(dz)^{k_n}}Y(a^n,z):\,.
    $$
\end{theorem}
To get \eqref{Eq: translationu} simply use $Ta = a_{(-2)}\ket{0}$ and compare the coefficients on both sides.

We now recall the definition of an additional structure on vertex algebras called a conformal element or conformal vector. They induce a homomorphism from the Virasoro vertex algebra to the given vertex algebra. As we will see in Section \ref{sec: VOA from sheaf theory}, conformal element in the setting of Joyce's vertex algebra gives rise to a compact way to summarize all the information contained in the Virasoro constraints.
\begin{definition}\label{Def: conformal vertex algebra}
 A \textit{conformal element} $\omega$ on a vertex algebra $V_\bullet$ is an element of $V_4$ such that its associated fields $L_n=\omega_{(n+1)}$, defined by
  $$Y(\omega,z) = \sum_{n\in \mathbb{Z}}L_nz^{-n-2}\,,$$
satisfy
 \begin{enumerate}
     \item the Virasoro bracket 
     \[\big[L_n,L_m\big] = (n-m)L_{m+n} +\frac{n^3-n}{12}\delta_{n+m,0}\cdot C
     \,,\]
     where $C\in \BC$ is a constant called the \textit{central charge} of $\omega$,
     \item $L_{-1}  = T\,,$
     \item and $L_0$ is diagonalizable. 
 \end{enumerate}

A vertex algebra $V_\bullet$ together with a conformal element $\omega$ is called a \textit{conformal vertex algebra} or \textit{vertex operator algebra}. We denote by $V^\omega_{\bullet}$ the \textit{conformal grading} on the underlying vector space, where $V^\omega_i$ is the $i\in \BC$ eigenspace of $L_0$. We denote the \textit{conformal degree} by
$$
\deg_\omega(a) =i\quad \text{if} \quad a\in V^\omega_i\,,
$$
to distinguish it from the original degree on $V_\bullet$. 
\end{definition}

\subsection{Lattice vertex (operator) algebras}
We next describe a particular construction of a vertex operator algebra which we will be working with, called \textit{lattice vertex algebras}. The following theorem is a summary (and slight generalization, see footnote \ref{footnote: kac}) of the constructions and statements in \cite[Sections 3.5, 3.6 and 5.5]{Ka98}. In the rest of the section we will explain the construction in further detail.

\begin{theorem}[Kac]\label{thm: voaconstruction}
Assume that we have the following data:
\begin{enumerate}
    \item A $\BZ_2$-graded $\BC$-vector space $\Lambda=\Lambda_{\overline 0}\oplus \Lambda_{\overline 1}$ with a symmetric bilinear pairing $Q\colon \Lambda\times \Lambda\to \BC$ which is a direct sum of its restrictions $Q_{\overline{i}}:\Lambda_{\overline i}\times \Lambda_{\overline i}\to \BC$.
\item The pairing $Q$ is obtained as the symmetrization of a not necessarily symmetric pairing $q\colon \Lambda\times \Lambda\to \BC$, i.e.,\footnote{We use symmetrization instead of supersymmetrization as in \cite{Gr19}.}
\[Q(v, w)=q(v, w)+q(w, v).\] 
\item An abelian group of a finite rank $\Lambda_{\sst}$ that admits a $\BC$-linear inclusion 
\[\Lambda_{\sst}\otimes_{\BZ}\BC\hookrightarrow \Lambda_{\overline 0}\]
such that the restriction of $q$ to $\Lambda_\sst$ is integer valued. This makes $(\Lambda_\sst, Q)$ an even generalized integeral lattice in the sense of \cite[Section 3]{Gr19}. 
\end{enumerate}
Then there is a uniquely defined vertex algebra $V_\bullet$ whose underlying vector space is
\[V_\bullet=\BC[\Lambda_{\sst}]\otimes \BD_{\Lambda},\]
where
\[\BD_{\Lambda} = \SSym\big[\text{CH}_{\Lambda}\big]\quad\text{and}\quad \text{CH}_{\Lambda}  =\bigoplus_{k>0} \Lambda \cdot t^{-k}\,.\]
The state-field correspondence is determined by equations \eqref{Eq: explicitfield}, \eqref{Eq: field2} and \eqref{Eq: reconstruction} and the translation operator is defined by equation \eqref{Eq: translation}. 

Suppose moreover that we have:
\begin{enumerate}
\item[(4)] The pairing $Q$ is non-degenerate.
\item[(5)] A decomposition of $\Lambda_{\overline 1}$ as a sum of two maximal isotropic subspaces, i.e.,
\[\Lambda_{\overline 1}=I\oplus \hat I.\]
\end{enumerate}
Then $V_\bullet$ admits a conformal element $\omega$ defined as \eqref{Eq: omega} whose central charge is
\[\mathrm{sdim}\, \Lambda\coloneqq\dim \Lambda_{\overline 0}-\dim \Lambda_{\overline{1}}.\]
\end{theorem}

We will use the notation $v_{-k}=v \cdot t^{-k}\in \BD_\Lambda$ and $e^\alpha\in \BC[\Lambda_{\sst}]$ for the elements associated to $v\in \Lambda$ and $\alpha\in \Lambda_{\sst}$, respectively. The $\BZ$-grading of $V_\bullet$ is defined by the degree assignment
$$
\deg \Big(e^{\alpha}\otimes v_{-k_1}^1\cdots v_{-k_n}^n\Big) = \sum_{i=1}^n (2k_i-|v_i|) + Q(\alpha,\alpha)\,.
$$
A vacuum vector is defined as $|0\rangle\coloneqq e^0\otimes 1\in V_0$. 

For $v\in \Lambda$ and $k>0$, the \textit{creation operator} $v_{(-k)}$ on $V_\bullet$ is defined as a left multiplication by the element $v_{-k}\in \BD_\Lambda$. This gives $V_\bullet$ the structure of a free $\BD_\Lambda$-module with basis $\{e^\alpha\}_{\alpha\in \Lambda_\sst}$. Thus, specifying an operator $A\colon V_\bullet\to V_\bullet$ is equivalent to describing its commutators with creation operators $[A, v_{(-k)}]$ and its action on the basis $A(e^\alpha)$. The operators that appear are often derivations of the $\BD_\Lambda$-module $V_\bullet$, and we will often define them by describing their action on the generators $v_{-k}$ of $\BD_\Lambda$ and on the basis $e^\alpha$. 

The translation operator $T$ is a $\BC$-linear even derivation of the $\BD_\Lambda$-module $V_{\bullet}$ determined by 
\begin{equation}
\label{Eq: translation}
T(v_{-k}) = kv_{-k-1}\,,\qquad T(e^{\alpha}) = e^\alpha\otimes \alpha_{-1}\,. 
\end{equation}
Note that when we write $\alpha_{-1}$ for $\alpha\in \Lambda_{\textup{sst}}$ we are implicitly using the map
\[\Lambda_{\textup{\sst}}\to \Lambda_{\textup{sst}}\otimes_\BZ\BC\hookrightarrow \Lambda_{\overline 0}\] to regard $\alpha$ as an element of $\Lambda_{\overline 0}$.
For $k\geq 0$, the \textit{annihilation operator} $v_{(k)}$ is defined as a derivation of the $\BD_{\Lambda}$-module $V_\bullet$ as we explain below, following Kac \cite[Sections 3.6 and 5.4]{Ka98}. The field corresponding to $v_{-1}$ is then obtained as a sum of creation and annihilation operators:
\begin{align*}
    \label{Eq: field}
    Y(v_{-1},z) &= \sum_{k\in \BZ}v_{(k)}z^{-k-1}\,.
    \numberthis
\end{align*}
We separate the construction of the fields into two cases. If our lattice $\Lambda$ is concentrated in even degree (i.e., $\Lambda_{\overline 1}=\{0\}$) we use Kac's bosonic construction and if it is concentrated in odd degree (i.e., $\Lambda_{\overline 0}=\{0\}$) we use a fermionic construction. In the general case, we take the tensor product of the two.
\begin{definition}
\label{Def: latticeVA}
    Suppose that 
    \begin{enumerate}
        \item     $\Lambda_{\overline{1}} = \{0\}$, then the resulting vertex algebra $V_{\ovZ,\bullet}=\BC[\Lambda_\sst]\otimes \BD_{\Lambda_{\overline 0}}$ is called a \textit{bosonic} lattice vertex algebra. In this case,  the action of the annihilation operators $\{v_{(k)}\}_{k\geq 0}$ for $v\in \Lambda_{\overline 0}$ is an even derivation on the $\BD_{\Lambda_{\ovZ}}$-module $V_{\ovZ,\bullet}$ determined by
$$
v_{(k)}(w_{-l}) = k\,Q(v,w)\,\delta_{k-l,0}\,,\qquad v_{(k)}(e^\alpha) = Q(v,\alpha)\,\delta_{k,0}\,e^\alpha\,, \quad k\geq 0, l>0 \,.
$$
Kac \cite[Section 5.5]{Ka98}\footnote{\label{footnote: kac}The construction in \cite[Section 5.4]{Ka98} is only for the case in which $\Lambda_\sst$ is a lattice (i.e., a free abelian group) and $\Lambda_\sst\otimes_\BZ \BC=\Lambda_{\overline 0}$, but everything works just fine in this slightly more general setting. This modification was already considered by Gross in \cite[Definition 3.5]{Gr19}. Note also that the case $\Lambda_\sst=\{0\}$ corresponds to what Kac calls the vertex algebra of free bosons, see \cite[Section 3.5]{Ka98}; this slightly more general construction can be thought of as interpolating the lattice vertex algebra and the bosonic vertex algebra.} endows $V_{\ovZ,\bullet}$ with a vertex algebra structure such that
\begin{align*}
    Y(v_{-1},z) &= \sum_{k\in \BZ}v_{(k)}z^{-k-1}\,,
\end{align*}
and 
\begin{align*}
   \label{Eq: field2}
 Y(e^{\alpha},z) &= (-1)^{q(\alpha, \beta)}z^{Q(\alpha,\beta)}e^{\alpha} \exp\Big[-\sum_{ k<0}\frac{\alpha_{(k)}}{k}z^{-k}\Big]\exp\Big[-\sum_{k>0}\frac{\alpha_{(k)}}{k}z^{-k}\Big]
 \numberthis
\end{align*}
on $e^\beta\otimes \BD_{\Lambda}$; in the formula, $e^\alpha$ stands for the operator sending $e^\beta\otimes w$ to $e^{\alpha+\beta}\otimes w$. Note that the signs $\varepsilon_{\alpha, \beta}=(-1)^{q(\alpha, \beta)}$ satisfy \cite[Equations 5.4.14]{Ka98}. The general state-field correspondence is set to be
     \begin{multline}\label{Eq: reconstruction}     
Y\big(e^\alpha\otimes v^1_{-k_1-1}\cdots v^n_{-k_n-1},z\big)\\
= \frac{1}{k_1!k_2!\cdots k_n !}:Y(e^\alpha,z)\frac{d^{k_1}}{(dz)^{k_1}}Y(v^1_{-1},z)\cdots \frac{d^{k_n}}{(dz)^{k_n}}Y(v^n_{-1},z): \,.
\end{multline}
\item  $\Lambda_{\ovZ} = \{0\}$, then the resulting vertex algebra $V_{\ovO, \bullet}=\BD_{\Lambda_{\overline 1}}$ is called a \textit{fermionic} vertex algebra, see \cite[Section 3.6]{Ka98}. It is determined uniquely by setting the annihilation operators $\{v_{(k)}\}_{k\geq 0}$, for $v\in \Lambda_{\overline 1}$, to be odd derivations on the supercommutative algebra $\BD_{\Lambda_{\overline 1}}$ such that 
$$
v_{(k)}(w_{-l}) =Q(v,w)\delta_{k-l+1,0}\,, \quad k\geq 0, l>0 \,.
$$
The remaining fields are obtained again by the reconstruction Theorem \cite[Theorem 4.5]{Ka98}, which states the analog of \eqref{Eq: reconstruction} without $e^\alpha$.
    \end{enumerate}
  In general, $V_\bullet$ is defined as the tensor product of the two lattice vertex algebras, \[V_\bullet=V_{\overline{0}, \bullet}\otimes V_{\ovO,\bullet} \]
   which are associated to $\Lambda_{\sst}\otimes_{\BZ}\BC\hookrightarrow \Lambda_{\overline 0}$ and $\Lambda_{\ovO}$, respectively. The resulting fields are determined uniquely by
$$
Y({a_{\ovZ}\otimes a_{\ovO}},z) = Y(a_{\ovZ},z)\otimes Y(a_{\ovO},z)
$$
whenever $a_{\overline{i}}\in V_{\,\overline{i},\,\bullet}$.
\end{definition}

We can summarize the above description by formulating the commutation relations of the operators $v_{(k)}$ as
\begin{equation}\label{lattice VOA commutation}
    [v_{(k)},w_{(\ell)}]=k^{(1-|v|)}Q(v,w)\delta_{k+\ell+|v|,0}\,,
\end{equation}
which holds for all $k,\ell\in \BZ$. If we want to obtain a closed formula for \eqref{Eq: field}, we may rewrite the annihilation operators in terms of derivatives to get
\begin{equation}
\label{Eq: explicitfield}
Y(v_{-1},z) = \sum_{k\geq0}v_{(-k-1)}z^{k} +\sum_{k\geq 1-|v|}\sum_{w\in B} k^{(1-|v|)} Q(v,w)\frac{\partial}{\partial w_{-k-|v|}}z^{-k-1} + Q(v,\beta)z^{-1}\,,
\end{equation}
on $e^\beta\otimes \BD_\Lambda\subseteq V_\bullet$, where $B\subset\Lambda$ is a basis. 

We will later identify Joyce's vertex algebra with a lattice vertex algebra. For that, the following proposition, which is a simple corollary of \cite[Proposition 5.4]{Ka98}, will be useful. 

\begin{proposition}
\label{Cor: calpha}
    Let $V_{\bullet}$ be a vertex algebra with the underlying vector space $\BC[\Lambda_{\sst}]\otimes \BD_{\Lambda}$ such that:
    \begin{enumerate}
    \item[(i)] The vacuum vector is $e^0\otimes 1$; 
    \item[(ii)] The fields $Y(e^0\otimes v_{-1}, z)$ are given by \eqref{Eq: explicitfield};
    \item[(iii)] We have
        \begin{equation}
    \label{Eq: calpha}
    [z^{Q(\alpha,\beta)}] Y(e^{\alpha},z)e^{\beta}=(-1)^{q(\alpha, \beta)}e^{\alpha+\beta}\,.
    \end{equation}
    \end{enumerate} 
Then $V_\bullet$ is isomorphic to the lattice vertex algebra of Theorem \ref{thm: voaconstruction}.
\end{proposition}
\begin{proof}
By \cite[Proposition 5.4]{Ka98}, conditions (1) and (2) imply that $V_\bullet$ is uniquely determined by a choice of operators $c_\alpha\colon V_\bullet\to V_\bullet$ for each $\alpha\in \Lambda_\sst$ satisfying 
\[c_0=\id\,, \quad c_\alpha\ket{0}=\ket{0}\, , \quad [v_{(k)}, c_\alpha]=0\textup{ for }v\in \Lambda, k\in \BZ.\]

For such a choice of operators, the field $Y(e^\alpha, z)$ is given by 
\begin{align*}
 Y(e^{\alpha},z) &=z^{Q(\alpha,\beta)}e^{\alpha} \exp\Big[-\sum_{ k<0}\frac{\alpha_{(k)}}{k}z^{-k}\Big]\exp\Big[-\sum_{k>0}\frac{\alpha_{(k)}}{k}z^{-k}\Big]c_\alpha
\end{align*}
To show that $V_\bullet$ is the lattice vertex algebra we need to show that $c_\alpha$ acts as $(-1)^{q(\alpha, \beta)}$ on $e^\beta\otimes \BD_\Lambda$. Since $c_\alpha$ commutes with creation operators, it is enough to show that $c_\alpha(e^\beta)=(-1)^{q(\alpha,\beta)}e^{\beta}$ for all $\beta\in \Lambda_\sst$.  We have
\begin{align*}
Y(e^\alpha, z)e^\beta&=z^{Q(\alpha,\beta)}e^{\alpha} \exp\Big[-\sum_{ k<0}\frac{\alpha_{(k)}}{k}z^{-k}\Big]\exp\Big[-\sum_{k>0}\frac{\alpha_{(k)}}{k}z^{-k}\Big]c_\alpha(e^\beta)\\
&=z^{Q(\alpha,\beta)}e^{\alpha} \exp\Big[-\sum_{ k<0}\frac{\alpha_{(k)}}{k}z^{-k}\Big]c_\alpha\exp\Big[-\sum_{k>0}\frac{\alpha_{(k)}}{k}z^{-k}\Big]e^\beta\\
&=z^{Q(\alpha,\beta)}e^{\alpha} \exp\Big[-\sum_{ k<0}\frac{\alpha_{(k)}}{k}z^{-k}\Big]c_\alpha e^\beta
\end{align*}
Extracting the coefficient of $z^{Q(\alpha, \beta)}$ from both sides and using (3) the result follows. \qedhere
\end{proof}

\subsection{Kac's conformal element}
\label{sec: Kacconformal}

Let $V_\bullet = \BC[\Lambda_\sst]\otimes \BD_\Lambda$ be a lattice vertex algebra associated to a non-degenerate symmetric pairing 
$$
Q:\Lambda\times \Lambda\longrightarrow \BZ.
$$
Recall that $V_\bullet$ is expressed as a tensor product  $V_\bullet = V_{\ovZ,\,\bullet}\otimes V_{\ovO,\,\bullet}$. This allows us to address the construction of the conformal element as
$$\omega=\omega_{\ovZ}+\omega_{\ovO},\quad \omega_{\overline{i}}\in V_{\overline{i},4}.
$$

Let us first fix a basis $B_{\,\ovZ}$ of $\Lambda_{\,\ovZ}$ and its dual $\hat{B}_{\,\ovZ}$ with respect to $Q$; we denote by $\hat{v}\in \hat{B}_{\,\ovZ}$ the dual of $v\in B_{\,\ovZ}$, so that
\[Q(v, \hat w)=\delta_{v,w}\quad\textup{for}\quad v, w\in B_{\overline 0}\,.\]Then the bosonic part has a natural choice of a conformal element given by
$$
\omega_{\ovZ} = e^0\otimes \frac{1}{2}\sum_{v\in B_{\ovZ}}\hat{v}_{-1} v_{-1} \in V_{\ovZ,4}\,.
$$
See \cite[Proposition 5.5]{Ka98}. The central charge of $\omega_{\overline 0}$ is $\dim(\Lambda_{\overline 0})$.

We now consider the fermionic part. Recall that we have in the assumptions of Theorem \ref{thm: voaconstruction} a splitting
\begin{equation}
\label{Eq: splittingisotropic}
\Lambda_{\ovO} = I\oplus \hat{I}\,.
\end{equation}
into maximal isotropic subspaces. Given such a splitting, Kac \cite[3.6.14]{Ka98} constructs a family of conformal elements $\omega^\lambda_{\ovO}$ parameterized by $\lambda\in \BC$. To give its explicit description, pick a basis $B_I\subset I$; then its dual basis is denoted by $B_{\hat{I}}\subset \hat{I}$ with elements $\hat{w}$ satisfying 
$Q(v,\hat{w}) = \delta_{v,w}\,.$ One then sets
$$
\omega^\lambda_{\ovO} = (1-\lambda)\sum_{v\in B_I}\hat{v}_{-2}v_{-1} + \lambda \sum_{v\in B_I}v_{-2}\hat{v}_{-1}\,.
$$
Notice that the expression is independent of a choice of bases $B_I,B_{\hat{I}}$ and swapping $I$ and $\hat{I}$ only interchanges $\lambda$ and $(1-\lambda)$. We set $\lambda=0$ and denote\footnote{We suppress the dependence of $\omega_{\ovO}$ on the choice of $I$ and $\hat{I}$ for simplicity of the notation.}
$$
\omega_{\ovO} = \sum_{v\in B_I}\hat{v}_{-2}v_{-1}\,.
$$
The central charge of $\omega_{\ovO}$ is 
\[2\,\mathrm{sdim}(I)=\mathrm{sdim}(\Lambda_{\overline 1})=-\dim(\Lambda_{\overline 1})\,,\]
by \cite[3.6.16]{Ka98} after plugging $\lambda=0$. Therefore,  the central charge of the full conformal element $\omega=\omega_{\ovZ}+\omega_{\ovO}$ is given by
$$\dim(\Lambda_{\overline{0}})-\dim(\Lambda_{\overline{1}})=\mathrm{sdim}(\Lambda)\,.
$$

\begin{remark}
    The choice of the conformal element $\omega$ is equivalent to the choice of a splitting 
 \eqref{Eq: splittingisotropic}, so we will often denote the latter piece of data also by $\omega$. 
\end{remark}

The corresponding conformal grading on $e^\alpha\otimes \BD_\Lambda\subset V_\bullet$, depending on the choice of the splitting $\omega$, is determined by the operator
\begin{align*}
L_0
&=[z^{-2}]\big\{Y(\omega,z)\big\}\\ 
&=[z^{-2}]\Big\{\frac{1}{2}\sum_{v\in B_{\ovZ}}:Y(\hat{v}_{-1},z)Y(v_{-1},z):
+\sum_{v\in B_I}:\frac{\partial}{\partial z}Y(\hat{v}_{-1},z)Y(v_{-1},z):\Big \}\\
&=\frac{1}{2}\sum_{\begin{subarray} a k\in\BZ\\
v\in B_{\ovZ}\end{subarray}}:\hat{v}_{(-k)}v_{(k)}:
+\sum_{\begin{subarray} a k\in\BZ\\
v\in B_I\end{subarray}}k:\hat{v}_{(-k-1)}v_{(k)}:
\\
&=\sum_{\begin{subarray} a k> 0\\
v\in B_{\ovZ}\end{subarray}}k\,v_{(-k)}\frac{\partial}{\partial v_{-k}}
+\sum_{\begin{subarray} a k> 0\\ v\in B_I\end{subarray}}
\Big[(k-1)\,\hat{v}_{(-k)}\frac{\partial}{\partial \hat{v}_{-k}}+k\,v_{(-k)}\frac{\partial}{\partial v_{-k}}\Big] 
+\frac{1}{2}Q(\alpha,\alpha)
\,.
\end{align*}
The last equality is obtained by separating the sum into $k>0$ or $k<0$ terms and using \eqref{Eq: explicitfield} in the form
\[v_{(k-1)}=\frac{\partial}{\partial \hat v_{-k}}\,,\quad \hat v_{(k-1)}=\frac{\partial}{\partial v_{-k}} \quad \textup{for}\quad v\in B_I,\ k> 0\,.\]
Specifically, the induced conformal grading is 
$$\deg_{\omega}(v_{-k})=\begin{cases}
k &\textup{if }v\in \Lambda_{\overline 0}\,,\\
k &\textup{if }v\in I\subseteq \Lambda_{\overline 1}\,,\\
k-1 &\textup{if }v\in \hat I\subseteq \Lambda_{\overline 1}\,,
\end{cases}
$$
for each $k>0$, while the $\BZ$-grading on $V_\bullet$ that we started with is given by 
$$\deg(v_{-k})=\begin{cases}
2k &\textup{if }v\in \Lambda_{\overline 0}\,,\\
2k-1 &\textup{if }v\in I\subseteq \Lambda_{\overline 1}\,,\\
2k-1 &\textup{if }v\in \hat I\subseteq \Lambda_{\overline 1}\,.
\end{cases}
$$

The conformal grading $V_\bullet^\omega$ is useful when it comes to studying Virasoro operators as we explain now. Starting from the notation for the decomposition  
$$v=v_I+v_{\hat{I}}\in \Lambda_{\overline{1}} \,,\quad \text{where}\quad  v_{I}\in I, v_{\hat{I}}\in\hat{I}\,,$$
we label the conformal shift
$$v^{\omega}_{-1}\coloneqq \begin{cases}
v_{-1}& \textnormal{if }v\in \Lambda_{\ovZ}\,,\\
(v_I)_{-1}+(v_{\hat{I}})_{-2}& \textnormal{if }v\in \Lambda_{\ovO}\,.
\end{cases}
$$
We also define a new pairing $Q^\omega(v,w)$ as 
$$Q^\omega(v,w)\coloneqq 
\begin{cases}
Q(v,w)&\textnormal{if }v,w\in \Lambda_{\ovZ}\,,\\
Q(v_I,w_{\hat{I}})-Q(v_{\hat{I}},w_I)&\textnormal{if }v,w\in \Lambda_{\ovO}\,,\\
0&\textnormal{otherwise.}
\end{cases}
$$
This pairing is non-degenerate and {\it supersymmetric} because $Q$ is non-degenerate and symmetric. Using the shift notation $v^\omega_{-1}$ and a new pairing $Q^\omega$, we can rewrite the conformal element as
\begin{align}\label{Eq: omega}
    \omega
    &=\frac{1}{2}\sum_{v\in B_{\ovZ}}\hat{v}_{-1}v_{-1}+\sum_{v\in B_I}\hat{v}_{-2}v_{-1}\\
    &=\frac{1}{2}\sum_{v\in B_{\ovZ}}\hat{v}_{-1}v_{-1}
    +\frac{1}{2}\sum_{v\in B_I}\hat{v}_{-2}v_{-1}-\frac{1}{2}\sum_{\hat{v}\in B_{\hat{I}}}v_{-1}\hat{v}_{-2}\nonumber\\
    &=\frac{1}{2}\sum_{v\in B} \tilde{v}^\omega_{-1}v^\omega_{-1}\nonumber\,.
\end{align}
In the last equality, we use the basis $B=B_{\ovZ}\sqcup B_I\sqcup B_{\hat{I}}$ of $\Lambda$ and $\{\tilde{v}\}$ denotes the dual basis to $B=\{v\}$ with respect to $Q^\omega$, such that $Q^\omega(v,\tilde{w})=\delta_{v,w}$. If we also shift the notation for the creation and annihilation operators defining $v^\omega_{(k)}$ by the formula
$$Y(v^\omega_{-1},z)\coloneqq \sum_{k\in \BZ}v^\omega_{(k)}\,z^{-k-1}\,,
$$
the Virasoro operators can be written as
\begin{equation}\label{eq: operatorsomegashift}
L_n=\frac{1}{2}\sum_{\substack{i+j=n\\ v\in B}} :\tilde{v}^\omega_{(i)}\,v^\omega_{(j)}:\,,\quad n\in \BZ\,.
\end{equation}

We also note that shifted creation and annihilation operators are subject to the bracket relations that are similar to bosonic fields 
\begin{equation}\label{eq: shifted bracket}
\left[v^{\omega}_{(n)},w^{\omega}_{(m)}\right]= n\, Q^{\omega}(v,w)\delta_{n+m,0}\,.
\end{equation}
This can be shown by dividing into three cases;
$$\textnormal{i) } v, w\in \Lambda_{\overline{0}}\,,\quad 
\textnormal{ii) } v\in I,\ w\in \hat{I}\,,\quad 
\textnormal{iii) } v\in \hat{I},\ w\in I\,,
$$
and using the translation formula \eqref{Eq: translationu}.

\subsection{Lie algebra of physical states}
\label{Sec: physical}

Borcherds \cite{Borcherds} associates to any vertex algebra $V_\bullet$ a graded Lie algebra
\begin{equation}\label{lie algebra}
\widecheck{V}_\bullet = V_{\bullet+2}\,/\,T(V_\bullet),\quad [\overline{a},\overline{b}]=\overline{a_{(0)}b}\,. 
\end{equation}
In this section, we study the Lie subalgebra of {\it primary} states (also known as {\it physical} states) that is closely related to Virasoro constraints. Algebraic statements in this section will translate into a compatibility between wall-crossing and Virasoro constraints for sheaves and pairs in the geometric setting. 

\begin{definition}[{\cite{Borcherds}, \cite[Corollary 4.10]{Ka98}}]
\label{Def: physicalstates}
Let $\big(V_\bullet,\omega\big)$ be a conformal vertex algebra. The space of primary states of conformal weight $i\in \BZ$ is defined as 
$$P_i\coloneqq \big\{a\in V^\omega_i\,\big|\, \ L_n(a)=0\ \textnormal{for all}\ n\geq 1\big\}\,.
$$
\end{definition}

The following assumption is used to construct the Lie subalgebra of primary states.
\begin{assumption}\label{standard assumption}
We have $\ker(T)\cap V^\omega_i=\{0\}$ for all $i<0$.
\end{assumption}
This assumption is satisfied for the lattice vertex algebras with Kac's conformal element because the kernel
$$\ker(T)=\textnormal{span}_{\BC}\big\{e^{\alpha}\otimes 1\,\big|\,\textnormal{torsion } \alpha\big\}
$$
consists only of elements with zero conformal weight.

Primary states yield a smaller Lie subalgebra by the proposition below. To state it, we introduce the Lie subalgebra 
\[\widecheck{V}_0^\omega\coloneqq V_1^\omega/T(V_0^\omega)\subseteq V_\bullet/T(V_\bullet)\,.\]
The fact that $\widecheck{V}_0^\omega$ is closed under the Lie bracket on $V_\bullet/T(V_\bullet)$ is a consequence of the fact $V_1^\omega$ is closed with respect to the $0$-th product.
\begin{proposition}[{\cite{Borcherds}}]\label{prop: wallcrossingcompatibility1}
Let $\big(V_\bullet,\omega\big)$ be a conformal vertex algebra satisfying Assumption \ref{standard assumption}. Then $\widecheck{P}_0\coloneqq P_1/T(P_0)$ defines a natural Lie subalgebra of $\widecheck{V}_0^\omega$. 
\end{proposition}
\begin{proof}
We record the proof to be self-contained. For any $a\in V_\bullet$ and $n\in \BZ$, we have 
\begin{align}\label{Ln and T}
    L_n(Ta)
    &=[L_n, T]a+T(L_n(a))\\
    &=(n+1)L_{n-1}(a)+T(L_n(a))\,.\nonumber
\end{align}
This implies that if $a\in P_0$ then $Ta\in P_1$, hence making sense of the quotient $\widecheck{P}_0=P_1/T(P_0)$. 

In order to show that the natural map $P_1/T(P_0)\rightarrow V_1^\omega/T(V_0^\omega)$ is injective, we need to prove that $a\in V_0^\omega$ with $Ta\in P_1$ implies $a\in P_0$. Suppose that $a\in V_0^\omega$ with $Ta\in P_1$. We use induction on $n$ to prove that $L_n(a)=0$. The base case $n=0$ follows from $a\in V_0^\omega$. If $L_{n-1}(a)=0$ for some $n\geq 1$ by induction hypothesis, then \eqref{Ln and T} implies $T(L_n(a))=0$ since $Ta\in P_1$. Since $L_n(a)\in \ker(T)$ has a negative conformal weight $-n$, Assumption \ref{standard assumption} implies $L_n(a)=0$.\footnote{The operator $L_n$ decreases the conformal grading by $n$ since $[L_0,L_n]=(-n)L_n$. } 

Finally, this defines a Lie subalgebra because if $a, b\in P_1$ then
\begin{align*}
    L_n(a_{(0)}b)=[L_n, a_{(0)}]b+u_{(0)}(L_n(b))=0\,,\quad n\geq 1
\end{align*}
hence $a_{(0)}b\in P_1$. Here we used the fact that if $a\in P_1$ then the operator $a_{(0)}$ commutes with any Virasoro operators $L_n$ \cite[Section 5]{Borcherds}.
\end{proof}

We give an alternative way to define a Lie subalgebra of primary states that is related to the weight 0 Virasoro operator. For this new definition, we need a partial lift of the Borcherds' Lie bracket. 

\begin{lemma}\label{partial lift}
There is a well-defined linear map\,\footnote{We abuse the Lie bracket notation that was originally used for $[-,-]:\widecheck{V}_i\times \widecheck{V}_j\rightarrow \widecheck{V}_{i+j}.$}
\begin{equation}
\label{Eq: liftofLie}[-,-]\colon \widecheck{V}_{i}\times V_{j}\rightarrow V_{i+j},\quad (\overline{a},b)\mapsto a_{(0)}b\,,
\end{equation}
which makes $V_\bullet$ a representation of a graded Lie algebra $\widecheck{V}_\bullet$. 
\end{lemma}
\begin{proof}
It suffices to check the factoring property of $a_{(0)}b$ in the first coordinate. This follows from \eqref{Eq: translationu} because $(Ta)_{(0)}=0$.
\end{proof}

\begin{proposition}\label{prop: wallcrossingcompatibility2}
Let $\big(V_\bullet,\omega\big)$ be a conformal vertex algebra satisfying Assumption \ref{standard assumption}. Then the linear map in Lemma \ref{partial lift} restricts to 
$$[-,-]\colon\widecheck{P}_0\times P_i\rightarrow P_i 
$$
which makes $P_i$ a subrepresentation of $V_i^\omega$ with respect to the Lie algebra $\widecheck{P}_0$. 
\end{proposition}
\begin{proof}
The proof is similar to that of Proposition \ref{prop: wallcrossingcompatibility1} relying on the fact that if $a\in P_1$ then the operator $a_{(0)}$ commutes with any Virasoro operators $L_n$.
\end{proof}

\begin{definition}
Let $\big(V_\bullet,\omega\big)$ be a conformal vertex algebra. We define a Lie subalgebra
$$\widecheck{K}_0\coloneqq \big\{\overline{a}\in \widecheck{V}^\omega_0\,\big|\,[\overline{a},\omega]=0
\big\}\subset\Big(\widecheck{V}^\omega_{0}\,,\, [-,-]\Big)\,.
$$
\end{definition}
In the definition above, the operator $[-,\omega]$ is a linear map defined in Lemma \ref{partial lift}. The fact that $\widecheck{K}_0$ defines a Lie subalgebra follows from the representation property in Lemma \ref{partial lift}. Connection of $\widecheck{K}_0$ to the formulation of the Virasoro constraints for moduli spaces of sheaves will be explained using the Lemma below; note the similarity between the weight 0 Virasoro operator introduced in Definition \ref{def: Linv} and the formula for $[-, \omega]$.
\begin{lemma}\label{WC with conformal element}
Let $\big(V_\bullet,\omega\big)$ be a conformal vertex algebra. Then we have 
$$[-,\omega]=\sum_{n\geq -1}\frac{(-1)^n}{(n+1)!}T^{n+1}\circ L_n:\widecheck{V}_0^\omega\rightarrow V_{2}^\omega\,. 
$$
\end{lemma}
\begin{proof}
By  \eqref{Eq: skewsymmetry}, we have 
\begin{align*}
    [\overline{a},\omega] = a_{(0)}\omega
   =\sum_{n\geq -1}\frac{(-1)^{n}}{(n+1)!}T^{n+1}\circ L_n(a)\,.\quad\qedhere
\end{align*}
\end{proof}

From Lemma \ref{WC with conformal element}, it is clear that $\widecheck{P}_0\subseteq \widecheck{K}_0$. The next lemma states the converse when working with lattice vertex algebras. 

\begin{proposition}\label{prop: primaryomegabracket}
Let $(V_\bullet, \omega)$ be a lattice vertex operator algebra as in Theorem \ref{thm: voaconstruction} with $V_\bullet=\BC[\Lambda_{\sst}]\otimes \BD_\Lambda$. Then we have $\widecheck{P}_0=\widecheck{K}_0$ on the summands $e^\alpha\otimes \BD_\Lambda$ such that $\alpha$ is not torsion.
\end{proposition}
\begin{proof}
Let $\overline{a}$ be an element admitting a lift $a\in e^\alpha\otimes \BD_\Lambda$ for some non torsion $\alpha$. Suppose that $\overline{a}\in \widecheck{K}_0$, i.e., $a_{(0)}\omega=0$. We show that $\overline{a}\in \widecheck{P}_0$ by constructing another lift of $\overline{a}$, not necessarily the same as $a$, that lies in $P_1$. Since the pairing $Q$ is non-degenerate, there exists some $b\in \Lambda_{\ovZ}$ such that $Q(\alpha,b)=1$. In the proof of this lemma, we denote $e^0\otimes b_{-1}\subset V_1^\omega$ simply by $b$; note that the creation/annihilation operators associated to $b\in \Lambda_{\ovZ}$ and the fields induced by $b=e^0\otimes b_{-1}\in V_\bullet$ are the same by definition, so the symbol $b_{(k)}$ is unambiguous. We claim that such $b$ provides a desired lift of $\overline{a}$ defined as
$$\eta_b(\overline{a})\coloneqq -a_{(0)}b\in V^\omega_1\,.
$$

Since \eqref{lie algebra} defines a Lie bracket, we know that 
$$
    \eta_b(\overline{a}) \in b_{(0)}a+T(V_\bullet)\,.
$$
Recall that the operator $b_{(0)}$ acts on $e^\alpha\otimes \BD_\Lambda$ as multiplication by $Q(\alpha,b)=1$. Therefore $\eta_b(\overline{a})$ is indeed another lift of $\overline{a}$. 

To show that $\eta_b(\overline{a})\in V^\omega_1$ is a primary state in $P_1$, we must show
$$\omega_{(n+1)}(a_{(0)}b)=0\quad  \textnormal{for any }  n\geq 1.
$$
This follows from the assumption $a_{(0)}\omega=0$ and the identity \eqref{Eq: skewjacobi}
\begin{align*}
    \omega_{(n+1)}(a_{(0)}b)
    =a_{(0)}(\omega_{(n+1)}b)-(a_{(0)}\omega)_{(n+1)}b\,,
\end{align*}
together with the basic fact that $b=e^0\otimes b_{-1}\in P_1$ \cite[page 81]{Ka98}. 
\end{proof}

\section{VOA from sheaf theory}\label{sec: VOA from sheaf theory}
The general treatment of vertex algebras in the previous section is now paralleled by their geometric construction formulated by Joyce \cite{Jo17}. Beginning from the application to the moduli stack of pairs of perfect complexes, we compare the resulting vertex algebras to lattice vertex algebras following the work of Gross \cite{Gr19}. The later parts of the section are focused on describing the duals of the operators $\mathsf{L}_k$ as Virasoro operators for a natural conformal element.

\subsection{Joyce's vertex algebra construction}\label{sec: Joycevertexalgebra}

We begin by describing the assumptions needed for the geometric construction of the vertex algebra for perfect complexes following Joyce's \cite{Jo17}\footnote{The manuscript \cite{Jo17} is not published and the author declares it as an incomplete draft. The main thing we need from there is the construction of the vertex algebra and the proof that it satisfies the vertex algebra axioms, i.e., Theorem \ref{thm: geometricvaconstruction}, for which the author provides full details.}. We then follow it up with how to extend it to pairs of complexes which is the most natural setting to work in for conformal elements and the rank reduction arguments later on.

We use the notation
$$
\pi_{J}\colon\prod_{i\in I}Z_i\longrightarrow \prod_{j\in J}Z_j
$$
for projections to components whenever $J\subset I$ are finite sets,
and denote for a $K$-theory class $\mathcal{K}$ on $\prod_{j\in J}Z_j$ the pullback by $\mathcal{K}_J =\pi^\ast_{J}(\mathcal{K})$. Pushforwards, pullbacks and duality below are all understood to be derived.
\begin{definition}
\label{def: VOAconstruction}
\begin{enumerate}
\item We work with a (higher) moduli stack $\MX$ of perfect complexes on $X$ constructed by Toën-Vaquié \cite{TV07} which admits a universal perfect complex $\CG$ on $ \MX\times X$.  The two structures we are interested in are the direct sum map 
$\Sigma\colon\MX\times \MX\to \MX
\,,$
such that 
$$
(\Sigma\times\id_X)^\ast \CG = \CG_{1,3}\oplus \CG_{2,3}\,,
$$
and an action $\rho\colon B\BG_m\times\MX\to \MX$ determined by
$$
(\rho\times \id_X)^\ast(\CG) = \CQ_1\otimes \CG_{2,3}
$$
for the universal line bundle $\CQ$ on $B\BG_m$.
   \item The second major ingredient in constructing vertex algebras is a complex
    \begin{equation}
    \label{Eq: Extcomplex}
    \Ext = (\pi_{1,2})_\ast\Big(\CG_{1,3}^\vee\otimes \CG_{2,3}\Big)\,,    \end{equation}
   on $\CM_X\times \CM_X$. Denoting by  $\sigma\colon \MX\times \MX\to \MX\times \MX$ the map swapping the factors we construct its symmetrization
    $$
    \Theta = \Ext^\vee\oplus \sigma^\ast\Ext
    $$
    which satisfies
    $
   \sigma^* \Theta\cong \Theta^\vee.
    $
    \item For any topological type $\alpha\in K_\sst^0(X)\simeq \pi_0(\mathcal{M}_X)$, we denote the corresponding connected component by $\Malpha\subseteq \CM_X$. Any restriction of an object living on $\MX$ to $\Malpha$ will be labelled by adjoining the subscript $(-)_{\alpha}$. This allows us to express
    $$
   \chi_{\sym}(\alpha,\beta)\coloneqq \rk\big(\Theta_{\alpha,\beta}\big) = \chi(\alpha,\beta) + \chi(\beta,\alpha)
    $$
    for $\chi\colon K_\sst^0(X)\times K_\sst^0(X)\to \BZ$ the usual Euler form. 
 \end{enumerate}
\end{definition}
By the construction of Blanc \cite[Section 3.4]{B16}, there exists a topological space $\CS^{\text{t}}$ defined up to homotopy which is assigned to each (higher) stack $\CS$. The homology and cohomology of $\CS$ are then defined by 
$$
H_\bullet(\CS) \coloneqq H_\bullet(\CS^{\text{t}})\, ,\qquad H^\bullet(\CS)\coloneqq H^\bullet(\CS^{\text{t}})\,.
$$
Each perfect complex $\CE$ on $\CS$ corresponds to a map 
$
\CS\xrightarrow{\CE} \text{Perf}_{\BC}
$ which induces $\CS^{\text{t}}\xrightarrow{\CE^\text{t}}BU\times \BZ$ and the $K$-theory class
$$
\llbracket \CE \rrbracket= (\CE^{\text{t}})^*(\CU) \in K^0(\CS^{\text{t}})\,.
$$
The Chern classes of $\CE$ are defined to be the Chern classes of $\llbracket \CE\rrbracket$.
\begin{remark}\label{rem: classD}
  We will only work with geometries where the natural map
    $$
    K^i_{\sst}(X, \BZ)\longrightarrow K^i(X, \BZ)
    $$
    for the semi-topological $K$-theory $K^i_{\sst}(X, \BZ)$ in \cite{FW} is an isomorphism for all $i>0$ and an injection when $i=0$. Gross \cite{Gr19} calls the class of such varieties class D; it includes curves, surfaces, rational 3-folds and rational 4-folds. From \cite[Theorem 4.21]{B16} it follows that for such varieties
    $$
    \pi_i(\MX^\text{t}) = K^i_{\sst}(X,\BZ)=\begin{cases}K^i(X, \BZ)&\textup{for }i>0\\
    K_\sst^0(X)&\textup{for }i=0\,.\end{cases}
    $$
\end{remark}
The vertex algebra on the shifted homology 
\begin{equation}
\label{eq: homologyvertex}
V_\bullet = \bigoplus_{\alpha\in K_\sst^0(X)}\widehat{H}_\bullet(\Malpha) 
\end{equation}
for the shift 
$$
\widehat{H}_\bullet\big(\Malpha\big)= H_{\bullet-2\chi(\alpha,\alpha)}\big(\Malpha,\BC\big)
$$
was constructed by \cite{Jo17} as follows:
\begin{theorem}[{\cite[Theorem 3.12]{Jo17}}]\label{thm: geometricvaconstruction}
There is a vertex algebra structure on $V_\bullet$ from \eqref{eq: homologyvertex} defined by 
\begin{enumerate}
    \item letting $0\colon*\to \CM_0$ be the inclusion of the zero object and setting 
    $$
    \ket{0}=0_*(*)\in H_0(\CM_0)\,.
    $$
    \item taking $t\in H_2\big(B\BG_m\big)$ to be the dual of $c_1(\CQ)\in H^2\big(B\BG_m\big)$ and setting 
    $$
    T(u) = \rho_\ast\big(t\boxtimes u\big)\,.
    $$
    \item constructing the state-field correspondence by the formula 
    \begin{equation}\label{eq: joycefields}
    Y(u,z)v = (-1)^{\chi(\alpha,\beta)}z^{\chi_\sym(\alpha,\beta)}\Sigma_\ast\Big[e^{zT}\boxtimes \text{id}\big(c_{z^{-1}}(\Theta)\cap (u\boxtimes v)\big)\Big]
    \end{equation}
for any $u\in \widehat{H}_\bullet(\mathcal{M}_\alpha)$ and $v\in \widehat{H}_\bullet(\mathcal{M}_{\beta})$.
\end{enumerate}
\end{theorem}

While moduli spaces of sheaves (recall Section \ref{sec: modulisheavespairs}) naturally define an element in the vertex algebra $V_\bullet$ (or the associated Lie algebra $\widecheck V_\bullet$), this vertex algebra is not suitable to study wall-crossing in moduli spaces of pairs. Essentially this is due to the fact that the complex $\Ext$ in Definition \ref{def: VOAconstruction} captures the deformation theory of sheaves but not of pairs. 

Thus, when working with pairs we will work with a larger vertex algebra $V_\bullet^{\pa}$ that is constructed from the homology of a stack parametrizing a pair of complexes and replaces the class $\Ext$ by $\Ext^{\pa}$, more related to the deformation theory of pairs. It also turns out that the vertex algebra $V_\bullet^{\pa}$ is the adequate place to construct a conformal element that produces the Virasoro operators.

\begin{definition}
\label{def: pairVOA}
\begin{enumerate}
\item Let $\CP_X\coloneqq\CM_X\times \CM_X$ and denote by $\CV=\CG_{1,3}$ and $\CF=\CG_{2,3}$ the pullbacks of $\CG$ via the two possible projections $\CM_X\times \CM_X\times X\to \CM_X\times X$. This stack has a natural direct sum map $\Sigma^{\pa}$ and a $B\BG_m$--action~$\rho^{\pa}$
    \begin{align*}
\Sigma^{\pa}&\coloneqq \big(\Sigma\times \Sigma\big)\circ \sigma_{2,3} : \CP_X\times \CP_X\longrightarrow \CP_X\,,\\
\rho^{\pa}&\coloneqq \big(\rho\times\rho\big) \circ\big(\Delta_{B\BG_m}\times \id_{\CP_X}\big): B\BG_m\times \CP_X\longrightarrow \CP_X\,.
\end{align*}
where $\sigma_{2,3}$ swaps the second and third copy of $\CM_X$ in $\CP_X\times \CP_X=\CM_X^{\times 4}$. 
\item One extends $\Ext$ to a perfect complex 
    \begin{equation}
    \label{Eq: Extpacomplex}
    \Ext^{\pa} = (\pi_{1,2})_\ast\Big((\CF_{1,3}\oplus \CV_{1,3}[1])^\vee\otimes \CF_{2,3}\Big)\,
    \end{equation}
on $\CP_X\times \CP_X$.
We introduce its symmetrization
$$
\Theta^{\pa} = (\Ext^{\pa})^\vee \oplus (\sigma^{\pa})^\ast\Ext^{\pa}\,,
$$
where $\sigma^{\pa}:\CP_X\times \CP_X\to \CP_X\times \CP_X$ swaps the two factors.

\item We now work with two copies of the $K$-theory  with the connected components labeled by $\CP_{\alpha^{\pa}}$ whenever $\alpha^{\pa}=(\alpha_1,\alpha_ 2)\in K_\sst^0(X)^{\oplus 2}$ and continue to denote by the subscript $(-)_{\alpha^{\pa}}$ the restriction of some object to each connected component. We  define the pairing 
\begin{align*}
\label{Eq: pairing}
\chi^{\pa}\big(\alpha^{\pa},\beta^{\pa}\big) \coloneqq \chi(\alpha_2 -\alpha_1,\beta_2) =\rk\big(\Ext^{\pa}_{\alpha^{\pa},\beta^{\pa}}\big) \,.
\end{align*}
and its symmetrization
\begin{equation*}\label{Eq: pairing}
\chi^{\pa}_{\sym}(\alpha^{\pa},\beta^{\pa}\big) \coloneqq \chi(\alpha^{\pa},\beta^{\pa}\big)+ \chi(\beta^{\pa}, \alpha^{\pa}\big)=\rk\big(\Theta^{\pa}_{\alpha^{\pa},\beta^{\pa}}\big) \,.
\end{equation*}
\end{enumerate}
\end{definition}

The pair vertex algebra has the underlying graded vector space 
$$
V^{\pa}_\bullet =\widehat {H}_\bullet(\CP_{X})= \bigoplus_{\alpha^{\pa}}\widehat{H}_{\bullet}\big(\CP_{\alpha^{\pa}}\big)\,,
$$
where $\widehat{H}_\bullet\big(\CP_{\alpha^{\pa}}\big) = H_{\bullet-2\chi^{\pa}(\alpha^{\pa}, \alpha^{\pa})}\big(\CP_{\alpha^{\pa}}\big)$. The structure of a vertex algebra is constructed exactly as in Theorem \ref{thm: geometricvaconstruction}. In particular, the state-field correspondence is given by
    \begin{equation}\label{eq: joycefieldspa}
    Y(u,z)v = (-1)^{\chi^{\pa}(\alpha^{\pa}, \beta^{\pa})}z^{\chi_\sym^{\pa}(\alpha^{\pa},\beta^{\pa})}\Sigma^{\pa}_\ast\Big[e^{zT}\otimes \text{id}\big(c_{z^{-1}}(\Theta^{\pa})\cap u\boxtimes v\big)\Big]
    \end{equation}
for $u\in \widehat{H}_\bullet\big(\CP_{\alpha^{\pa}}\big)$, $v\in \widehat{H}_\bullet\big(\CP_{\beta^{\pa}}\big)$. 
Note that the inclusion of $\CM_X\hookrightarrow \CP_X$ sending $F\mapsto (0, F)$ realizes $V_\bullet$ as a vertex subalgebra of $V_\bullet^{\pa}$.

\begin{remark}\label{rem: stackNX}
In Joyce's theory \cite[Section 8]{Jo21}, wall-crossing for pairs plays an important role in the definition of the invariant classes $[M]^\inva$. However, the vertex algebra that he uses to formulate such wall-crossing formulae is not $V_\bullet^{\pa}$. For our purposes, it will be enough to consider the stack $\CN_X$ parametrizing triples $(U,F,f)$ where $U$ is a vector space, $F$ is a sheaf and $f:U \otimes \CO_X\to F$ is a morphism (cf. \cite[Definition 8.2]{Jo21} with $L=\CO_X$ or \cite[Definition 2.12]{Bo21} when $n=0$). The stack $\CN_X$ maps to $\CP_X$ by   
\[\Omega:(U\otimes \CO_X\to F)\mapsto (U\otimes \CO_X,  F).\]
If one replaces $\Ext^{\pa}$ by
$$
\Ext^{\text{vpa}} = (\pi_{1,2})_\ast\Big((\CF_{1,3}\oplus \CV_{1,3}[1])^\vee\otimes \CF_{2,3}\Big)\oplus \CV_{1,3}^\vee\otimes \CV_{2,3}
$$
in the definition of \eqref{eq: joycefieldspa},
then the vertex algebra structure on $\CN_X$ used by Joyce makes the pushforward $\Omega_*$ in homology into a vertex algebra homomorphism. 

When doing wall-crossing in Section \ref{sec: rankreduction} and in the Appendix \ref{app:A}, we only work with $Y(u,z)v$ where
$$
u\in V^{\pa}_{\bullet}\quad \text{and}\quad v \in V_{\bullet}\subset V^{\pa}_{\bullet}\,.
$$
The restrictions of $\Ext^{\pa}$ and $\Ext^{\text{vpa}}$ to $\CP_X\times \CM_X\subset \CP_X\times \CP_X$ coincide, so the resulting wall-crossing formulae remain the same. This was discussed in more detail for quivers in \cite[Remark 3.5 and 3.12, Lemma 4.4]{BoVir} where the comparison of vertex algebras was stated explicitly. 
\end{remark}

\subsection{Joyce's vertex algebra as a lattice vertex algebra}
\label{sec: grossiso}

We will now give an explicit description of $V_\bullet$ and $V^{\pa}_\bullet$ as lattice vertex algebras (see Theorem \ref{thm: voaconstruction}) following Gross' work \cite{Gr19}.  
 
Let $\Lambda$ be the $\BZ_2$-graded vector space
\[\Lambda=\Lambda_{\overline{0}}\oplus \Lambda_{\overline{1}}=K^0(X)\oplus K^1(X)=K^\bullet(X).\]
Let also $\Lambda_\sst=K_\sst^0(X)$. Next we need to describe the bilinear forms $Q, q$ on $\Lambda$ extending the natural Euler pairing on $\Lambda_\sst=K_\sst^0(X)$. Recall that we have a natural Chern character isomorphism $K^\bullet(X)\cong H^\bullet(X)$.  Define the dual $(-)^\vee\colon K^\bullet(X)\to K^\bullet(X)$ by identifying with $H^\bullet(X)$ and setting
$$
\ch(v^\vee) = (-1)^{\floor*{ \frac{\deg\ch(v)}{2}}}\ch(v)\,,
$$
where $\deg\ch(v)$ is the cohomological degree. This leads to \begin{equation}
\label{Eq: splittingduals}
(v\cdot w)^\vee = (-1)^{|v||w|}v^\vee\cdot w^\vee\,.
\end{equation}
We define the extension of the Euler pairing to $K^\bullet(X)$ as
$$
\chi(v,w) = \int_X \ch(v^\vee)\cdot \ch(w)\cdot \td(X)\,,\qquad v,w\in K^\bullet(X)\,.
$$
We will denote its symmetrization by $\chi_\sym$:
\[\chi_\sym(v,w)=\chi(v,w)+\chi(w,v)\,.\]

For the pair version, we let 
\[\Lambda^{\pa}_{\overline{i}} = \Lambda_{\overline{i}}^{\oplus 2}\,,\quad \Lambda^{\pa} =\Lambda^{\oplus 2}=\Lambda^{\pa}_{\ovZ}\oplus \Lambda^{\pa}_{\ovO}\,,\quad \Lambda^{\pa}_\sst=\Lambda_\sst^{\oplus 2}.\]
Given $v\in \Lambda$ we will denote by $v^\CV=(v,0),\, v^\CF=(0,v)\in \Lambda^{\pa}$ the corresponding elements in the first and second copies of $\Lambda$, respectively. Given two elements 
\[v^{\pa}=(v_1, v_2)=v^\CV_1 + v^\CF_2\,,\quad w^{\pa}=(w_1, w_2)=w^\CV_1 + w^\CF_2\]
their pairing is defined as
\begin{equation}
\label{Eq: chipa}
\chi^{\pa}\big(v^{\pa},w^{\pa}\big) =  \chi(v_2-v_1,w_2) \,,
\end{equation}
and as usual its symmetrization is
\begin{equation}
\label{Eq: chipasym}
\chi^{\pa}_\sym\big(v^{\pa},w^{\pa}\big) = \chi^{\pa}\big(v^{\pa},w^{\pa}\big)+\chi^{\pa}\big(w^{\pa},v^{\pa}\big) \,.
\end{equation}
Note that the forms $\chi^{\pa}$ and $\chi^{\pa}_{\sym}$ extend the ones in Definition \ref{def: pairVOA}.(3).

\begin{lemma}
\label{Lem: nondegenerate}
The form $\chi^{\pa}_{\sym}$ is non-degenerate.
\end{lemma}
\begin{proof}
Using the decomposition $\Lambda^{\pa}=\Lambda\oplus \Lambda$ the symmetric form $\chi^{\pa}_{\sym}$ can be represented by the block matrix
\[\chi^{\pa}_{\sym}=\begin{bmatrix}
0&-\chi\\
-\chi& \chi_\sym
\end{bmatrix}.\]
Since clearly $\chi$ is non-degenerate it follows that $\chi^{\pa}_{\sym}$ is non-degenerate as well.
\end{proof}

The following is a necessary modification of \cite[Theorem 5.7]{Gr19} which we write out in full detail to avoid imprecisions and to include the analog statement for the pair vertex algebra. We point out that while all the ideas already appear in loc. cit. the computation of the odd degree fields is not present there and their description has a degree shift to the correct one. Furthermore, unlike the Definition of the generalized lattice vertex algebra in \cite{Gr19}, we do not rely on the construction of Abe \cite{Abe07}.

\begin{theorem}\label{thm: gross}
Let $X$ be a variety in class D (cf. Remark \ref{rem: classD}).
Then we have isomorphisms of vertex algebras
\begin{align}
\label{Eq: sheafVOA}
V_\bullet &\cong \BC\big[\Lambda_\sst\big]\otimes \BD_{\Lambda}\,,\\
\label{Eq: pairVOA}
V^{\pa}_\bullet &\cong \BC\big[\Lambda^{\pa}_{\sst}\big]\otimes \BD_{\Lambda^{\pa}}\,,
\end{align}
where: the left hand sides are the vertex algebras from Joyce's geometrical construction with the data of Definitions \ref{def: VOAconstruction} and \ref{def: pairVOA}, respectively; the right hand sides are the lattice vertex algebras from Theorem \ref{thm: voaconstruction} and the symmetric bilinear forms $q=\chi, Q=\chi_\sym$ and $q=\chi^{\pa}, Q=\chi_\sym^{\pa}$, respectively.
\end{theorem}

We begin the proof of the Theorem by explaining how to identify both sides as graded vector spaces. Using the universal sheaf $\CG_{\alpha}$ on $\CM_{\alpha}\times X$ we have for each $\alpha\in K_\sst^0(X)$ a geometric realization morphism \eqref{eq: geometricrealizationalpha}
\[\xi_{\CG_\alpha}\colon \BD^X_\alpha\to H^\bullet(\CM_\alpha).\]
 
\begin{lemma}[{\cite[Theorem 4.15]{Gr19}}]
\label{Lem: isomorphism}
Let $X$ be a variety in class D. Then the map $\xi_{\CG_{\alpha}}$ is an isomorphism, 
\[H^\bullet(\CM_\alpha)\cong \BD^X_\alpha.\]
Similarly,
\[H^\bullet(\CP_{\alpha^{\pa}})\cong \BD^{X, \pa}_{\alpha^{\pa}}\coloneqq \BD^X_{\alpha_1}\otimes \BD^X_{\alpha_2}.\]
\end{lemma}
\begin{proof}
Gross shows that $H^\bullet(\CM_\alpha)$ is freely generated by the Kunneth components of $\ch_k(\CG_\alpha)\in H^\bullet(\CM_\alpha\times X).$ But these are precisely the geometric realization of descendents, see \eqref{eq: chkunnethdescendents}. The result for pairs follows from the sheaf version.\qedhere
\end{proof}

We know from the previous lemma that
$$
H^\bullet(\CM_\alpha) \cong \BD^X_\alpha =\SSym\llbracket \tCH^X_\alpha\rrbracket \,,
$$
and we define the pairing $
\langle -,-\rangle \colon \text{CH}^X_{\alpha}\times \text{CH}_{\Lambda} \longrightarrow \BC
$
given by \begin{equation}
\label{Eq: capproductchk}
\Big\langle\ch_{k}(\gamma),v_{-j}\Big\rangle  =\int_X\gamma \cdot \ch(v) \frac{\delta_{k,j}}{(k-1)!}\,.
\end{equation}
The reader can recall the definition of $\tCH^X_\alpha$ and $\tCH_\Lambda$ in Definition \ref{def: topdescendentalgebra} and Theorem~\ref{thm: voaconstruction}, respectively. The pairing above is a perfect pairing, so it identifies the dual of the graded vector space $\tCH^X_\alpha$ with $\tCH_\Lambda$; we will use $\dagger$ to denote duals. Recalling the discussion from Section \ref{sec: supercommutative}, we then get an identification between
\[H_\bullet(\CM_\alpha)=H^\bullet(\CM_\alpha)^\dagger=
(\BD_\alpha^X)^\dagger=\SSym\llbracket\tCH_\alpha^X\rrbracket ^\dagger=\SSym[\tCH_\Lambda]=\BD_\Lambda\,.\]
Moreover, by Definition \ref{Def: pairingofV} the pairing between $\BD_\alpha^X$ and $\BD_\Lambda$ can be promoted to a cap product
$$\cap:\BD^X_{\alpha}\times \BD_{\Lambda}\to \BD_{\Lambda}$$
such that for a basis $B$ of $K^\bullet(X)$, we have
\begin{equation}\label{eq: capdescendents}
\ch_{k}(\gamma)\cap(-)=\frac{1}{(k-1)!}\sum_{w\in B}\int_X\gamma\cdot \ch(w)\frac{\partial}{\partial w_{-k}}\,.
\end{equation}

The isomorphisms $\BD_{\Lambda}\cong H_\bullet(\CM_\alpha)$ and $\BD_{\alpha}^X\cong H^\bullet(\CM_\alpha)$ identify this abstract cap product with the topological cap product, i.e.,
\begin{equation}
\label{Eq: homologyisomorphism}
\begin{tikzcd}
\BD^X_{\alpha}\times \arrow[d,"{\xi_{\CG_{\alpha}}\times \xi_{\CG_{\alpha}}^{-\dagger}}"]\BD_{\Lambda}\arrow[r,"\cap"]&\BD_{\Lambda}\arrow[d,"{\xi_{\CG_{\alpha}}^{-\dagger}}"]\\
H^\bullet(\CM_{\alpha})\times H_\bullet(\CM_{\alpha})\arrow[r,"\cap"]&H_\bullet(\CM_{\alpha})
\end{tikzcd}
\end{equation}
commutes where $\xi^{-\dagger}$ denotes the inverse of the isomorphism $\xi^{\dagger}$. 

By assembling all the isomorphisms of graded vector spaces $\widehat H_\bullet(\CM_\alpha)\cong e^\alpha \BD_\Lambda$ we get
\begin{equation}\label{eq: isovectorspaceva}\widehat H_\bullet(\CM_X)=\bigoplus_{\alpha\in K_\sst^0(X)}\widehat H_\bullet(\CM_\alpha)\cong \bigoplus_{\alpha\in \Lambda_\sst}e^\alpha \BD_\Lambda=\BC[\Lambda_\sst]\otimes \BD_\Lambda
\end{equation}
and similarly 
\begin{equation}\label{eq: isovectorspacevapa}
\widehat H_\bullet(\CP_X)\cong \BC[\Lambda_\sst^{\pa}]\otimes \BD_{\Lambda^{\pa}}.
\end{equation}
To prove Theorem \ref{thm: gross} it remains to show that the vertex algebra that Joyce defined on the left hand side and the lattice vertex algebra on the right hand side are compatible under this isomorphism. We will do this using Proposition \ref{Cor: calpha}. 

Before we analyze the fields required to show that the two vertex algebra structures are compatible, it will be useful to identify the translation operators on both sides. Recall Definition \ref{def: invariantdescendents} of the operator $\mathsf{R}_{-1}: \BD^X_\alpha\to \BD^X_\alpha$.

\begin{lemma}\label{lem: translationdualr-1}
Under the identification of $\BD^X_\alpha$ with $H^\bullet(\CM_\alpha)\cong H_\bullet(\CM_\alpha) ^\dagger$, the translation operator $T\colon H_\bullet(\CM_\alpha)\to H_{\bullet+2}(\CM_\alpha)$ (defined in Theorem \ref{thm: geometricvaconstruction}) is dual to $\bR_{-1}$. Moreover, the isomorphisms \eqref{eq: isovectorspaceva}, \eqref{eq: isovectorspacevapa} preserve the respective translation operators on both sides.
\end{lemma}
\begin{proof}
We recall the operator 
\[\bE=e^{\zeta \bR_{-1}}\] from Lemma \ref{lem: changeuniversalsheaf}. We claim that the following diagram commutes:
\begin{center}
\begin{tikzcd}
\BD^X_\alpha \arrow[r, "\bE"] \arrow[d, "\xi_{\CG_\alpha}"]\arrow[rd, "\xi_{\CQ \boxtimes \CG_\alpha}"]&
\BD^X_\alpha\llbracket \zeta\rrbracket \arrow[rd, "\xi_{\CG_\alpha}"]&
\,
\\
H^\bullet(\CM_\alpha)\arrow[r, "\rho^\ast"]&
H^\bullet(B\BG_m\times \CM_\alpha)\arrow[r, "\sim"]&
H^\bullet(\CM_\alpha)\llbracket \zeta\rrbracket
\end{tikzcd}
\end{center}
The triangle on the left commutes because $\rho^*(\CG_\alpha) = \CQ\boxtimes \CG_\alpha$. The parallelogram on the right commutes by Lemma \ref{lem: changeuniversalsheaf}. It follows that, under the identification $\BD^X_\alpha\cong H^\bullet(\CM_\alpha)$, $\bE=\rho^\ast$ and thus $\bR_{-1}=(t\backslash-)\circ \rho^\ast$ where the $(t\backslash -)$ stands for the slant product with the generator $t\in H_2(B\BG_m)$ dual to $\zeta=c_1(\CQ)\in H^2(B\BG_m)$. This is clearly the dual to Joyce's translation operator defined by $T=\rho_\ast\circ (t\boxtimes -).$

To show that Joyce's translation operator agrees with the one constructed in the lattice vertex algebra it is enough to check equations \eqref{Eq: translation}. This is straightforward with the identification of $T$ with the dual of $\bR_{-1}$. 
\end{proof}

Before we reproduce the proof of Theorem \ref{thm: gross}, we introduce some notation. From now on we will omit the isomorphism $\xi_{\CG_\alpha}$ and simply write $\ch_i(\gamma)\in H^\bullet(\CM_\alpha)$. We also set $\ch_0(\gamma)$ to be $\int_X \gamma\cdot \ch(\alpha)\in H^0(\CM_\alpha)$. We will use the notation $\ch_i^\CV(\gamma)=\ch_i(\gamma)\otimes 1$ and $\ch_i^\CF(\gamma)=1\otimes \ch_i(\gamma)$ for the generators of $\BD^{X, {\pa}}_{\alpha^{\pa}}=\BD^X_{\alpha_1}\otimes \BD^X_{\alpha_2}$, as we did in Section \ref{sec: virasoropairs}. We also use the same notation for the images in $H^\bullet(\CP_{\alpha^{\pa}})$. Given $\gamma\in H^\bullet(X)$ we let $\gamma^\CV=(\gamma, 0), \gamma^{\CF}=(0, \gamma)\in H^\bullet(X)^{\oplus 2}$. Given $\gamma^{\pa}\in H^\bullet(X)^{\oplus 2}$ we introduce the symbol $\ch_i(\gamma^{\pa})\in \BD^{X, {\pa}}_{\alpha^{\pa}}$ so that $\ch_i(\gamma^{\CV})=\ch_i^{\CV}(\gamma)$, $\ch_i(\gamma^{\CF})=\ch_i^{\CF}(\gamma)$. Given $\gamma^{\pa}\in H^\bullet(X)^{\oplus 2}$ and $w^{\pa}\in K^\bullet(X)^{\oplus 2}$ we define the pairing
\[\langle \gamma^{\pa}, w^{\pa} \rangle=\int_X \gamma_1\cdot \ch(w_1)+\int_X \gamma_2\cdot \ch(w_2).\]
If $\{w^{\pa}\}$ is a basis of $K^\bullet(X)^{\oplus 2}$ we then have the analogue of \eqref{eq: capdescendents} for pairs:
\begin{equation}
\label{eq: capdescendentspairs}
\ch_{k}(\gamma^{\pa})\cap(-)=\frac{1}{(k-1)!}\sum_{w^{\pa}}\langle \gamma^{\pa}, w^{\pa}\rangle\frac{\partial}{\partial w^{\pa}_{-k}}\,.
\end{equation}

\begin{proof}[Proof of Theorem \ref{thm: gross}]
We will only give the proof of the statement for pairs, since the sheaf version is easier and a direct consequence. Since we have already identified the underlying vector spaces, by Proposition \ref{Cor: calpha} it is enough to check conditions (i)-(iii) for the vertex algebra defined by Joyce. The vacuum condition (i) is immediate. Before we compute the fields necessary in (ii), (iii) we will obtain a formula for the Chern classes of $\Theta^{\pa}$. Fix a basis $\{\gamma\}\subseteq H^\bullet(X)$ and let $\{ \overline{\gamma}\}\subseteq H^\bullet(X)$ be the dual basis, such that $\int_X \gamma_1\cdot\overline{\gamma_2}=\delta_{\gamma_1,\gamma_2}$. Then the class of the diagonal in $X\times X$ is
\[\Delta_\ast(1)=\sum_{\gamma}\overline\gamma\otimes \gamma\in H^\bullet(X\times X)\,,\]
where the sum is over the basis we fixed. We then have
\begin{equation}\label{eq: chkunnethdescendents}\ch(\CF)  = (\pi_{1,2})_\ast\big((\id_{\CP_X}\times \Delta)_* (\ch(\CF)\big) = \sum_{i\geq 0}\ch^{\CF}_i(\overline{\gamma})\boxtimes\gamma 
\end{equation}
and an analogous formula for $\CV$. The term $\ch_{i}^\CF(\overline{\gamma})\otimes \gamma$ contributes to $\ch_k(\CF)$ for \[k=\frac{2i-|\overline \gamma|+\deg \gamma}{2}=i+\left\lfloor\frac{\deg(\gamma)}{2}\right \rfloor.\]
Thus,
\[\ch(\CF^\vee)=\sum_{\gamma}\sum_{i\geq 0} (-1)^{i+\left\lfloor\frac{\deg \gamma}{2}\right \rfloor}\ch_{i}^\CF(\overline{\gamma})\otimes \gamma\,.\]
 Using Grothendieck-Riemann-Roch as in \cite[Proposition 5.2]{Gr19} and being careful with signs obtained from commuting odd variables we find that
\begin{align}\nonumber 
\ch(\Ext^{\pa}) &=\sum_{\gamma, \delta}\sum_{i, j\geq 0}(-1)^{i+|\gamma|}\chi(\gamma, \delta)\ch_i^{\CF-\CV}(\overline \gamma)\boxtimes\ch_j^{\CF}(\overline \delta)\\
&=\sum_{\gamma^{\pa}, \delta^{\pa}}\sum_{i, j\geq 0}(-1)^{i+|\gamma^{\pa}|}\chi^{\pa}(\gamma^{\pa}, \delta^{\pa})\ch_i(\overline \gamma^{\pa})\boxtimes\ch_j(\overline \delta^{\pa})\in H^\bullet(\CP_X\times \CP_X)\,.\label{eq: chextpa}
\end{align}
In the first line, we sum over $\gamma, \delta$ in our prescribed basis. In the second line, we sum over the basis 
\[\{\gamma^{\pa}\}=\{\gamma^\CV\}\cup \{\gamma^\CF\}\subseteq H^\bullet(X)^{\oplus 2}\]
and use the definition of $\chi^{\pa}$.
The pairings $\chi, \chi^{\pa}$ coincide with the ones for $K$-theory previously described under the Chern character isomorphism, i.e.,
\[\chi(\gamma, \delta)=\int_X(-1)^{\lfloor\frac{\deg \gamma}{2}\rfloor}\gamma\cdot \delta \cdot \td(X).\]
Note that all the non-zero terms in \eqref{eq: chextpa} have $|\gamma|=|\delta|=|\overline \gamma|=|\overline \delta|$, so we replace the occurrence of any of those parities by $|\gamma|$.
From \eqref{eq: chextpa} (and being again careful with the signs introduced by taking the dual and by $(\sigma^{\pa})^\ast$) we get the Chern character of the symmetrization 
\begin{equation}
\label{eq: chthetapa} \ch(\Theta^{\pa})=\sum_{\gamma^{\pa}, \delta^{\pa}}\sum_{i, j\geq 0}(-1)^{j}\chi^{\pa}_\sym(\gamma^{\pa}, \delta^{\pa})\ch_i(\overline \gamma^{\pa})\boxtimes\ch_j(\overline \delta^{\pa})\,.
\end{equation}
By Newton's identities we have
\begin{align}\label{eq: chernclasstheta}c_{z^{-1}}(\Theta^{\pa})&=\exp\Big[\sum_{k\geq 0}(-1)^{k}k!\ch_{k+1}(\Theta^{\pa})z^{-k-1}\Big]\\ \nonumber
&=\exp\Big[\sum_{\gamma^{\pa}, \delta^{\pa}}\sum_{i+j\geq |\gamma^{\pa}|+1}(-1)^{i-|\gamma^{\pa}|-1}(i+j-|\gamma^{\pa}|-1)!z^{-i-j+|\gamma^{\pa}|}\\
&\qquad \qquad\qquad \chi^{\pa}_\sym(\gamma^{\pa}, \delta^{\pa})\ch_i(\overline \gamma^{\pa})\boxtimes\ch_j(\overline \delta^{\pa})\Big]\, .\nonumber
\end{align}
It suffices to consider the expansion of this exponential up to the linear terms in
$$Y(v^{\pa}_{-1},z)= \Sigma_\ast^{\pa}\big[e^{zT}\boxtimes 1\big(c_{z^{-1}}(\Theta^{\pa})\cap (v^{\pa}_{-1}\boxtimes -)\big)\big]\,,
$$
because quadratic terms and beyond annihilate $(v^{\pa}_{-1}\boxtimes -)$ for degree reasons. The constant term of the exponential \eqref{eq: chernclasstheta}, namely $1$, determines the creation part of the field $Y(v^{\pa}_{-1},z)$ as
\begin{equation}\label{eq: creation part}
    \Sigma_\ast^{\pa}\big[e^{zT} v_{-1}^{\pa}\boxtimes -\big]=\sum_{k\geq 0}v^{\pa}_{(-k-1)}z^k\,.
\end{equation}
This uses the fact that
\[\frac{T^k}{k!} v_{-1}^{\pa}=v_{-k-1}^{\pa}\,, \quad \Sigma_\ast^{\pa}\big(v_{-k-1}^{\pa}\boxtimes -\big)=v_{(-k-1)}^{\pa}\]
for $k\geq 0$, see \cite[Lemmas 5.3, 5.5]{Gr19} (the first one also follows from Lemma \ref{lem: translationdualr-1}). On the other hand, it suffices to consider the linear terms of \eqref{eq: chernclasstheta} with $i=1$ and $j-|\gamma^{\pa}|=k\geq 0$ for degree reasons. They determine the annihilation part of the field $Y(v^{\pa}_{-1},z)$ as 
\begin{align}\label{eq: annihilation part}
    \Sigma^{\pa}_*\bigg[\sum_{\gamma^{\pa}, \delta^{\pa}}\nonumber
    &\sum_{k\geq 0}(-1)^{|\gamma^{\pa}|}k!z^{-k-1}
    \chi^{\pa}_\sym(\gamma^{\pa}, \delta^{\pa})\ch_1(\overline \gamma^{\pa})\boxtimes\ch_{k+|\gamma^{\pa}|}(\overline \delta^{\pa})\cap (v^{\pa}_{-1}\boxtimes-)\bigg]\nonumber\\
    &=\sum_{\gamma^{\pa}, \delta^{\pa}}\sum_{k\geq 0}k!z^{-k-1}
    \chi^{\pa}_\sym(\gamma^{\pa}, \delta^{\pa})\langle\overline{\gamma}^{\pa},v^{\pa}\rangle\,\ch_{k+|\gamma^{\pa}|}(\overline \delta^{\pa})\cap \nonumber\\
    &=\sum_{\delta^{\pa}}\sum_{k\geq 0}k!z^{-k-1}
    (-1)^{|\delta^{\pa}|}\chi^{\pa}_\sym(v^{\pa}, \delta^{\pa})\,\ch_{k+|\delta^{\pa}|}(\overline \delta^{\pa})\cap\,.
\end{align}
The sign disappears in the second line due to the interaction between the cap product and the tensor product \cite[12.17]{dold} and reappears in the last line by using
\[(-1)^{|\delta^{\pa}|}\sum_{\gamma^{\pa}}\langle\overline{\gamma}^{\pa},v^{\pa}\rangle\chi^{\pa}_\sym(\gamma^{\pa}, \delta^{\pa})=\sum_{\gamma^{\pa}}\langle v^{\pa}, \overline{\gamma}^{\pa}\rangle\chi^{\pa}_\sym(\gamma^{\pa}, \delta^{\pa})=\chi^{\pa}_\sym(v^{\pa}, \delta^{\pa})\,.\]
We can further simplify this expression by replacing descendent actions with derivatives. Let $\{w^{\pa}\}\subseteq K^\bullet(X)^{\oplus 2}=\Lambda^{\pa}$ be the basis obtained by applying the inverse of the Chern character isomorphism to $\{\delta^{\pa}\}$. By \eqref{eq: capdescendentspairs},
\[\ch_j(\overline \delta^{\pa})\cap -=\frac{(-1)^{|\delta^{\pa}|}}{(j-1)!}\frac{\partial}{\partial w^{\pa}_{-j}}\, ,\quad \textup{for }j\geq 1\,. \]
The constant descendent action by $\ch_0(\overline{\delta}^{\pa})$ is treated separately as a multiplication by $\langle \overline{\delta}^{\pa},\alpha^{\pa}\rangle$ on $H_\bullet(\CP_{\alpha^{\pa}})$. Therefore, we have 
\begin{align*}
    \eqref{eq: annihilation part}
    &=\sum_{k\geq 1-|v^{\pa}|}\sum_{w^{\pa}}\chi^{\pa}_{\sym}(v^{\pa},w^{\pa})k^{1-|v^{\pa}|}\frac{\partial}{\partial w^{\pa}_{-k-|v^{\pa}|}} z^{-k-1}+\chi^{\pa}(v^{\pa},\alpha^{\pa})z^{-1}
\end{align*}
on $H_\bullet(\CP_{\alpha^{\pa}})$. Combining with the creation part \eqref{eq: creation part}, this matches exactly with \eqref{Eq: explicitfield}, so we are done with (ii).

For (iii) we are left with computing $Y(e^{\alpha^{\pa}},z)e^{\beta^{\pa}}$:

\begin{align*} Y(e^{\alpha^{\pa}},z)e^{\beta^{\pa}}
&= (-1)^{\chi^{\pa}(\alpha^{\pa},\beta^{\pa})}z^{\chi_\sym^{\pa}(\alpha^{\pa},\beta^{\pa})}\Sigma_*\big[e^{zT}\boxtimes \id (c_{z^{-1}}(\Theta^{\pa})\cap (e^{\alpha^{\pa}}\boxtimes e^{\beta^{\pa}}))\big]\\
&=(-1)^{\chi^{\pa}(\alpha^{\pa},\beta^{\pa})}z^{\chi_\sym^{\pa}(\alpha^{\pa},\beta^{\pa})}\Sigma_*\big[e^{zT}(e^{\alpha^{\pa}})\boxtimes  e^{\beta^{\pa}}\big]\,,
\end{align*}
where we used the fact that $\ch_i(\overline \gamma^{\pa})\cap -$ and $\ch_j(\overline \delta^{\pa})\cap -$ annihilate $e^{\alpha^{\pa}}, e^{\beta^{\pa}}$ for any $i,j>0$. The $z^{\chi_\sym^{\pa}(\alpha^{\pa},\beta^{\pa})}$ coefficient is simply
\[(-1)^{\chi^{\pa}(\alpha^{\pa},\beta^{\pa})}\Sigma_*(e^{\alpha^{\pa}}\boxtimes e^{\beta^{\pa}})=(-1)^{\chi^{\pa}(\alpha^{\pa},\beta^{\pa})}e^{\alpha^{\pa}+\beta^{\pa}}\,,\]
as required in (iii), thus finishing the proof.\qedhere
\end{proof}

Recall from Section \ref{Sec: physical} that associated to a vertex algebra $V_\bullet$ we have a Lie algebra $\widecheck V_\bullet=V_{\bullet+2}/TV_\bullet$. This quotient of $V_{\bullet}$ also has a geometric interpretation observed by \cite{Jo17} and related to $\BD^X_{\inv, \alpha}$ from Definition \ref{def: invariantdescendents}. We begin by recalling that the $B\BG_m$--action $\rho$ on $\CM_X$ leads to a \textit{rigidification} (see \cite[Appendix]{AOV} and \cite[Proposition 2.25 (b)]{Jo17}) which quotients out $\BG_m\cdot \id$ from the stabilizer of each object in the moduli stack. We denote it by 
$$
\CM_\alpha^\rig=\CM_\alpha \mkern-7mu\fatslash B\BG_m\,.
$$
\begin{lemma}
\label{Lem: DXinv}
Let $X$ be a variety in class D. Let $\ch(\alpha)\neq 0$ and $\pi^{\rig}_{\alpha}: \CM_\alpha\to \CM^{\rig}_\alpha$ be the projection, then $(\pi^{\rig}_\alpha)^\ast: H^\bullet(\CM^{\rig}_\alpha)\to H^\bullet(\CM_{\alpha})$ is injective and the isomorphism of Lemma \ref{Lem: isomorphism} induces
$$
\BD^X_{\inv,\alpha}\cong(\pi^{\rig}_{\alpha})^\ast\big(H^\bullet(\CM^{\rig}_{\alpha})\big)\cong H^\bullet(\CM^{\rig}_{\alpha})\,.
$$
Equivalently, we have the isomorphism
$$
\widecheck{H}_{\bullet}(\CM^{\rig}_{\alpha}) \coloneqq H_{\bullet-2\chi(\alpha,\alpha)+2}(\CM^{\rig}_\alpha)\cong \widecheck V_{\alpha, \bullet}\,.
$$
\end{lemma}
\begin{proof}
A big part of this proof is already due to Joyce \cite[Proposition 3.24]{Jo17}, \cite[Theorem 4.8, Remark 4.10]{Jo21}. He proves that 
$$
H_\bullet(\CM^{\rig}_{\alpha}) =H_\bullet(\CM_\alpha)/T(H_{\bullet-2}(\CM_\alpha))\,,
$$
when $\ch(\alpha)\neq 0$. This is the last statement of the lemma. Dually, we have injectivity of the pull-back and
\[(\pi^{\rig}_{\alpha})^\ast\big(H^\bullet(\CM^{\rig}_{\alpha})\big) =\{D\in H^\bullet(\CM_{\alpha}) \colon \rho^*D = 1\boxtimes D \textup{ in }H^\bullet(B \BG_m\times \CM_\alpha) \}.\] The isomorphism with $\BD^X_{\inv,\alpha}$ finally follows from the the fact that $\rho^\ast=e^{\zeta \bR_{-1}}$ under the identification $H^\bullet(\CM_\alpha)\cong \BD^X_\alpha$ as we showed in the proof of Lemma \ref{lem: translationdualr-1}. 
\end{proof}

\begin{example}\label{ex: invdescendents}
We continue with the Example \ref{Ex: DXinvdescendents} proving it differently for the full stack.  On $\CM_{\alpha}$ and $\CM^{\rig}_{\alpha}$, virtual tangent bundles $T^{\vir}\CM_{\alpha}$ and $T^{\vir}\CM_{\alpha}^{\rig}$ are defined as the topological K-theory classes of the duals of the natural obstruction theory complexes\footnote{In \cite{TV07}, the stack $\CM_{\alpha}$ is constructed by truncating its derived refinement. Because derived refinements naturally induce obstruction theories (see \cite[Proposition 1.2]{STV}), we have a natural obstruction theory on $\CM_{\alpha}$ given by the formula that follows as shown in \cite[Corollary 3.29]{TV07}. Because these derived refinements can be rigidified in the same way as the original stacks, there is also a natural choice of an obstruction theory on $\CM^{\rig}_{\alpha}$.}.  Before rigidifying, the virtual tangent bundle is given by $$T^\vir\CM_{\alpha}=-\RHom_{\CM_\alpha}(\CG,\CG)\,,$$
and its relation to  $T^\vir\CM^{\rig}_{\alpha}$ in $K$-theory is
$$
T^\vir\CM_\alpha = (\pi^{\rig}_\alpha)^*(T^\vir\CM^{\rig}_\alpha) - \CO_{\CM_\alpha}\,.
$$
  We see from Lemma \ref{Lem: DXinv} that $\ch(T^{\vir})\in \BD^X_{\inv, \alpha}$. This clearly holds for any $K$-theory class pulled back from $\CM^{\rig}_{\alpha}$.
\end{example}

\subsection{Virasoro operators from Kac's conformal element}
\label{sec: virasorocomparison}

Construction of conformal element depends on Assumption \ref{Ass: vanishing} which we assume throughout this section. Recall that we need to make a choice of a maximal isotropic decomposition of $K^1(X)^{\oplus 2}$ with respect to a symmetric non-degenerate pairing $\chi^{\pa}_\sym$ (recall Lemma~\ref{Lem: nondegenerate}) to define Kac's conformal element. Hodge decomposition is a source of the decomposition that we use. Define the subspaces $K^{\bullet,\bullet+1}(X)$ and $K^{\bullet+1,\bullet}(X)$ of $K^1(X)$ such that via the Chern character isomorphism we have 
$$K^{\bullet,\bullet+1}(X)\xrightarrow{\sim}\bigoplus_{p}H^{p,p+1}(X),\quad K^{\bullet+1,\bullet}(X)\xrightarrow{\sim}\bigoplus_{p}H^{p+1,p}(X)\,.
$$
We consider a decomposition of $K^1(X)^{\oplus 2}=I\oplus \hat{I}$ given by 
\begin{equation}\label{Hodge isotropic decomposition}
    I\coloneqq K^{\bullet,\bullet+1}(X)^{\oplus 2},\quad \hat{I}\coloneqq K^{\bullet+1,\bullet}(X)^{\oplus 2}\,,
\end{equation}
that is maximally isotropic due to Hodge degree consideration. This defines a conformal element $\omega$ inside $V_\bullet^{\pa}=\BC[\Lambda^{\pa}_{\sst}]\otimes \BD_{\Lambda^{\pa}}$ hence the vertex Virasoro operators $L_n^{\pa}$ for all $n\in\BZ$. On the other hand, we defined in Section \ref{sec: virasoropairs} descendent Virasoro operators $\bL_n^{\pa}$ for $n\geq -1$ acting on the formal pair descendent algebra $\BD^{X,\pa}$. 

In this section, we prove duality between the two Virasoro operators $L_n^{\pa}$ and $\bL_n^{\pa}$ for $n\geq -1$. To set up the stage where we state the duality, let $\alpha^{\pa}\in \Lambda_{\sst}^{\pa}$ and consider the realization homomorphism 
$$p_{\alpha^{\pa}}:\BD^{X,{\pa}}\rightarrow \BD^{X,\pa}_{\alpha^{\pa}}\,,
$$
as in Definition \ref{def: topdescendentalgebra}. By Assumption \ref{Ass: vanishing}, this realization homomorphism is surjective. Furthermore, descendent Virasoro operators $\bL_k^{\pa}$ factor through this quotient as follows from noting that $\ker(p_{\alpha})$ is only generated by symbols of the form $\ch^H_0(\gamma)$ giving us $$\bR_k\big(\!\ker(p_{\alpha})\big)=0 \,.$$
We use the same notation for the resulting operators on $\BD^{X,\pa}_{\alpha^{\pa}}$. 
\begin{theorem}\label{thm: duality}
For any $\alpha^{\pa}\in \Lambda_{\sst}^{\pa}$ and $n\geq -1$, Virasoro operators $L_n^{\pa}$ and $\bL_n^{\pa}$ are dual to each other with respect to the perfect pairing 
$$\BD^{X,\pa}_{\alpha^{\pa}}\otimes e^{\alpha^{\pa}}\, \BD_{\Lambda^{\pa}}\rightarrow \BC\,.
$$
\end{theorem}
\begin{proof}
On $K^\bullet(X)^{\oplus 2}$ we are given a symmetric non-degenerate pairing
\begin{align*}
    \chi^{\pa}_{\sym}\big(v^{\pa},w^{\pa}\big)&=\chi^{\pa}(v^{\pa},w^{\pa})+\chi^{\pa}(w^{\pa},v^{\pa})\\
    &=\chi(v_2-v_1,w_2)+\chi(w_2-w_1,v_2)\,.
\end{align*}
The maximal isotropic decomposition \eqref{Hodge isotropic decomposition} determines the conformal element $\omega\in V^{\pa}_\bullet$ and a new supersymmetric non-degenerate pairing (see Section \ref{sec: Kacconformal}) $Q^\omega=\chi^{\HH,\pa}_{\ssym}$. We use the superscript $\HH$ instead of $\omega$ to indicate the relevance to the holomorphic degree in the Hodge decomposition. This new pairing can be written as
\begin{align*}
    \chi^{\HH,\pa}_{\ssym}\big(v^{\pa},w^{\pa}\big)&=\chi^{\HH,\pa}(v^{\pa},w^{\pa})+(-1)^{|v^{\pa}||w^{\pa}|}\chi^{\HH,\pa}(w^{\pa},v^{\pa})\\
    &=\chi^\HH(v_2-v_1,w_2)+(-1)^{|v^{\pa}||w^{\pa}|}\chi^\HH(w_2-w_1,v_2)\,,
\end{align*}
where $\chi^\HH$ is a pairing on $K^\bullet(X)$ defined as
$$\chi^\HH(v,w)\coloneqq (-1)^p\int \ch(v)\ch(w)\td(X)\ \ \textnormal{if}\ \ch(v)\in H^{p,\bullet}(X)\,.
$$
The supersymmetric pairing $\chi^{\HH,\pa}_{\ssym}$ allows a simple formulation of the conformally shifted field:
\begin{equation}\label{eq: shifted field}
    Y(v_{-1}^{\HH,\pa},z)
    =\sum_{k\geq 0} v^{\HH,\pa}_{(-k-1)}z^k+\sum_{k\geq 0}\sum_{\delta^{\pa}}k!z^{-k-1}\chi^{\HH,\pa}_{\ssym}(\delta^{\pa},v^{\pa})\,\ch_{k}^\HH(\overline{\delta}^{\pa})\cap\,.
\end{equation}
This formula is proven by the non-shifted version for $Y(v_{-1}^{\pa},z)$ as in \eqref{eq: creation part}, together with a case division analysis as in the proof of \eqref{eq: shifted bracket}. 

Recall that the vertex Virasoro operator are written as 
$$L_n^{\pa}=\frac{1}{2}\sum_{\substack{ i+j=n \\ v^{\pa}\in B^{\pa}}}:\tilde{v}^{\HH,\pa}_{(i)} v^{\HH,\pa}_{(j)}:\quad n\in\BZ\,.
$$
When $n\geq -1$, we write $L_n^{\pa}=R_n^{\pa}+T_n^{\pa}$ where 
$$R_n^{\pa}\coloneqq \frac{1}{2}\sum_{\substack{ i+j=n\\ i<0\textnormal{ or }j<0 \\ v^{\pa}\in B^{\pa}}}:\tilde{v}^{\HH,\pa}_{(i)} v^{\HH,\pa}_{(j)}:\,,\quad
T_n^{\pa}\coloneqq \frac{1}{2}\sum_{\substack{ i+j=n\ \\i,j\geq 0 \\ v\in B^{\pa}}}:\tilde{v}^{\HH,\pa}_{(i)} v^{\HH,\pa}_{(j)}:\,.
$$
From the computation \eqref{eq: shifted field}, we have
\begin{align*}
    T_n^{\pa}
    &=\frac{1}{2}\sum_{\substack{i+j=n\\i,j\geq 0}}i!j!\sum_{v^{\pa}\in B^{\pa}}\sum_{\delta^{\pa}, \gamma^{\pa}}
    \chi^{\HH,\pa}_{\ssym}(\delta^{\pa},\tilde{v}^{\pa})
    \chi^{\HH,\pa}_{\ssym}(\gamma^{\pa},v^{\pa})\,
    \ch_i^\HH(\overline{\delta}^{\pa})\ch_j^\HH(\overline{\gamma}^{\pa})\cap -\\
    &=\frac{1}{2}\sum_{\substack{i+j=n\\i,j\geq 0}}i!j!
    \sum_{\delta^{\pa}, \gamma^{\pa}}
    \Big(\chi^{\HH,\pa}(\gamma^{\pa},\delta^{\pa})+(-1)^{|\gamma^{\pa}||\delta^{\pa}|}\chi^{\HH,\pa}(\delta^{\pa},\gamma^{\pa})\Big)
    \ch_i^\HH(\overline{\delta}^{\pa})\ch_j^\HH(\overline{\gamma}^{\pa})\cap -\\
    &=\sum_{\substack{i+j=n\\i,j\geq 0}}i!j!
    \sum_{\delta^{\pa}, \gamma^{\pa}}
    \chi^{\HH,\pa}(\gamma^{\pa},\delta^{\pa})\,
    \ch_i^\HH(\overline{\delta}^{\pa})\ch_j^\HH(\overline{\gamma}^{\pa})\cap -\\
    &=\sum_{\substack{i+j=n\\i,j\geq 0}}i!j!
   \sum_{\delta,\gamma}
    \chi^{\HH}(\gamma,\delta)\,
    \ch_i^{\HH,\CF}(\overline{\delta})\ch_j^{\HH,\CF-\CV}(\overline{\gamma})\cap -\\
    &=\sum_{\substack{i+j=n\\i,j\geq 0}}i!j!
   \sum_{\delta,\gamma}
   (-1)^{\dim(X)-p(\overline{\gamma})}
    \left(\int \delta \cdot \gamma\cdot \td(X)\right) \ch_i^{\HH,\CF-\CV}(\overline{\gamma})\ch_j^{\HH,\CF}(\overline{\delta})\cap -\\
    &=\sum_{\substack{i+j=n\\i,j\geq 0}}i!j!\sum_t (-1)^{\dim(X)-p_t^L}\ch_i^{\HH,\CF-\CV}(\gamma_t^L)\ch_j^{\HH,\CF}(\gamma_t^R)\cap -\,,
\end{align*}
where in the last equality the summation takes over
$$\Delta_*(\td(X))=\sum_t \gamma_t^L\otimes \gamma_t^R.
$$
This is exactly dual to the multiplication operator $\bT_n^{\pa}$. 

Now we prove that operators $R_n^{\pa}$ and $\bR_n^{\pa}$ are dual to each other. Let $R_n^{{\pa}, \dagger}\colon \BD^{X,\pa}_{\alpha^{\pa}}\to \BD^{X,\pa}_{\alpha^{\pa}}$ be the dual to $R_n^{\pa}$. For $n\geq -1$, both $R_n^{{\pa}, \dagger}$ and $\bR_n^{\pa}$ annihilate $1$ so it is enough to show that their commutators with right multiplication by descendents agree, i.e.,
\[[R_n^{\pa, \dagger}, \cdot\,\chh_k(\gamma^{\pa})]=[\bR_n^{\pa}, \cdot\,\chh_k(\gamma^{\pa})]=\left(\prod_{j=0}^n(k+j)\right)\cdot\ch_{k+n}^\HH(\gamma^{\pa})\,.\]
Dually, this is equivalent to
\begin{equation*}
\left[\ch_k^{\HH}(\gamma^{\pa})\cap\,,\,R_n^{\pa}\right]=\left(\prod_{j=0}^n(k+j)\right)\ch_{k+n}^\HH(\gamma^{\pa})\cap\,.
\end{equation*}

We finish the proof by showing the required commutator relation. Since $\chi^{\HH,\pa}_{\ssym}$ is a perfect pairing, there is a unique $w^{\pa}\in K(X)^{\pa}$ such that 
$$\sum_{\delta^{\pa}}\chi^{\HH,\pa}_{\ssym}(\delta^{\pa}, w^{\pa})\,\overline{\delta}^{\pa}=\gamma^{\pa}\,.
$$
From \eqref{eq: shifted field}, this implies that $w_{(k)}^{\HH,\pa}=k!\ch_k^\HH(\gamma^{\pa})\cap\,$. On the other hand, we have
\begin{align*}
    R_n^{\pa}
    &=\frac{1}{2}\sum_{\substack{ i+j=n\\ i<0\\ v^{\pa}\in B^{\pa}}}\tilde{v}^{\HH,\pa}_{(i)} v^{\HH,\pa}_{(j)}
    +\frac{1}{2}\sum_{\substack{ i+j=n\\ j<0 \\ v^{\pa}\in B^{\pa}}}(-1)^{|v^{\pa}|} v^{\HH,\pa}_{(j)}\tilde{v}^{\HH,\pa}_{(i)}
    =\sum_{\substack{ i+j=n\\ i<0\\ v^{\pa}\in B^{\pa}}}\tilde{v}^{\HH,\pa}_{(i)} v^{\HH,\pa}_{(j)}\,,
\end{align*}
since the dual $\tilde{v}^{\pa}$ is defined by supersymmetric pairing $\chi^{\HH,\pa}_{\ssym}$. Therefore, we obtain 
\begin{align*}
    \left[\ch_k^{\HH}(\gamma^{\pa})\cap\,,\,R_n^{\pa}\right]
    &=\frac{1}{k!}\sum_{\substack{ i+j=n\\ i<0\\ v\in B^{\pa}}}
    \left[w_{(k)}^{\HH,\pa}, \tilde{v}^{\HH,\pa}_{(i)}v^{\HH,\pa}_{(j)}\right]\\
    &=\frac{1}{k!}\sum_{\substack{ i+j=n\\ i<0\\ v\in B^{\pa}}} \left[w^{\HH,\pa}_{(k)}, \tilde{v}^{\HH,\pa}_{(i)}\right]v^{\HH,\pa}_{(j)}+\tilde{v}^{\HH,\pa}_{(i)}\left[w^{\HH,\pa}_{(k)},v^{\HH,\pa}_{(j)}\right]\\
    &=\frac{1}{k!}\sum_{v\in B^{\pa}}k\cdot\chi^{\HH,\pa}_{\ssym}(w^{\pa},\tilde{v}^{\pa})v_{(k+n)}^{\HH,\pa}\\
    &=\left(\prod_{j=0}^n(k+j)\right)\chh_{k+n}(\gamma^{\pa})\cap\,,
\end{align*}
where we used the bracket formula \eqref{eq: shifted bracket} with $k\geq 0$.  
\end{proof}

\begin{remark}
\label{rem: obstructionvsvirasoro}
 To summarize the ideas leading up to this final statement, we explain where the obvious similarity between $\bT_k$ and the virtual tangent bundle (see Example \ref{Ex: Tvir}) comes from in general. It is best represented by the diagram
 $$
 \begin{tikzcd}
&\chi^{\pa}_{\ssym},\chi^{\HH,\pa}_{\ssym}\arrow[dr,"(\textnormal{iii})", bend left]&\\
\arrow[d,"(\textnormal{i})"']\Ext^{\pa},\Theta^{\pa} \arrow[ur,"(\textnormal{ii})", bend left]&&\omega\arrow[d, "(\textnormal{iv})"]\\
T^{\vir}\arrow[rr,dashed,"?"]&&\bL_k^{\pa}\,.
 \end{tikzcd}
 $$
Meaning of the arrows are explained below. 
 \begin{enumerate}[label = (\roman*)]
\item   represents pullback along the diagonal which restricts $\Ext^{\pa}$ to the virtual tangent bundle of any pair moduli space mapping to $\CP_X$. This is the content of assumption \cite[Assumption 4.4]{Jo21} and is satisfied in larger generality than our Definition \ref{def: pairVOA}.
\item  corresponds to taking ranks of the symmetrization $\Theta^{\pa}$ of $\Ext^{\pa}$ and then fixing a choice of an isotropic splitting  $\Lambda_{\ovO}^{\pa} = I\oplus \hat{I}$ of the odd part of $\Lambda^{\pa}$ as we did in \eqref{Eq: splittingisotropic}. Out of it, we constructed a supersymmetric pairing  which in the case 
$$
I= K^{\bullet,\bullet+1}(X)^{\oplus 2},\quad \hat{I}= K^{\bullet+1,\bullet}(X)^{\oplus 2}\,,
$$
led to $\chi^{\HH,\pa}_{\ssym}$. This can be generalized to any setting where the induced pairing is non-degenerate so that isotropic splitting can be chosen.
\item is assigning the conformal element for a given choice of an isotropic splitting in the procedure described in Section \ref{sec: Kacconformal} (more explicitly see \eqref{Eq: omega}) and works as long as Joyce's construction in Section \ref{sec: Joycevertexalgebra} leads to a lattice vertex algebra.
\item is filled in by Theorem \ref{thm: duality} that fundamentally depends on the comparison between $\BD^{X,\pa}$ and $H^\bullet(\CP_{X})$ contained in the description of the lattice vertex algebra structure on $\widehat{H}_\bullet(\CP_X)$ proved in Theorem \ref{thm: gross}. This would again work in a setting where a similar comparison can be made.
 \end{enumerate}
 A direct geometric relation represented by the ``?" between Virasoro constraints and the virtual tangent bundle is however unclear.
\end{remark}

\subsection{Virasoro constraints and primary states}
\label{subsec: virasoroprimary}

Let $M=M_\alpha$, for $\ch(\alpha)\neq 0$, be a moduli space of sheaves as in Section \ref{sec: modulisheavespairs} with a universal sheaf $\BG$. By the universal property of the stack $\CM_X$ there is a map $f_\BG\colon M\to \CM_X$ such that
$(f_\BG\times \id_X)^\ast \CG=\BG$
where $\CG$ is the universal complex in $\CM_X\times X$. Even without a universal sheaf $\BG$, we always have an open embedding into the rigidified stack $\iota\colon M \hookrightarrow \CM_X^\rig$. When a universal sheaf exists this embedding is the composition
\[\iota\colon M\overset{f_\BG}{\longrightarrow} \CM_X\longrightarrow \CM_X^\rig\,.\]
We define classes in $\CM_X$, $\CM_X^\rig$ by
\begin{align*}
[M]^\vir&\coloneqq \iota_\ast [M]^\vir \in H_\bullet(\CM_X^\rig)\,,\\
[M]^\vir_\BG&\coloneqq (f_\BG)_\ast [M]^\vir \in H_\bullet(\CM_X)=V_\bullet\,.
\end{align*}
By Lemma \ref{Lem: DXinv} we may regard the class $[M]^\vir$ as being in Borcherds Lie algebra $\widecheck V_\bullet=V_{\bullet+2}/TV_\bullet$. Given any $\BG$ the class $[M]^\vir_\BG$ is a lift of $[M]^\vir\in  V_{\bullet+2}/TV_\bullet$ to the vertex algebra $V_\bullet$ -- quotienting by $T$ removes the ambiguity in the choice of $\BG$. The integrals of geometric realizations of descendents $D\in \BD^X_\alpha$ can be expressed in terms of these classes by
\[\int_{[M]^\vir}\xi_\BG(D)=\int_{[M]^\vir}(f_\BG)^\ast \big(\xi_\CG(D)\big)=\int_{[M]^\vir_\BG}\xi_\CG(D)\,.\]
By Theorem \ref{Lem: isomorphism}, if $X$ is a variety in class D then $\xi_\CG$ is an isomorphism between $\BD^X_\alpha$ and $H^\bullet(\CM_\alpha)$, so knowing the class $[M]^\vir_\BG$ is precisely the same as knowing all the descendent integrals. Similarly, by Lemma \ref{Lem: DXinv} the class $[M]^\vir \in \widecheck V_\bullet$ contains precisely the information of integrals of weight 0 descendents $D\in \BD^X_{\alpha, \inv}$.

An analogous situation happens for moduli spaces of pairs and the stack $\CP_X$. Given a moduli of pairs as in Section \ref{sec: modulisheavespairs} with universal sheaf $q^\ast V\to \BF$, by the universal property of the stack $\CP_X$ we have a map
\[f_{(q^\ast V, \BF)}\colon P\to \CP_X\,,\]
such that 
\[(f_{(q^\ast V, \BF)}\times \id_X)^\ast \CV=q^\ast V\,,\quad (f_{(q^\ast V, \BF)}\times \id_X)^\ast \CF=\BF\,.\]
The classes 
\[[P]^\vir_{(q^\ast V, \BF)}\coloneqq (f_{(q^\ast V, \BF)})_\ast [P]^\vir\in V_\bullet^{\pa}\] 
contain exactly the information of the descendent integrals
\[\int_{[P]^\vir}\xi_{(q^\ast V, \BF)}(D)\,\quad\textup{for}\quad D\in \BD^{X, {\pa}}_{\alpha^{\pa}}\,.\]
A class $[M]^\vir$ coming from a moduli of sheaves can also be considered in $\widecheck V_\bullet^{\pa}$ via the embedding $\CM_X\hookrightarrow \CP_X$ sending $G\mapsto (0,G)$.

We now use Theorem \ref{thm: duality}, which states the duality between the Virasoro operators on the descendent algebra and and on the vertex algebra, to prove Theorem \ref{thm: virasorophysical} saying that the Virasoro constraints holding for some moduli space of sheaves $M$ or pairs $P$ are equivalent to their respective classes on the Lie/vertex algebra being primary states.
\begin{proof}[Proof of Theorem \ref{thm: virasorophysical}]
We start with part (2) which refers to pairs. Under Assumption \ref{Ass: vanishing} the morphism 
$p_{\alpha^{\pa}}\colon \BD^{X, \pa}\to \BD^{X, \pa}_{\alpha^{\pa}}$ is surjective, so Conjecture \ref{conj: pairvirasoro} holds if and only if
\[\int_{[P]^\vir}\xi_{(q^\ast V, \BF)}(\bL_n^{\pa}(D))=0\]
for every $D\in \BD^{X, \pa}_{\alpha^{\pa}}$ and $n\geq 0$. By Theorem \ref{thm: duality} and the previous observations
\[\int_{[P]^\vir}\xi_{(q^\ast V, \BF)}(\bL_n^{\pa}(D))=\int_{[P]^\vir_{(q^\ast V, \BF)}}\xi_{(\CV, \CF)}(\bL_n^{\pa}(D))=\int_{L_n^{\pa}([P]^\vir_{(q^\ast V, \BF)})}\xi_{(\CV, \CF)}(D)\,.\]
Since $\xi_{(\CV, \CF)}$ defines an isomorphism between $\BD^{X, \pa}_{\alpha^{\pa}}$ and the cohomology $H^\bullet(\CP_{\alpha^{\pa}})$ the last integral vanishes for all $D$ if and only if $L_n^{\pa}\big([P]^\vir_{(q^\ast V, \BF)}\big)=0$, i.e.,
\[[P]^\vir_{(q^\ast V, \BF)}\in P_0^{\pa}\,.\]

The claim (1) for sheaves follows in a similar way by noting that the operator $\bL_\inv$ is dual to $[-, \omega]$ by Theorem \ref{thm: duality} and Lemma \ref{WC with conformal element} and using the equivalente characterization of primary states in $\widecheck P_0$ provided by Proposition \ref{prop: primaryomegabracket}.\qedhere
\end{proof}

Joyce defines more generally (under certain conditions, see \cite[Theorem 5.7]{Jo21} or Section \ref{sec: joyceclasses}) classes $[M]^\inva\in \widecheck V_\bullet$ even when $M$ contains strictly semistable sheaves; when it does not contain strictly semistable sheaves, this class coincides with $[M]^\vir$. As in Remark \ref{rem:nocoarsedirect} and Conjecture \ref{conj:actualinvvir}, we say that $M$ satisfies the Virasoro constraints if $[M]^\inva$ is a primary state, i.e.
$$
  [M]^{\inva}\in \widecheck{P}_0\,.
$$

\section{Rank reduction via wall-crossing}
\label{sec: rankreduction}

In this section, we will explain the main step in the proof of Theorem \ref{thm: main}, which consists of a rank reduction argument via wall-crossing as described in \cite[Section 8.6]{Jo21} for positive rank and generalized here to include the case (3). The rank-reduction of loc. cit. is stated for slope stability only, so we use additional wall-crossing from slope stability to Gieseker stability to conclude Virasoro constraints for the latter stability condition, see Corollary \ref{cor:slopetogies}.

We will treat the 3 cases (1), (2), (3) in a uniform way by fitting them into the general framework of Joyce \cite{Jo21}.  Then their virtual fundamental classes viewed as elements of $H_{\bullet}(\CM^{\rig}_X)$ are related in terms of the Lie bracket defined on $\widecheck{V}_\bullet$. 

Let $X$ be a smooth projective variety of dimension $m=1$ or $m=2$. In the case of $m=2$, we assume that $h^{0,2}(X)=0$. Let $H$ be a fixed polarization of $X$. Given $1\leq d\leq m$, we consider the moduli spaces of $d$-dimensional slope semistable sheaves $F$ (cf. \cite[Theorem 4.3.4, Definition 8.2.7]{HL})
$$M_\alpha=M_\alpha^{\textup{ss}}(\mu)$$
with respect to the slope stability $\mu$, where $\alpha$ is the topological type of the sheaves we consider. In our case, we will work with $\alpha \in K^0(X)$ which can be represented by a non-zero sheaf on $X$, and we label this subset of $K^0(X)$ by $C(X)$. This fits the description of the wall-crossing data given in \cite[Assumption 5.1 (b)]{Jo21}\footnote{This assumption introduces a quotient $K\big(\Coh(X)\big)$ of $K_0\big(\Coh(X)\big)$  with $C\big(\Coh(X)\big)\subset K\big(\Coh(X)\big)$ containing elements represented by non-zero sheaves. The wall-crossing is then expressed in terms of $\alpha$'s contained in $C(X)$. In our case, we directly define $C\big(\Coh(X)\big) =C(X)$.}: there is a morphism 
$$
K_{0}\big(\Coh(X)\big)\to K^0(X)
$$
from the Grothendieck group of $\Coh(X)$ by smoothness and projectivity of $X$ and only the zero coherent sheaf has trivial topological $K$-theory class. The restriction of this morphism to the semigroup of elements of $K_{0}\big(\Coh(X)\big)$ represented by a non-zero sheaf maps onto $C(X)$ by its definition. We will not mention that $\alpha$ lies in $C(X)$ from now on. 

Recall that $\mu$ is defined by
\[\mu(F)=\frac{\deg(F)}{r(F)}\in \BQ\sqcup\{+\infty\}\]
where $\deg(F), r(F)$ are (normalized) coefficients of the Hilbert polynomial $P_F(z)$; denoting by $[z^k]P(z)$ the $z^k$ coefficient of a polynomial $P$,
\[\deg(F)=(d-1)![z^{d-1}]P_F(z) \quad \textup{ and } r(F)=d![z^{d}]P_F(z).\]
When $d=m$, the number $r(F)$ is (up to multiplication by a constant) the rank of $F$. In general, we regard the number $r(F)$ as a generalized rank; it is a non-negative integer for every sheaf of dimension at most $d$. The cases (1), (2), (3) in Theorem~\ref{thm: main} correspond, respectively, to $(m,d)=(1,1), (2,2), (2,1)$.

Twisting by $\CO_X(H)$ induces isomorphisms
\[M_\alpha\cong M_{\alpha(H)}.\]
It is a standard fact that, given a fixed $\alpha$, for large $n$ all the $\mu$-semistable sheaves with topological type $\alpha(nH)$ are globally generated and have vanishing higher cohomology (e.g. \cite[Corollary 1.7.7]{HL}). Replacing $\alpha$ by $\alpha(nH)$ we shall often assume this to be the case.
\begin{assumption}\label{ass: alphabig}
All the $\mu$-semistable sheaves $F$ of topological type $\alpha$ are globally generated and have vanishing higher cohomology $H^{>0}(F)=0$. 
\end{assumption}

\subsection{Bradlow pairs}
\label{sec: Bradlowpairs}
We now describe the notion of Bradlow stability,  depending on a parameter $t\in \BR_{>0}$, on pairs $(F, s)$ where $s\colon \CO_X\to F$ is a section.

\begin{definition}\label{def: bradlow}
Let $t> 0$. A pair $s\colon \CO_X\to F$ is $\mu^{t}$-(semi)stable if it is non-zero and:
\begin{enumerate}
    \item For every subsheaf $G\hookrightarrow F$ we have 
    \[\mu(G)(\leq)\mu(F)+\frac{t}{r(F)}.\]
    \item For every proper subsheaf $G\hookrightarrow F$ through which the section $s$ factors \[s\colon \CO_X\to G\to F\,,\] we have 
    \[\mu(G)+\frac{t}{r(G)}(\leq) \mu(F)+\frac{t}{r(F)}\,.\]
\end{enumerate}
The symbol $(\leq)$ stands for $\leq$ in the semistable case and $<$ in the stable case.
\end{definition}
 We denote by $P_\alpha^t$ the moduli space of $\mu^t$-semistable pairs $\CO_X\to F$ with topological type $\llbracket F\rrbracket=\alpha$. These are often called moduli spaces of Bradlow pairs and constructed by \cite[Section 1]{thaddeus} and \cite{LeP,lin} for curves and higher dimensions, respectively.\footnote{Strictly speaking, these moduli spaces are constructed for Gieseker type Bradlow stability as opposed to the slope type Bradlow stability in Definition \ref{def: bradlow}. However, we need the construction of the moduli space only when there are no strictly $\mu^t$-semistable pairs in which case the two stability notions coincide. When there exist strictly $\mu^t$-semistable pairs, Joyce's theory \cite{Jo21} uses only the moduli stacks rather than the moduli spaces.} If $t$ is such that $t\notin \frac{1}{r(\alpha)!}\BZ$ then the moduli spaces $P_{\alpha}^t$ have no strictly semistable objects. When this is the case, $P_\alpha^t$ is a fine moduli space with a universal pair \[\CO_{P_\alpha^t\times X}\to \BF\,,\] admitting a virtual class $[P_\alpha^t]^\vir$ by the standard argument as recorded below.

\begin{lemma}
\label{Lem:obs}
Assume that there are no strictly $\mu^t$-semistable pairs in $P^t_\alpha$. Then the moduli space $P^t_\alpha$ has a natural 2-term perfect obstruction theory given by $\RHom([\CO_{P^t_\alpha\times X}\to \BF], \BF)$ when $(m,d) = (1,1),(2,2),(2,1)$.
\end{lemma}
\begin{proof}
The construction of a morphism from the dual of $\RHom([\CO_{P^t_\alpha \times X}\to \BF], \BF)$ to the cotangent complex is standard, see for example \cite[Section 5.3]{mochizuki}.

Using the long exact sequence 
\[\Ext^{i}(F, F)\to H^{i}(F)\to\Ext^{i}([\CO_X\to F], F)\to \Ext^{i+1}(F, F)\to H^{i+1}(F)\,.\]
Mochizuki \cite[Lemma 6.1.14]{mochizuki} (see also Joyce \cite[Section 8.3.2]{Jo21}) showed the vanishing of $\Ext^i$ for $(m,d) = (1,1),(2,2)$ and $i\neq  0,1$ . When $(m,d)$ is such that $d\leq 1$, then the terms vanish immediately for $i\neq-1,0,1$ because $H^2(F) = 0$. 
The vanishing for $i=-1$ would follow from the injectivity of
$$\Ext^0(F,F)\xrightarrow{\,-\,\circ\, s\,} H^0(F)\,.
$$
Suppose for the contradiction that there exists a non-zero morphism $\phi\in \Ext^0(F,F)$ such that $\phi\circ s=0$. Consider the induced short exact sequence
$$0\rightarrow F_1\rightarrow F\rightarrow F_2\rightarrow 0\,,
$$
where $F_1=\ker(\phi)$ and $F_2=\coker(\phi)\subsetneq F$. By $\mu^t$-stability of $\CO_X\xrightarrow{s} F$, we have 
$$\mu(F_1)+\frac{t}{r(F_1)}<\mu(F)+\frac{t}{r(F)}\,,\quad \mu(F_2)<\mu(F)+\frac{t}{r(F)}\,.
$$
Using the usual arithmetic of ratios, this gives the contradiction.
\end{proof}

These moduli spaces fit in the framework of Section 8 of \cite{Jo21}. There, Joyce considers the abelian category $\grave{\CA}$ of pairs $U\otimes \CO_X\to F$ where $U$ is a $\BC$-vector space and $F\in \Coh(X)$, and introduces the stability function on such pairs given by
\[\mu^t(U\otimes \CO_X\to F)=\frac{\deg(F)+t \dim(U)}{r(F)}.\]
The $\mu^t$-(semi)stable pairs with $U=\BC$ are precisely the $\mu^t$-(semi)stable pairs in Definition \ref{def: bradlow}. Conditions (1) and (2) in Definition \ref{def: bradlow} correspond to looking for destabilizing subpairs of the form $0\to G$ and $\CO_X\to G$, respectively. 

The stack parametrizing objects in the category $\grave \CA$ is the stack $\CN_X$ from Remark~\ref{rem: stackNX}. Hence the moduli spaces $P_\alpha^t$ admit a map
\[P_\alpha^t\hookrightarrow \CN_X^\rig\rightarrow \CP_X^\rig\,,\]
and thus define a class $[P_\alpha^t]^\vir$ in the Lie algebras $H_\bullet(\CN_X^\rig)$ or $ H_\bullet(\CP_X^\rig)$ by pushing forward $[P_\alpha^t]^\vir\in H_\bullet(P_\alpha^t)$ along this map. Moreover, since we have a universal pair $\CO\to\BF$ there is actually a lift of this map to the non-rigidified stacks
\[f_{(\CO,\BF)}\colon P_\alpha^t\to \CN_X \rightarrow \CP_X\,.\]
This defines a lift of the class $[P_\alpha^t]^\vir$
\[[P_\alpha^t]^\vir_{(\CO,\BF)}\coloneqq (f_{(\CO,\BF)})_\ast [P_\alpha^t]^\vir\]
to the vertex algebras $H_\bullet(\CN_X)$ and $ H_\bullet(\CP_X)=V_\bullet^{\pa}$ when $P_\alpha^t$ does not contain strictly semistable pairs. This class is in the connected component $\CP_{(\llbracket \CO_X\rrbracket, \alpha)}$; to alleviate notation, we will write 
\[\CP_{(1, \alpha)}\coloneqq \CP_{(\llbracket \CO_X\rrbracket, \alpha)}\,.\]

\subsection{Limits $t\rightarrow 0$ and $t\to \infty$}
\label{sec: limits}
Our rank reduction argument will be based on using the wall-crossing formula to compare the $\mu^t$ Bradlow pairs with $t$ small and $t$ large. We now identify the moduli spaces $P_\alpha^t$ in these two limits.

\begin{proposition}[{\cite[Theorem 8.13, Ex. 5.6]{Jo21}}]

\label{prop: limitt0}
Let $0<t<1/r(\alpha)!$ and $F$ be a sheaf of topological type $\alpha$. Then  $s\colon \CO_X\to F$ is $\mu^{t}$-semistable if and only if it is $\mu^{t}$-stable if and only if the following three conditions hold:
\begin{enumerate}
    \item $F$ is semistable with respect to $\mu$;
    \item $s\neq 0$;
    \item there is no $0\neq G\subsetneq F$ with $\mu(G)=\mu(F)$ such that $\mathrm{im}(s)\subseteq G$.
\end{enumerate}
\end{proposition}

Since the stable objects do not change for such small $t$ we denote by $P_\alpha^{0+}=P_\alpha^t$ the moduli of $\mu^t$-stable pairs for $0<t\ll 1$ and by $\mu^{0+}$ the limit stability $\mu^t$ with $t\rightarrow 0$. The limit stability can be explicitly defined by
\[\mu^{0+}(U\otimes \CO_X\to F)=(\mu(F), \dim(U))\in (-\infty, \infty]\times \BZ_{\geq 0}\]
where $(-\infty, \infty]\times \BZ_{\geq 0}$ is given the lexicographic order.

Note in particular that if $M_\alpha$ has no strictly semistable sheaves then condition $(3)$ is vacuous, so $P_\alpha^{0+}$ parametrizes stable sheaves $[F]\in M_\alpha$ together with a non-vanishing section $s\in H^0(F)$ (up to scaling of the section). Assuming \ref{ass: alphabig}, 
\[P_\alpha^{0+}=\BP_{M_\alpha}(p_\ast \BG)\]
 is a projective bundle over $M_\alpha$ with fiber $\BP(H^0(F))\cong \BP^{\chi(\alpha)-1}$ over $[F]\in M_\alpha$. 

When $M_\alpha$ has strictly semistable sheaves, $P_\alpha^{0+}$ plays a crucial role in Joyce's definition of the classes $[M_\alpha]^{\mathrm{inv}}\in \widecheck{H}_\bullet(\CM_X)$ in \cite[Section $9.1$]{Jo21}, as we will recall next. This idea is also present in Mochizuki's work \cite{mochizuki}. 

\begin{proposition}[{\cite[Lemma 1.3]{PT}}]
\label{prop: limittinfty}
Let $t\gg 0$ be large enough. Then $s\colon \CO_X\to F$ is $\mu^{t}$-semistable if and only if it is $\mu^t$-stable if and only if $F$ is pure of dimension $d$ and $\coker(s)$ is supported in dimension at most $d-1$.
\end{proposition}

We denote by $P_\alpha^\infty$ the moduli of such pairs and by $\mu^\infty$ the limit stability. We proceed now to identify this moduli space in the three cases of interest to us, $(m,d)=(1,1), (2,2), (2,1)$.
\begin{enumerate}
\item Suppose that $(m,d)=(1,1)$ and $\rk(\alpha)>1$. Then $P_\alpha^\infty=\emptyset$ since \[\rk(\coker(s))\geq \rk(F)-\rk(\CO_X)>0\] for any $s\colon \CO_X\to F$ with $\rk(F)>1$.  Suppose now that $\rk(\alpha)=1$; then the elements of $P_\alpha^\infty$ are non-zero pairs $s\colon \CO_X\to F$ such that $F$ is a torsion-free rank 1 sheaf. A torsion-free rank 1 sheaf on a smooth curve $C$ with a non-zero section is given by an effective divisor. When $\ch(\alpha)=1+n \cdot \pt$ for $n\geq 0$ this implies that
\[P^\infty_{\alpha}=\{\CO_X\to \CO_{X}(E)\textup{ such that }|E|=n\}=C^{[n]}\]
is the $n$-th symmetric power of $C$.

\item Suppose that $(m,d)=(2,2)$. As before, $P_\alpha^\infty=\emptyset$ whenever $\rk(\alpha)>1$. Given a torsion-free rank 1 sheaf $F$ on a surface $S$, we get an embedding into its double dual $F\hookrightarrow F^{\vee\vee}$, which must be a line bundle $\CO_S(E)$ \cite[Example 1.1.16]{HL}. Twisting by $-E$ gives 
\begin{center}\begin{tikzcd}
\CO_S(-E)\arrow[r,"s(-E)"]&
F(-E)\arrow[r, hookrightarrow]&
\CO_S
\end{tikzcd}\end{center}
so $F(-E)=I_Z$ is the ideal sheaf of a zero-dimensional subscheme $Z\subseteq E$. Thus, the moduli space $P_{\alpha}^\infty$ is isomorphic to the nested Hilbert scheme $S_\beta^{[0,n]}$ in \cite{gsy} parametrizing a pair $(E,Z)$ of a divisor $E$ in class $\beta\in H_2(X, \BZ)$ and a zero-dimensional subscheme $Z\subseteq E$ of length $n$, where 
\[\ch(\alpha)=1+\beta+\frac{\beta^2}{2}-n\cdot \pt\,.\]
The isomorphism is given by sending
\[(E,Z)\mapsto \big(\CO_S\to I_Z(E)\big)\,.\]
The obstruction theory, and hence virtual fundamental class, of $P_\alpha^\infty$ is easily seen to match the ones defined for $S_\beta^{[0,n]}$ in \cite{gsy}. Indeed, the obstruction theory of the latter at $(E,Z)$ (see Proposition 2.2 and the proof of Proposition 3.1 in loc. cit.) is given by
\begin{align*}\textrm{Cone}\big(\RHom(\CO_S(-E), I_Z)&\to \RHom(I_Z, I_Z)\big)\\
&=\RHom(\CO_S\to I_Z(E), I_Z(E))\,.\end{align*}

\item  If $(m,d)=(2,1)$ then $P_{\alpha}^\infty$ is shown in \cite[Proposition 3.1.5]{gsy} to also be isomorphic to the nested Hilbert scheme $S_\beta^{[0,n]}$ with $\ch(\alpha)=\beta-\frac{\beta^2}{2}+n\cdot \pt$. The isomorphism sends 
\[(E, Z)\mapsto \big(\CO_S\to \CO_E(Z)\big)\,.\] 
The virtual fundamental classes of $P^\infty_\alpha$ and $S_\beta^{[0,n]}$ are also shown to agree.
\end{enumerate}

\subsection{Invariant classes $[M]^\inva$}
\label{sec: joyceclasses}
When there are strictly semistable sheaves in $M_\alpha$, we cannot obtain a class in $\widecheck V_\bullet$ by simply pushing forward a virtual fundamental class from $H_\bullet(M_\alpha)$. However, Joyce constructs classes $[M_\alpha]^{\inva}$ for every $\alpha$ such that $[M_\alpha]^{\inva}=[M_\alpha]^\vir$ when there are no strictly semistable sheaves. The classes $[M_\alpha]^{\inva}$ appear when one writes down wall-crossing formulae. We will now summarize -- and slightly reformulate -- the construction of these classes in \cite[Theorem 5.7, Section 9]{Jo21} when $(m,d)$ is one of our three cases.

First, we observe that it is enough to define the classes $[M_\alpha]^{\inva}$ when $\alpha$ satisfies Assumption \ref{ass: alphabig}. The definition extends to all $\alpha$ by requiring that $[M_{\alpha(H)}]^\inva$ is obtained from $[M_{\alpha}]^\inva$ via the map $H_\bullet(\CM_\alpha^\rig)\to H_\bullet(\CM_{\alpha(H)}^\rig)$ induced by tensoring with $\CO_X(H)$. The argument that this definition is consistent is one of the core arguments in Joyce's theory proved in \cite[Proposition 9.12]{Jo21}. 

Let $\Pi\colon \CP_X\to \CM_X^\rig$ be the composition of the projection $\CP_X\to \CM_X$ onto the second component with the rigidification map $\CM_X\to \CM_X^\rig$. We define the $K$-theory class $\T^{\rel}$ in $\CP_X$ by
\[\T^{\rel}=Rp_\ast \CF-\CO_{\CP_X}\,,\]
where $p\colon \CP_X\times X\to \CP_X$ is the projection. Then, one defines the class
\begin{equation}\label{eq: upsilonalpha}
\Upsilon_\alpha=\Pi_\ast\left( c_{\chi(\alpha)-1}(\T^{\rel})\cap [P_\alpha^{0+}]^\vir_{(\CO,\BF)} \right)\in H^\bullet(\M_\alpha^\rig)\subseteq \widecheck V_\bullet\,.
\end{equation}
Alternatively, note that if we define $\Pi_\alpha$ as the composition
\begin{center} \begin{tikzcd}
\Pi_\alpha\colon P_\alpha^{0+}\arrow[r,"f_{(\CO,\BF)}"]&
\CP_X\arrow[r, "\Pi"]&
\CM_X^\rig
\end{tikzcd}\end{center}
then by Assumption \ref{ass: alphabig} the relative tangent bundle $\T_{\Pi_\alpha}$, which has fibers 
\[H^0(F)/\BC\cdot s\textup{ over a pair }[s\colon \CO_X\to F]\in P_\alpha^{0+}\,,\] is a vector bundle of rank $\chi(\alpha)-1$ such that \[\T_{\Pi_\alpha}=f_{(\CO,\BF)}^\ast \T^{\rel}\,\]
in $K$-theory. Equation \eqref{eq: upsilonalpha} can be rewritten as 
\begin{equation}\label{eq: upsilonalpha2}
\Upsilon_\alpha=(\Pi_\alpha)_\ast\left( c_{\tp}(T_{\Pi_\alpha})\cap [P_\alpha^{0+}]^\vir \right)\in \widecheck V_\bullet\,,
\end{equation}
which is the form in \cite[(5.29)]{Jo21}.

Before we begin discussing wall-crossing, we set some notation. We will use $\underline{\alpha} = (\alpha_1,\cdots,\alpha_l)$ for the vector of $K$-theory classes and denote by $\underline{\alpha}\vdash \alpha$ the fact that it is a partition of $\alpha$, i.e., $\alpha_1+\cdots +\alpha_l = \alpha$, where $l$ always denotes the length of $\alpha$. When we write $\sum_{\underline{\alpha}\vdash \alpha}$ we mean a sum over all $\alpha_1, \ldots, \alpha_l$ such that $\alpha_1+\ldots+ \alpha_l=\alpha$.

The classes $[M_\alpha]^\inva\in \widecheck V^\bullet$ are now defined by
\begin{equation}
\Upsilon_\alpha=\sum_{\substack{\underline{\alpha}\vdash \alpha\\\mu(\alpha_i)=\mu(\alpha)}}\frac{(-1)^{l+1}\chi(\alpha_1)}{l!}\big[\big[\ldots  \big[ [M_{\alpha_1}]^\inva,[M_{\alpha_2}]^\inva \big],\ldots \big ],[M_{\alpha_l}]^\inva \big]\label{eq: upsilonwallcrossing}\,.
\end{equation}
Equation \eqref{eq: upsilonwallcrossing} provides an inductive definition of the classes $[M_{\alpha}]^\inva$ by comparing
\[\Upsilon_\alpha=\chi(\alpha)[M_\alpha]^\inva+\ldots\]
where $\ldots$ is expressed in terms of classes $[M_{\alpha_i}]^\inva$ such that $r(\alpha_i)<r(\alpha).$ Note that Assumption \ref{ass: alphabig} implies $\chi(\alpha)>0$.

\begin{remark}\label{rem: projectivebundle}
When $M_\alpha^{\textup{ss}}=M_\alpha^{\textup{s}}$ there are no decompositions $\alpha=\alpha_1+\ldots+\alpha_l$ with $l>1$, $\mu(\alpha_i)=\mu(\alpha)$ and $M_{\alpha_i}\neq \emptyset$. Hence the right hand side of \eqref{eq: upsilonwallcrossing} has only one non-zero term, so
\[\Upsilon_\alpha=\chi(\alpha)[M_\alpha]^\inva. \]
Recall that without strictly semistable sheaves $f\colon P_\alpha^{0+}\to M_\alpha$ is a projective bundle with fibers $\BP(H^0(F))\cong \BP^{\chi(\alpha)-1}$ over $[F]\in M_\alpha$ by Proposition \ref{prop: limitt0} and the discussion following; moreover, $[M_\alpha]^\inva=[M_\alpha]^\vir$ is actually the (pushforward to $H_\bullet(\M_X^\rig)$ of the) virtual fundamental class. Indeed, we have
\[f_\ast\left(c_\tp(\T_f)\cap [P_\alpha^{0+}]^\vir\right)=\chi(\alpha)[M_\alpha]^\vir\]
by the virtual pullback formula \cite[Theorem 4.7]{manolache}.
The following heuristic is useful to keep in mind: the class $\Upsilon_\alpha$ is constructed so that the intersection theories against $[P_\alpha^{0+}]^\vir$ and $\Upsilon_\alpha$ are related in a way similar to a projective bundle; the wall-crossing type formula \eqref{eq: upsilonwallcrossing} defines $[M_\alpha]^\inva$ by ``correcting'' $\Upsilon_\alpha$.
\end{remark}

Above we only discussed the case of slope stability, but Joyce defines the classes $[M_{\alpha}(\tau)]^{\inva}$ for any stability condition $\tau$ satisfying his list of assumptions \cite[Assumption 5.1 and 5.2]{Jo21}. In particular, this holds when $(m,d) = (2,2)$ and $\tau$ is the Gieseker stability, by \cite[Section 7]{Jo21}; note that Gieseker stability coincides with slope stability when $d=1$, so only the $(2,2)$ case is new. 

It is well-known that slope stability $\mu$ \textit{dominates} Gieseker stability $\tau$ in the sense of \cite[Definition 3.8]{Jo21}.  In other words -- if $p_E(z)\leq p_F(z)$ for any two sheaves $E,F$ and any $z\gg 0$, then $\mu(E)\leq \mu(F)$. The dominant wall-crossing formula stated in \cite[Theorem 5.8, Theorem 7.64]{Jo21} applies precisely to this scenario and takes the form 
$$
[M_{\alpha}(\tau)]^\inva=\sum_{\underline{\alpha}\vdash \alpha
}U(\underline{\alpha};\mu,\tau)
\big[\big[\ldots  \big[ [M_{\alpha_1}(\mu)]^\inva,[M_{\alpha_2}(\mu)]^\inva \big],\ldots \big ],[M_{\alpha_l}(\mu)]^\inva \big]
$$
where $U(\underline{\alpha};\mu,\tau)$ are the coefficients defined in  \cite[Section 3.2]{Jo21}.  Applying the compatibility of Virasoro constraints with the Lie bracket that follows from Theorem \ref{thm: virasorophysical} and Proposition \ref{prop: wallcrossingcompatibility1}, we reduce Virasoro constraints for Gieseker stability to slope stability.
\begin{corollary}
\label{cor:slopetogies}
    If Virasoro conjecture holds for $[M_{\alpha}(\mu)]^{\inva}$ for any $\alpha$, i.e., $[M_{\alpha}(\mu)]^{\inva}\in \widecheck{P}_0$, then it is also satisfied by $[M_{\alpha}(\tau)]^{\inva}$ where $\tau$ is the Gieseker stability.
\end{corollary}




\subsection{Wall-crossing formula for Bradlow stability}
\label{sec: WCformula}
\sloppy When working with pair wall-crossing formulae, we will include $\alpha_0$ into the partition $\underline{\alpha}\vdash \alpha$ by $\underline{\alpha}=(\alpha_0,\alpha_1,\cdots,\alpha_l)$ and $\alpha_0+\alpha_1+\cdots +\alpha_l = \alpha$. Joyce's wall-crossing formula (Theorem 5.9 in \cite{Jo21}) between Bradlow stability $\mu^{t_-}$ and $\mu^{t_+}$-stability is
\begin{align}[P_{\alpha}^{t_{-}}]^\vir=\sum_{\underline{\alpha} \vdash \alpha\label{eq: generalbradlowwc}
}U(&\underline{\alpha}; \mu^{t_-}, \mu^{t_+})\\\nonumber 
&\big[[M_{\alpha_1}]^\inva,\big[[M_{\alpha_2}]^\inva,\ldots, \big[[M_{\alpha_l}]^\inva,[P_{\alpha_0}^{t+}]^\vir \big],\ldots \big]\big]\label{eq: generalbradlowwc}
\end{align}
where
\[U(\underline{\alpha}; \mu^{t_-}, \mu^{t_+})=U\big((0,\alpha_1), \ldots, (0, \alpha_l), (1, \alpha_0); \mu^{t_-},  \mu^{t_+}\big)\in \BQ\]
are combinatorial coefficients defined in \cite[Section 3.2]{Jo21}. In particular we have a wall-crossing formula between the (limit) stability conditions $\mu^{0+}$ and $\mu^\infty$ 
\begin{equation}
\label{eq: rankreductionwallcrossing}[P_{\alpha}^{0+}]^\vir=\sum_{\underline{\alpha}\vdash \alpha
}U(\underline{\alpha})
\big[[M_{\alpha_1}]^\inva,\big[[M_{\alpha_2}]^\inva,\ldots, \big[[M_{\alpha_l}]^\inva,[P_{\alpha_0}^\infty]^\vir \big]\ldots \big]\big]
\end{equation}
where
\[U(\underline{\alpha})=U\big((0,\alpha_1), \ldots, (0, \alpha_l), (1, \alpha_0); \mu^{0+}, \mu^\infty\big)\in \BQ\]
Equations \eqref{eq: generalbradlowwc} and \eqref{eq: rankreductionwallcrossing} are proven as equalities in the Lie algebra $\widecheck H_\bullet(\CN_X)$ under some technical assumptions (Assumptions 5.1-5.3 in loc. cit.) on the category $\grave{\CA}$ and on a set of permissible classes $C_{\textup{pe}}(\grave{\CA})$. The necessary assumptions are all verified in Section 8 of loc. cit. for the cases $(m,d)=(1,1), (2,2)$. The case $(m,d)=(2,1)$ is not treated, but the first author plans to address this in a separate work focusing on proving them for all pair theories relevant to Joyce's wall-crossing. 

\begin{assumption}
\label{ass:WCpair}
Let $\grave{\CA}$ be the abelian category of pairs $U\otimes \CO_X\to F$, where $F$ is a sheaf with $\dim(F)\leq 1$ and $U$ a vector space.\footnote{As opposed to $\acute{\CA}$ which was used in \cite{Jo21} to denote the category of all pairs without the restriction on dimension.} 
Then the assumptions in \cite[Assumption 5.1-5.3]{Jo21} hold for this category and a surface $S$ such that $h^{0,2}(S)=0$ with
\begin{align*}
C_{\textup{pe}}(\grave{\CA})&=\{(e,\llbracket F\rrbracket)\colon e=0,1\textup{ and }\dim(F)= 1\}\subseteq \BZ\times K_\sst^0(S)\,,\\
\mathscr{S}&=\{\mu^t\colon t\in \BR_{>0}\}\,.
\end{align*}
\end{assumption}
\begin{remark}Note that the data used in \cite[Assumption 4.4]{Jo21} was constructed in Definition~\ref{def: pairVOA}. Most of the assumptions 5.1 - 5.3 are satisfied by a simple adaptation of the arguments in \cite[Sections 8.2 and 8.3]{Jo21}. New ideas are needed only for the assumptions 5.2(b) and 5.2(h)  where one needs to show that the stacks of semistable pairs are finite type and the moduli spaces of stable quiver-pairs defined in \cite[Def. 5.5]{Jo21} are proper.\end{remark}

An important point in the rank reduction induction that we will use is that there are no contributions of rank 0 objects in the wall-crossing formula above.

\begin{lemma}\label{lem: norank0}
Let $\alpha$ be such that $r(\alpha)>0$. If the coefficient $U(\underline{\alpha})$ is non-zero, then $r(\alpha_i)>0$ for each $i=0, 1, \ldots, l$. 
\end{lemma}
\begin{proof}
Let $t$ be a wall and let $t_-<t<t_+$ define stability conditions on the two chambers adjacent to the wall. We have the wall-crossing formula between $\mu^{t_-}$-stability and $\mu^{t_+}$-stability given by \eqref{eq: generalbradlowwc}. To prove the Lemma it is enough to show that the coefficients $U(\underline{\alpha}; \mu^{t_-}, \mu^{t_+})$ vanish unless $r(\alpha_i)> 0$ for every $i$, since \eqref{eq: rankreductionwallcrossing} can be obtained by putting together the $\mu^{t_-}/\mu^{t_+}$ wall-crossing formulae; in other words, \cite[Theorem 3.11]{Jo21} can be applied iteratively to compute $U(\underline{\alpha})$ in terms of $U(\underline{\alpha}; \mu^{t_-}, \mu^{t_+})$ for $t_-, t_+$ in adjecent chambers. 

Since $\mu^{t_-}$, $\mu^{t_+}$ are in the adjacent chambers to the wall defined by $t$, the stability $\mu^t$ dominates (cf. \cite[Definition 3.8]{Jo21}) both $\mu^{t_-}$ and $\mu^{t_+}$. By \cite[Theorem 3.11]{Jo21}, $U(\underline{\alpha}; \mu^{t_-}, \mu^{t_+})=0$ unless
\[\mu^t(1, \alpha_0)=\mu(\alpha_1)=\ldots=\mu(\alpha_l)=\mu^t(1, \alpha).\]
Since $r(\alpha)>0$, $\mu^t(1, \alpha)<\infty$ so it follows that $r(\alpha_i)>0$ for each $i=0, 1, \ldots, l$. \qedhere
\end{proof}

Formula \eqref{eq: rankreductionwallcrossing} holds a priori in the Lie algebra $\widecheck{H}_\bullet(\CN_X)$. As explained in Remark \ref{rem: stackNX}, we can pushforward this identity to the Lie algebra $\widecheck{H}_\bullet(\CP_X)$ or, equivalently, $\widecheck V_\bullet^{\pa}$ (recall Lemma \ref{Lem: DXinv}). However, our formulation of the Virasoro constraints for the pair moduli spaces $P_\alpha^t$ is not in terms of the class $[P_\alpha^t]^\vir\in \widecheck V_\bullet^{\pa}$, but instead of its lift $[P_\alpha^t]^\vir_{(\CO,\BF)}\in V_\bullet^{\pa}$ to the vertex algebra (see Theorem \ref{thm: virasorophysical}). Hence it is desirable to lift the formula \eqref{eq: rankreductionwallcrossing} to the vertex algebra. We use the following Lemma to do so: 

\begin{lemma}\label{lem: liftvertexalgebra}
\begin{enumerate}
\item[a)]
Suppose that $\overline{u}\in \widecheck V_\bullet\subseteq \widecheck V_\bullet^{\pa}$ and $v\in V_\bullet^{\pa}$ is such that \[\ch_1^\CV(\pt)\cap v=0\,.\] Then
\[\ch_1^\CV(\pt)\cap [\overline u, v]=0\]
where the bracket is the partial lift to the vertex algebra from Lemma \ref{partial lift}.
\item[b)] Let $u,v\in V_{\bullet, (\alpha_1, \alpha_2)}^{\pa}$ with $\rk(\alpha_1)>0$ be such that 
\[\ch_1^\CV(\pt)\cap u=0=\ch_1^\CV(\pt)\cap v\,\]
and $\overline u=\overline v$ in $\widecheck V_\bullet^{\pa}$. Then $u=v$ in $V_\bullet^{\pa}$. 
\end{enumerate}
\end{lemma}
\begin{proof}
Since $\chi_\sym^{\pa}$ is non-degenerate by Lemma \ref{Lem: nondegenerate}, there is $w\in K^\bullet(X)^{\oplus 2}$ such that \[\chi_\sym^{\pa}(w, -)=\langle \pt^\CV, -\rangle\,.\]
Comparing \eqref{Eq: explicitfield} and \eqref{eq: capdescendentspairs} it follows that $\ch_1^\CV(\pt)\cap -=w_{(1)}$. Using the identity \eqref{Eq: skewjacobi} we have
\[\ch_1^\CV(\pt)\cap [\overline u, v]=w_{(1)}(u_{(0)}v)=u_{(0)}(w_{(1)} v)+(w_{(0)} u)_{(1)} v-(w_{(1)} u)_{(0)} v.\]
By hypothesis $w_{(1)} v=0$. Since $u\in V_\bullet$ both $w_{(0)} u=w_{(1)} u=0$ as the pullbacks of $\ch_1^\CV(\pt), \ch_0^\CV(\pt)$ to $H^\bullet(\CM_X)$ both vanish.

For the second part, suppose that $u-v=T(x)$. Consider the operator
\[\eta^\dagger=\sum_{j\geq 0}\frac{(-1)^j}{j!\rk(\alpha_1)^j}T^j\circ \left(\ch_1^\CV(\pt)^j\cap -\right)\]
on $V_{\bullet, (\alpha_1, \alpha_2)}^{\pa}$; note that this is dual to the operator $\eta$ from Remark \ref{rem: rationaldeltanormalized}. It follows from Lemma \ref{lem: translationdualr-1} that
\[[\ch_1^\CV(\pt)\cap -, T]=\ch_0^\CV(\pt)\cap -=\rk(\alpha_1)\,,\]
from which it is easy to show that $\eta^\dagger\circ T=0$. Thus, by the assumption,
\[u-v=\eta^\dagger(u-v)=\eta^\dagger(T(x))=0\,.\qedhere\]


\end{proof}

We can use the lemma to lift the previous wall-crossing formula to an equality
\begin{align}[P_{\alpha}^{0+}]^\vir_{(\CO,\BF)}=\sum_{\underline{\alpha} \vdash \alpha
}&U(\underline{\alpha})
&\big[[M_{\alpha_1}]^\inva,\big[[M_{\alpha_2}]^\inva,\ldots, \big[[M_{\alpha_l}]^\inva,[P_{\alpha_0}^\infty]^\vir_{(\CO,\BF)}  \big],\ldots \big]\big]
\label{eq: wallcrossinglift}
\end{align}
holding in $V_\bullet^{\pa}$, where all the brackets on the right hand side are the partial lift of the Lie bracket to the vertex algebra. To deduce this from the lemma we note that
\[\ch_1^\CV(\pt)\cap [P_{\alpha_0}^t]^\vir_{(\CO,\BF)}=0\]
since the pullback of $\ch_1^\CV(\pt)$ to $P_{\alpha_0}^t$ is $\xi_\CO(\ch_1(\pt))=0.$ By part a) of the lemma the right hand side is also annihilated by $\ch_1^\CV(\pt)$ and by part b) we must have equality -- both sides live in $V^{\pa}_{\bullet, (1, \alpha)}$ and their classes in $\widecheck V_\bullet^{\pa}$ agree by \eqref{eq: rankreductionwallcrossing}.

\subsection{Rank reduction of Virasoro}

We can now explain the rank reduction argument for proving Virasoro on $M_\alpha$ assuming that it holds for the stable pair moduli space $P_\alpha^\infty$. We described these spaces explicitly in Section \ref{sec: limits}.

\begin{theorem}\label{thm: rankreduction}

Working with $(m,d)= (1,1),(2,2)$, suppose that the pair Virasoro constraints (Conjecture \ref{conj: pairvirasoro}) hold for $P_\alpha^\infty$ for every $\alpha$ with $r(\alpha)> 0$.
Then the Virasoro conjecture holds for $M_\alpha$, $P^t_{\alpha}$ for every $\alpha$ with $r(\alpha)>0$ and $t\in [0+, \infty]$, i.e.,
\[[M_\alpha]^\inva \in \widecheck P_0\quad \textup{ and }\quad [P_\alpha^t]^\inva \in P_0^{\pa}\,.\]
If the Assumption \ref{ass:WCpair} holds, then the above is true also for $(m,d) =(2,1)$. 
\end{theorem}

The strategy of the proof is quite simple: we will argue by induction on $r(\alpha)$ and we will prove (assuming the induction hypothesis) that
\[[P_\alpha^\infty]^\vir_{(\CO, \BF)}\in P_0^{\pa}\overset{(\textup{I})}{\Longrightarrow} [P_\alpha^{0+}]^\vir_{(\CO, \BF)}\in P_0^{\pa}\overset{(\textup{II})}{\Longrightarrow} \Upsilon_\alpha\in \widecheck P_0\overset{(\textup{III})}{\Longrightarrow} [M_\alpha]^{\inva}\in  \widecheck P_0\]
Implications $(\textup{I})$ and $(\textup{III})$ will follow from \eqref{eq: wallcrossinglift}, \eqref{eq: upsilonwallcrossing} and the compatibility between wall-crossing and Virasoro constraints proven in Propositions \ref{prop: wallcrossingcompatibility1} and \ref{prop: wallcrossingcompatibility2}.

The implication $(\textup{II})$ is a projective bundle compatibility. We will postpone its proof until the next section, see Theorem \ref{lem: projectivebundle}, and prove Theorem \ref{thm: rankreduction} assuming it.

\begin{proof}[Proof of Theorem \ref{thm: rankreduction}]
We argue by induction on $r(\alpha)$. The base case is when $r(\alpha)>0$ is minimal and is dealt essentially in the same way as the induction step. Assume then that $[M_{\alpha'}]^\inva \in \widecheck P_0$ for every $\alpha'$ such that $0<r(\alpha')<r(\alpha)$; note in particular that this holds vacuously if $r(\alpha)$ is minimal. 

To prove implication $(\textup{I})$ we consider the wall-crossing formula \eqref{eq: wallcrossinglift} and we look at each individual summand 
\[\big[[M_{\alpha_1}]^\inva,\big[[M_{\alpha_2}]^\inva,\ldots, \big[[M_{\alpha_l}]^\inva,[P_{\alpha_0}^\infty]^\vir_{(\CO,\BF)}  \big],\ldots \big]\big]\]
with $\sum_{i=0}^l \alpha_i=\alpha$ and non-vanishing coefficient $U(\underline{\alpha})\neq 0$. By hypothesis, $[P_{\alpha_0}^\infty]^\vir_{(\CO, \BF)} \in P_{0}^{\pa}$. Moreover by Lemma \ref{lem: norank0} we have for $i=1, \ldots, l$ that
\[0<r(\alpha_i)<r(\alpha_i)+r(\alpha_0)\leq r(\alpha)\,.\]
So the induction hypothesis applies and
\[[M_{\alpha_i}]^{\mathrm{inv}}\in \widecheck P_0\quad \textup{ for}\quad i=1, \ldots, l\,.\]
By Propositions \ref{prop: wallcrossingcompatibility1} and \ref{prop: wallcrossingcompatibility2} we have 
\[\big[[M_{\alpha_1}]^\inva,\big[[M_{\alpha_2}]^\inva,\ldots, \big[[M_{\alpha_l}]^\inva,[P_{\alpha_0}^\infty]^\vir_{(\CO,\BF)}  \big],\ldots \big]\big]\in P_{0}^{\pa}\]
so $[P_\alpha^{0+}]^\vir_{(\CO, \BF)} \in P_{0}^{\pa}$. The exact same argument shows that $[P_\alpha^{t}]^\vir_{(\CO, \BF)} \in P_{0}^{\pa}$ for any $t>0$: just replace the ${0+}/\infty$ wall-crossing formula \eqref{eq: wallcrossinglift} by the $t/\infty$ wall-crossing. 

For both implications $(\textup{II})$ and $(\textup{III})$ we will assume that $\alpha$ satisfies Assumption~\ref{ass: alphabig}. This is enough to prove the result for every $\alpha$ since we may replace $\alpha$ by $\alpha(mH)$ for large enough $m$ so that $\alpha(mH)$ satisfies the assumption. As explained in Section~\ref{sec: joyceclasses}, Joyce classes $[M_\alpha]^\inva, [M_{\alpha(mH)}]^\inva$ are related by the automorphism $H_\bullet(\CM_X)~\to~H_\bullet(\CM_X)$ induced by $-\otimes \CO_X(mH)$. This automorphism preserves physical states by Lemma \ref{lem: twist}. 

The implication $(\textup{II})$ is precisely Theorem \ref{lem: projectivebundle}. So we are left with implication $(\textup{III})$. For that, we use \eqref{eq: upsilonwallcrossing} and induction as in $(\textup{I})$. The left hand side of \eqref{eq: upsilonwallcrossing} is $\Upsilon_\alpha\in \widecheck P_0$. The right hand side is the sum of the leading term $\chi(\alpha)[M_\alpha]^\inva$ with terms of the form
\[\big[\big[\ldots  \big[ [M_{\alpha_1}]^\inva,[M_{\alpha_2}]^\inva \big],\ldots \big ],[M_{\alpha_l}]^\inva \big]\]
with $l\geq 2$, $\underline{\alpha}\vdash \alpha$ and $\mu(\alpha_i)=\mu(\alpha)$. The latter condition implies that $r(\alpha_i)>0$, and since $l\geq 2$ it follows that $0<r(\alpha_i)<r(\alpha)$. Thus the induction hypothesis guarantees that $[M_{\alpha_i}]^\inva\in \widecheck P_0$, so by Proposition \ref{prop: wallcrossingcompatibility1} we get 
\[\big[\big[\ldots  \big[ [M_{\alpha_1}]^\inva,[M_{\alpha_2}]^\inva \big],\ldots \big ],[M_{\alpha_l}]^\inva \big]\in \widecheck P_0\,.\]
Finally this implies that the leading term also satisfies Virasoro, i.e., \[[M_\alpha]^\inva\in \widecheck P_0\]
since $\chi(\alpha)>0$ by Assumption \ref{ass: alphabig}. \qedhere
\end{proof}

\subsection{Projective bundle compatibility}
 We recall the reader of the notation $\Pi, \Pi_\alpha, \T^{\rel}, \T_{\Pi_\alpha}$ introduced in Section \ref{sec: joyceclasses}. For ease of notation, we omit the superscript ``pa'' from now on when it is clear from the context that we use the pair Virasoro operators.


\begin{theorem}\label{lem: projectivebundle}
Let $\alpha$ be a class satisfying Assumption \ref{ass: alphabig}. Suppose that $P^{0+}_\alpha$ satisfies the pair Virasoro constraints (Conjecture~\ref{conj: pairvirasoro}), i.e., $[P^{0+}_\alpha]^\vir_{(\CO, \BF)}\in P^{\pa}_{0}$. Then 
\[\Upsilon_\alpha=\Pi_\ast\big(c_{\chi(\alpha)-1}(\T^{\rel})\cap [P^{0+}_\alpha]^\vir_{(\CO, \BF)}\big)\in \widecheck P_0\]
satisfies the sheaf Virasoro constraints.
\end{theorem}

\begin{proof}
We ought to show that $\int_{\Upsilon_\alpha}\Linv(D)=0$ for any $D\in \BD^X_{\alpha}$ where we use the suggestive integral notation to denote the pairing between a homology class $\Upsilon_\alpha~\in~H_\bullet(\M^\rig_\alpha)$ and a cohomology class $\Linv(D)\in \BD^X_{\inv, \alpha}\cong H^\bullet(\M^\rig_\alpha).$ 

We have the following commutative diagram:
\begin{center}
\begin{tikzcd}[row sep=large, column sep=large]
H^\bullet(P_\alpha^{0+})&
H^\bullet(\CP_{(1,\alpha)})\arrow[l,swap, "(f_{(\CO,\BF)})^\ast"]&
H^{\bullet}(\CM_\alpha^{\rig})\arrow[l, swap, "\Pi^\ast"]\arrow[ll, bend right, swap, "\Pi_\alpha^\ast"]\\
\,&
\BD_{(1, \alpha)}^{X,\pa}\arrow[u, "\sim"'{rotate=90, above},swap, "\xi_{(\CV,\CF)}"]\arrow[ul, "\xi_{(\CO,\BF)}"] \arrow[r, hookleftarrow]&
\BD_{\inv, \alpha}^X\arrow[llu, bend left=50, "\xi_\BF"]\arrow[u, "\sim"'{rotate=90, above},swap, "\xi_{\CF}"]
\end{tikzcd}
\end{center}

Hence we can compute:
\begin{align}
\nonumber \int_{\Upsilon_\alpha}\Linv(D)&=\int_{(\Pi_\alpha)_\ast\left(c_{\tp}(\T_{\Pi_\alpha})\cap [P_\alpha^{0+}]^\vir\right)}\Linv(D)\\
\nonumber &
=\int_{[P_\alpha^{0+}]^\vir}\Pi^\ast_\alpha \big(\Linv(D)\big)c_{\tp}(\T_{\Pi_\alpha})\\
 &=\int_{[P_\alpha^{0+}]^\vir}\xi_\BF \left(\Linv(D)c_{\chi(\alpha)-1}\right)\nonumber\\
 &=\sum_{j\geq -1}\frac{(-1)^j}{(j+1)!}\int_{[P_\alpha^{0+}]^\vir}\xi_\BF \left(\bL_j(\bR_{-1}^{j+1}D)c_{\chi(\alpha)-1}\right) 
 \label{eq: integralUpsilon}
\end{align}

We use $c_{\chi(\alpha)-1}\in \BD_\alpha^X$ to denote the element in the algebra of descendents
\[c_{\chi(\alpha)-1}\coloneqq c_{\chi(\alpha)-1}(Rp_\ast \CF)\in H^\bullet(\CM_\alpha)\cong \BD_\alpha^X\,.\]
The penultimate equality is using that 
\[\T_{\Pi_\alpha}=(f_{(\CO,\BF)})^\ast \T^{\rel} =(f_{(\CO,\BF)})^\ast\big(Rp_\ast \CF-\CO\big)=Rp_\ast \BF-\CO\,.\]

We can calculate $c_{\chi(\alpha)-1}$ more explicitly: by Grothendieck-Riemann-Roch we have
\[\ch\big(Rp_\ast \CF\big)=p_\ast\big(\ch(\CF)\td(X)\big)=\xi_\CF\big(\ch_\bullet(\td(X))\big)\,;\]
and by Newton identities it follows that
\[c(Rp_\ast \CF)=\xi_\CF\left(\exp\left(\sum_{\ell\geq 1}(-1)^{\ell-1}(\ell-1)!\ch_\ell(\td(X))\right)\right).\]

We denote by $c$ the corresponding element in the algebra of descendents and we let $c_i$ be the degree $i$ part of $c$, i.e.,
\[\sum_{i\geq 0}c_i=c=\exp\left(\sum_{\ell\geq 1}(-1)^{\ell-1}(\ell-1)!\ch_\ell(\td(X))\right)\,.\]

Let $n=\chi(\alpha)-1$. We shall now argue that the integral at the end of \eqref{eq: integralUpsilon} vanishes assuming that $P_\alpha^{0+}$ satisfies the pair Virasoro constraints. For convenience, from now on we leave implicit the geometric realization map $\xi_\BF$ in all the integrals against $[P_\alpha^{0+}]^\vir$.

We start with the $j=-1$ and $j=0$ terms in the last line of \eqref{eq: integralUpsilon}, which we treat together.  Their sum vanishes by simple degree considerations:
\begin{align}\nonumber \int_{[P^{0+}_\alpha]^\vir} \big(\bL_0 \bR_{-1}(D)-\bR_{-1}(D)\big) c_{n}&= \int_{[P^{0+}_\alpha]^\vir} \big(\bL_0-\id-n\id)(c_n\bR_{-1}(D))\\
&=\int_{[P^{0+}_\alpha]^\vir} \bL_0^\CO(c_n\bR_{-1}(D))=0\,.\label{eq: j0j-1}
\end{align}
Note that we used
\[\xi_{\BF}(\ch_0^\HH(\td(X)))=\int_X \ch(\alpha)\td(X)=\chi(\alpha)=n+1\,.\]

Consider now $j\geq 1$. By the Virasoro constraints on $P^{0+}_\alpha$ we have
\begin{align}\nonumber 0=&\int_{[P_\alpha^{0+}]^\vir}\bL_j^{\CO}\big(\bR_{-1}^{j+1}(D)c_{n}\big)=\int_{[P_\alpha^{0+}]^\vir}\bL_j\big(\bR_{-1}^{j+1}(D)\big)c_{n}\\
&+\int_{[P_\alpha^{0+}]^\vir}\bR_{-1}^{j+1}(D)\bR_j\big(c_{n}\big)-j!\int_{[P_\alpha^{0+}]^\vir} \ch_j(\td(X))\bR_{-1}^{j+1}(D)c_{n}. \label{eq: largejI}
\end{align}
We analyze the term where the derivation $\bR_j$ applies to $c_{n}$; we may do so using the interaction between an exponential and a derivation:
\[\bR_j(c)=\left(\sum_{\ell \geq 1}(-1)^{\ell-1} (\ell+j)! \ch_{\ell+j}(\td(X))\right)c\,,\]
so
\[\bR_j(c_{n})=\sum_{\substack{a\geq j+1, b\geq 0\\
a+b=n+j}} (-1)^{a-j-1} a!\ch_{a}(\td(X))c_{b}.\]
Using the Newton identities and the fact that $\T_{\Pi_\alpha}$ is a vector bundle of rank $n$, the geometric realization of the above is
\begin{align}\nonumber \label{eq: Trelvectorbundle}\xi_\BF(\bR_j(c_{n}))&=\sum_{\substack{a\geq j+1, b\geq 0\\a+b=n+j}} (-1)^{a-j-1} a!\ch_{a}(\T_{\Pi_\alpha})c_{b}(\T_{\Pi_\alpha})=j!\ch_{j}(\T_{\Pi_\alpha})c_{n}(\T_{\Pi_\alpha})\\
&=\xi_\BF\big(j! \ch_j(\td(X))c_n\big)\,.
\end{align}
It follows that the two last terms in \eqref{eq: largejI} cancel out and we are left with 
\begin{equation}
\label{eq: largejII} \int_{[P_\alpha^{0+}]^\vir}\bL_j\big(R_{-1}^{j+1}(D)\big)c_{n}=0
\end{equation}
for $j\geq 1$. Using \eqref{eq: j0j-1} and \eqref{eq: largejII} for every $j\geq 1$ we have shown that the last integral in \eqref{eq: integralUpsilon} vanishes, and we are done. \qedhere
\end{proof}

\begin{remark}\label{rem: abstractprojective}
We can formulate the previous Theorem more abstractly as follows: let $u\in V_{\bullet, (1, \alpha)}^{\pa}$ be such that 
\begin{enumerate}
\item $\ch_i^\CV(\gamma)\cap u=0$ for $i>0, \gamma\in H^\bullet(X)$;
\item $c_{b}(Rp_\ast \CF)\cap u=0$ for $b\geq \chi(\alpha)$.
\end{enumerate}
Then
\[u\in P_0^{\pa}\Rightarrow \Pi_\ast\big(c_{\chi(\alpha)-1}(\T^{\rel})\cap u\big)\in \widecheck P_0\,.\]
The first condition is used when formulating the pair Virasoro constraints in terms of $\bL_k^\CO$ as in Conjecture \ref{conj: pairvirasoro}. Condition (2) was used in \eqref{eq: Trelvectorbundle}; in the setting of the Lemma, it is a consequence of the fact that $\T_{\Pi_\alpha}$ is a vector bundle of rank $\chi(\alpha)-1$.
\end{remark}

\section{Virasoro for $P_\alpha^\infty$}
\label{sec:lowrank}

In this section, we finish the proof of the main result of this paper:

\begin{theorem}\label{thm: main result of the paper}
    Let $X$ be a curve or a surface with $h^{1,0}=h^{2,0}=0$. Then the moduli space $M$ of slope or Gieseker semistable torsion-free sheaves on $X$ satisfies the Virasoro constraints, i.e., 
    \[[M]^{\inva}\in \widecheck{P}_0\,.\]
    Under Assumption \ref{ass:WCpair}, the same statement holds for the moduli spaces of slope semistable one-dimensional sheaves on such surfaces.
\end{theorem}
When there are no strictly semistable sheaves, the above statement proves Conjecture \ref{conj: virasorosheaves} in the cases $(m,d) = (1,1), (2,2)$ and $(2,1)$ under Assumption \ref{ass:WCpair} as shown in Section \ref{subsec: virasoroprimary}. More generally, this theorem provides strong evidence for Conjecture \ref{conj:actualinvvir} even when there exist strictly semistable sheaves.

In order to prove Theorem \ref{thm: main result of the paper}, we are left to prove the Virasoro constraints for $P_\alpha^\infty$ by Theorem \ref{thm: rankreduction}. This is what we now proceed to do. Recall that we explained in Section \ref{sec: limits} that these moduli spaces are: 
\begin{enumerate}
    \item symmetric powers $C^{[n]}$ for curves,
    \item nested Hilbert scheme $S_\beta^{[0,n]}$ (both in the torsion-free and torsion cases).
\end{enumerate}
In the case of the symmetric power of curves, we give an elementary and direct proof in Proposition \ref{prop: symmetricpowers}.

Recall from Section \ref{sec: limits} that when $\rk(\alpha)=1$, there is an identification $$P_\alpha^\infty=P_\alpha^{0+}$$
so it is enough to prove Virasoro constraints for the latter moduli space in this case. Under Assumption \ref{ass: alphabig}, the natural map $P_{\alpha}^{0+}\to M_{\alpha}$ is a projective bundle -- see Remark \ref{rem: projectivebundle}. For a general $\alpha$ the map is a virtual projective bundle, see Section~\ref{sec: virtualprojectivebundle}. This structure can be used to show the equivalence between Virasoro constraints on $P_\alpha^{0+}$ and $M_\alpha$. Indeed, the implication from $P_\alpha^{0+}$ to $M_\alpha$ was already used in the general rank reduction argument (Theorem \ref{lem: projectivebundle}). We will prove the implication from $M_\alpha$ to $P_{\alpha}^{0+}$ in Corollary \ref{cor: MimpliesP}. It will follow from the formula \eqref{eq: projectivebundlebracket}, expressing the class of $P_{\alpha}^{0+}$ in terms of $M_\alpha$ using Joyce's Lie bracket. Although this formula can be interpreted as a wall-crossing formula (see Remark \ref{rem: JSprojectivebundle}) we will give a direct proof from the definition of the Lie bracket.

Concretely, for surfaces, this virtual projective bundle is identified with the map
\[S_\beta^{[0,n]}\to S^{[n]}\]
that forgets the divisor $E$ and remembers the $0$-dimensional subscheme $Z$. Virasoro constraints were proven for Hilbert schemes of surfaces with $h^{0,1}=0$ in \cite{moreira,moop}. In Section \ref{sec: nestedhilbertscheme} we use this to prove the constraints for nested Hilbert schemes and finish the proof of Theorem \ref{thm: main} $(2),(3)$. 

Finally, we point out that some of the results in this section can alternatively be shown using Joyce-Song wall-crossing, which we discuss in the Appendix. In particular, Proposition \ref{prop: symmetricpowers} is a special case of the more general Theorem \ref{prop: quot}.

\subsection{Symmetric powers of curves}
\label{sec: symmetricpowers}
Symmetric powers $C^{[n]}$ parametrize divisors $E\subseteq C$ of degree $n$ or, equivalently, non-zero maps of the form $\CO_C\to \CO_C(E)$. They come equipped with a universal pair
\[\CO_{C^{[n]}\times C}\to \CO_{C^{[n]}\times C}( \CE)\]
where $\CE\subseteq C^{[n]}\times C$ is the universal divisor.

Let 
\[\{e_j\}_{j=1}^g\subseteq H^{0,1}(C)\,,\quad \{f_j\}_{j=1}^g\subseteq H^{1,0}(C)\]
be basis such that
\[\int_X f_je_i=-\int_X e_if_j=\delta_{ij}.\]
\begin{proposition}
\label{prop: symmetricpowers}
Let $C$ be a curve and $n\geq 1$. Then the pair Virasoro conjecture (cf. Conjecture \ref{conj: pairvirasoro}) holds for the symmetric powers $C^{[n]}$, i.e., $\big[C^{[n]}\big]_{(\CO, \CO(\CE))}\in V_\bullet^{\pa}$.
\end{proposition}

\begin{proof}

Let $f\colon C^{n}\to C^{[n]}$ be the projection from the $n$-fold product $C^n=C^{\times n}$ to the symmetric product; $f$ is a finite morphism of degree $n!$. The pullback via $f$ of the universal pair \[\CO_{C\times C^{[n]}}\to \CO_{C\times C^{[n]}}(\mathcal E)\]
to $C\times C^n$ is
\[\CO_{C^{n}\times C}\to \CO_{C^{n}\times C}( \Delta)\]
where $\Delta=\sum_{i=1}^n \Delta_i$ and $\Delta_i\subseteq C^n \times C$ is the pullback of the class of the diagonal $C\subseteq C\times C$ via the projection onto coordinates $i$ and~$n+1$. By the push-pull formula we have
\[\int_{C^n} \xi_{\CO(\Delta)}(D)=n!\int_{C^{[n]}}\xi_{\CO(\mathcal E)}(D).\]

Thus Virasoro for symmetric powers may be formulated entirely as a relation among integrals in $C^n$. Let us denote by $\alpha_i\in H^\bullet(C^n)$ the pullback of a class $\alpha\in H^\bullet(C)$ via projection onto the $i$-th coordinate with $i=1,\ldots, n$. We compute descendents in $H^\bullet(C^n)$; in the formulae below and from now on, we omit the geometric realization morphism $\xi_{\CO(\Delta)}$:
\begin{align*}
\ch_k^\HH(\pt)&=\sum_{I\subseteq [n]\,, \,|I|=k}\prod_{i\in I}\pt_i=\frac{1}{k!}\eta^k\\
\ch_k^\HH(1)&=n\ch_k^\HH(\pt)-\theta \ch_{k-1}^\HH(\pt)=n\frac{ \eta^k}{k!}-\frac{\theta\eta^{k-1}}{(k-1)!}\\
\ch_k^\HH(e_j)&=\ch_0^\HH(e_j)\ch_k^\HH(\pt)=\ch_0^\HH(e_j)\frac{\eta^k}{k!}\\
\ch_k^\HH(f_j)&=\ch_1^\HH(f_j)\ch_{k-1}^\HH(\pt)=\ch_1^\HH(f_j)\frac{\eta^{k-1}}{(k-1)!}
\end{align*}
where 
\begin{align*}
    \ch_0^\HH(e_j)=\sum_{i=1}^n e_{ji}\,&, \quad
    \ch_1^\HH(f_j)=\sum_{i=1}^n f_{ji}\, , \\
   \theta=\sum_{j=1}^g\ch_1^\HH(f_j)\ch_0^\HH(e_j)\, &, \quad\eta=\ch_1^\HH(\pt)=\sum_{i=1}^n \pt_i\,. 
\end{align*}

The formulae above show that the geometric realization map factors through the ring 
\[\widetilde\BD^C=\BC[\eta, \{\ch_0^\HH(e_j)\}_{j=1}^g, \{\ch_1^\HH(f_j)\}_{j=1}^g],\]
formally generated by symbols $\eta, \ch_0^\HH(e_j), \ch_1^\HH(f_j)$. Moreover the Virasoro operators are well defined on $\widetilde\BD^C$. Indeed, define
\[\widetilde \bL_k=\widetilde \bR_k+\widetilde \bT_k^\CO\colon \widetilde\BD^C\to \widetilde\BD^C\]
as follows:
\begin{enumerate}
    \item $\widetilde \bR_k$ is a derivation on $\widetilde \BD^C$ defined on generators by
\[
        \widetilde \bR_k(\eta)=\eta^{k+1}\, , \quad     \widetilde \bR_k(\ch_0^\HH(e_j))=0\, , \quad \widetilde \bR_k(\ch_1^\HH(f_j))=(k+1)\eta^k\ch_1^\HH(f_j)\,.
        \]
    \item $\widetilde \bT_k^\CO$ is multiplication by the element
    \[(1-g)k\eta^k-n \eta^k+k \theta \eta^{k-1}\in \widetilde \BD^C\,.\]
    \end{enumerate}

    \begin{claim}
    The following square commutes:
    \begin{center}
    \begin{tikzcd}
    \BD^C\arrow[r]\arrow[d, "\bL_k^\CO"]&
    \widetilde \BD^C \arrow[d, "\widetilde \bL_k^\CO"]\\
    \BD^C\arrow[r]&
    \widetilde \BD^C\arrow[r]&
    H^\bullet(C^n)
    \end{tikzcd}
    \end{center}
    \end{claim}
    \begin{proof}
    The proof is a straightforward computation.\qedhere
    \end{proof}
    
    We now take an element
    \[D=\eta^\ell \prod_{j=1}^g\ch_1^\HH(f_j)^{a_j} \ch_0^\HH(e_j)^{b_j}\in \widetilde \BD^C\,.\]
Then we have
\begin{equation}
\label{eq: symvirasorooutput}
    \widetilde \bL_k^\CO(D)=(\ell+(k+1)a)\eta^k \nonumber D+(1-g)k \eta^k D-m \eta^k D+k \theta \eta^{k-1}D
\end{equation}
where $a=\sum_{j=1}^g a_j$. By degree reasons, the integral of $\bL_k^\CO(D)$ vanishes unless $k+a+\ell=n$; when that is the case, it simplifies to
\[\widetilde \bL_k^\CO(D)=
k(a-g)\eta^k D+k \theta\eta^{k-1}D.\]
To finish the proof we are required to show that 
\begin{equation}
(g-a)\int_{C^n} \eta^k D=\int_{C^n} \eta^{k-1}\theta D.    \label{eq: finishsymmetric}
\end{equation}
We use the following easy claim:
\begin{claim}
The integral 
\[\int_{C^n} \eta^{k+\ell} \prod_{i=1}^g\ch_1^\HH(f_j)^{a_j} \ch_0^\HH(e_j)^{b_j}\]
vanishes unless $a_j=b_j\in \{0,1\}$ for every $j=1, \ldots, g$ and $k+\ell+\sum_{j=1}^g a_j=n$. In that case, the integral is equal to 
\[\int_{C^n}\eta^n=n!\]
\end{claim}
By the claim we may assume that $a_j=b_j\in \{0,1\}$, otherwise both sides of \eqref{eq: finishsymmetric} vanish. Letting $J=\{1\leq j\leq g\colon a_j=1\}$ we have
\begin{align*}
\int_{C^n}\eta^{k+\ell-1}&\theta \prod_{j\in J}\ch_1^\HH(f_j)\ch_0^\HH(e_j)\\
&=\sum_{t=1}^g \int_{C^n}\eta^{k+\ell-1}\ch_1^\HH(f_t)\ch_0^\HH(e_t)\prod_{j\in J}\ch_1^\HH(f_j)\ch_0^\HH(e_j)\\
&=\sum_{t\in [g]\setminus J} \int_{C^n}\eta^{k+\ell-1}\ch_1^\HH(f_t)\ch_0^\HH(e_t)\prod_{j\in J}\ch_1^\HH(f_j)\ch_0^\HH(e_j)\\
&=(g-|J|)n!=(g-a)\int_{C^n} \eta^{k+\ell} \prod_{j\in J}\ch_1^\HH(f_j)\ch_0^\HH(e_j)
\end{align*}
showing \eqref{eq: symvirasorooutput} and concluding the proof.\qedhere
\end{proof}



\subsection{Virtual projective bundle compatibility}
\label{sec: virtualprojectivebundle}

Let $M=M_\alpha$ be a moduli space with a universal sheaf $\BG$ as in Section \ref{sec: modulisheavespairs}, without strictly semistable sheaves. If $H^{\geq 1}(G)=0$ for every sheaf $[G]\in M$, then $Rp_\ast \BG=p_\ast \BG$ is a vector bundle of rank~$\chi(\alpha)$ and we may form the projective bundle
\[f\colon P=\BP_M(Rp_\ast \BG)\to M\,.\]
This projective bundle is naturally a moduli space of pairs: it parametrizes non-zero pairs of the form $\CO_X\to F$ such that $[F]\in M$. 

More generally, H. Park considers in \cite[Section 4]{park} the situation in which $H^{\geq 2}(G)~=~0$. In this case, we have a \textit{virtual projective bundle} $f\colon P\to M$
where 
\[P=\BP_M(Rp_\ast\BG)\coloneqq \textup{Proj Sym}^\bullet h^0((Rp_\ast \BG)^\vee)\,.\]
The morphism $f$ comes equipped with a natural relative perfect obstruction theory. By \cite{manolache}, there is a virtual pullback $f^!\colon A_\bullet(M)\to A_\bullet(P)$ between Chow groups. It is easily seen that the sheaf obstruction theory on $M$, the pair obstruction theory on $P$ and the relative obstruction theory on $f$ are compatible in the sense of \cite[Corollary 4.9]{manolache}, hence
\[[P]^\vir=f^![M]^\vir\,.\]
The moduli space $P$ comes equipped with a unique universal pair
\[\CO_{P\times X}\to \BF\coloneqq f^\ast\BG(1)\,.\]
Note that $\BF$ does not depend on the choice of $\BG$.

The virtual pullback relation between the virtual fundamental classes can be translated to Joyce's vertex algebra framework as follows:
\begin{proposition}\label{thm: projectivebundleformula}
Let $f\colon P\to M$ be a virtual projective bundle as described before. Then we have:
\begin{align}
\label{eq: projectivebundlebracket}[P]^\vir_{(\CO,\BF)}&=\big[[M]^\vir, e^{(1,0)}\big]\,,\\
\label{eq: projectivebundleupsilon}\chi(\alpha)[M]^\vir&=\Pi_\ast\left(c_{\chi(\alpha)-1}(\T^{\rel})\cap [P]^\vir_{(\CO,\BF)}\right)\,.
\end{align}
In the first formula, the bracket is the partial lift to the vertex algebra in Lemma \ref{partial lift} and $e^{(1,0)}$ is the class of the point $\{(\CO_X, 0)\}$ in $H_0(\CP_{(1, 0)})\subseteq V_\bullet^{\pa}$.
\end{proposition}
\begin{proof}
The second statement \eqref{eq: projectivebundleupsilon} is a consequence of \cite[Theorem 0.5 (2)]{park}. For the first formul we give a direct proof straight from the definition of the bracket. We do so by evaluating both sides against descendents \[D\in \BD^{X, \pa}_{(1, \alpha)}\cong H^\bullet(\CP_{(1,\alpha)})\,.\] By a similar argument to the one in Lemma \ref{lem: liftvertexalgebra} a) it is enough to consider $D\in \BD^X_\alpha\subseteq \BD^{X, \pa}_{(1, \alpha)}$ since otherwise both sides would vanish. 

We start with the left side:
\begin{align}\nonumber
\int_{[P]^\vir_{(\CO,\BF)}}D&=\int_{[P]^\vir}\xi_{\BF}(D)=\int_{f^![M]^\vir}\xi_{f^\ast\BG(1)}(D)=\deg\Big(f_\ast\big(\xi_{f^\ast \BG(1)}(D)\cap f^![M]^\vir)\big)\Big)\,.
\end{align}

By Lemma \ref{lem: changeuniversalsheaf} we have
\[\xi_{f^\ast \BG(1)}(D)=\sum_{j\geq 0}\frac{1}{j!}f^\ast\xi_{\BG}(\bR_{-1}^j D)c_1(\CO(1))^j\]
and the argument in the proof of Proposition 4.2 in \cite{park} shows that
\[f_\ast\big(c_1(\CO(1))^j \cap f^![M]^\vir\big)=s_{j-\chi(\alpha)+1}(Rp_\ast\BG)\cap [M]^\vir\,,\]
where $s_{i}(Rp_\ast \BG)=c_i(-Rp_\ast \BG)$ are the Segre classes of $Rp_\ast \BG$. Putting everything together, we find
\begin{equation} 
\label{eq: projleft}
\int_{[P]^\vir_{(\CO,\BF)}}D=\sum_{j\geq 0}\frac{1}{j!} \int_{[M]^\vir} \xi_{\BG}(\bR_{-1}^jD)s_{j-\chi(\alpha)+1}(Rp_\ast \BG)\,.
\end{equation}

The analogous formula for the pairing with the right hand side can be deduced directly from Joyce's definition of the fields \eqref{eq: joycefields}. Since $[M]^\vir_\BG$ is a lift of $[M]^\vir$ we can compute the bracket by
\[\big[[M]^\vir, e^{(1,0)}\big]=\Res_{z=0}Y([M]^\vir_\BG, z)e^{(1,0)}\,.\]
Recall that $\bR_{-1}$ is dual to $T$ and note that the pullback of $\Theta^{\pa}$ to $M$ via the map 
\[M\cong  M\times \{(\CO_X, 0)\}\to \CP_{(0,\alpha)}\times \CP_{(1,0)}\]
is precisely $-Rp_\ast \BG$. Using these two facts one checks that
\begin{align}\label{eq: projright}
\int_{Y([M]^\vir_\BG, z)e^{(1,0)}}D&=\sum_{j,i\geq 0}\frac{z^{j-i-\chi(\alpha)}}{j!} \int_{[M]^\vir} \xi_{\BG}(\bR_{-1}^jD)c_{i}(-Rp_\ast \BG)\,.
\end{align}
Clearly taking the residue in \eqref{eq: projright} gives \eqref{eq: projleft}, finishing the proof of \eqref{eq: projectivebundlebracket}.
\end{proof}

As a result, we get compatibility of the Virasoro constraints with respect to (virtual) projective bundles.

\begin{corollary}\label{cor: MimpliesP}
Let $f\colon P\to M$ be a virtual projective bundle as described before. Then the sheaf Virasoro constraints on $M$ imply the pair Virasoro constraints on $P$, i.e.,
\[[M]^\vir\in \widecheck V_0\Rightarrow [P]^\vir_{(\CO,\BF)}\in  V_0^{\pa}\,.\]
If $f\colon P\to M$ is actually a smooth projective bundle (i.e., $R^1p_\ast \BG=0$) we have the converse implication
\[ [P]^\vir_{(\CO,\BF)}\in  V_0^{\pa} \Rightarrow[M]^\vir\in \widecheck V_0\,.\]
\end{corollary}
\begin{proof}
The first implication follows from the first formula in Theorem \ref{thm: projectivebundleformula}, Proposition \ref{prop: wallcrossingcompatibility2} and the straightforward fact that $e^{(1,0)}\in P_0^{\pa}$. The converse implication follows from Theorem \ref{lem: projectivebundle}.\qedhere
\end{proof}

\begin{remark}
For the second implication, we really need $f$ to be a smooth projective bundle. This is due to the fact that in the proof of Theorem \ref{lem: projectivebundle} we used that $\T_{\Pi_\alpha}$ is a vector bundle of rank $\chi(\alpha)-1$, see Remark \ref{rem: abstractprojective}. Indeed, if the moduli space $M$ is such that $p_\ast \BG=0$ (e.g. if $M=M_{\alpha(mH)}$ for $m$ sufficiently negative) then $P$ is empty but the Virasoro constraints on $M$ are non-trivial.
\end{remark}

\subsection{Nested Hilbert scheme}
\label{sec: nestedhilbertscheme}
We now treat the base cases for parts (2), (3) of Theorem \ref{thm: main}. Let $S$ be a surface with $h^{0,1}=h^{0,2}=0$. Let $S_{\beta}^{[0,n]}$ be the nested Hilbert scheme as in \cite{gsy}. It parametrizes a pair of subschemes 
\[Z\subseteq E\subseteq S\]
where $E$ is a divisor in class $\beta$ and $Z\subseteq E$ is a 0 dimensional subscheme of length~$n$.
We have universal subschemes
\[\CZ\subseteq \CE\subseteq\CS\,,\] 
where we use $\CS=S\times S_\beta^{[0,n]}$. As explained in Section \ref{sec: limits}, the nested Hilbert scheme $S_{\beta}^{[0,n]}$ can be seen as a moduli of Bradlow pairs in 2 ways, by looking at a point $(E, Z)\in S_\beta^{[0,n]}$ either as
\[\CO_S\to I_Z(E)\quad\textup{ or}\quad\CO_S\to \CO_E(Z)\,.\]
That is,
\begin{align*}S_{\beta}^{[0,n]}&\cong P_{\left(1,\beta, -n+\beta^2/2\right)}^\infty\\
&\cong P_{\left(0,\beta, n-\beta^2/2\right)}^\infty\,.
\end{align*}
Each description comes with a natural universal pair, namely
\[\CO_{\CS}\to I_\CZ(\CE)\quad\textup{and}\quad\CO_{\CS}\to \CO_\CE(\CZ)\,.\]
The first description allows us to describe $S_\beta^{[0,n]}$ as a virtual projective bundle over the Hilbert scheme of points on $S$. Let $\alpha$ be such that $\ch(\alpha)=(1, \beta, -n+\beta^2/2)$. Since $\alpha$ does not decompose as $\alpha_1+\alpha_2$ with $r(\alpha_i)>0$, a pair $\CO_S\to F$ is $\mu^t$-(semi)stable if and only if $F$ is torsion free if and only if $F$ is stable, so the moduli space $P^t_\alpha$ does not change with $t$ and we have a map
\[f\colon P_\alpha^\infty=P_\alpha^{0+}\to M_\alpha\,.\]
Since $h^{0,1}=0$ there exists a unique line bundle $L_\beta$ with $c_1(L_\beta)=\beta$. Hence,
\[M_\alpha=\{I_Z\otimes L_\beta\colon Z\subseteq X\textup{ is }0\textup{ dimensional of length }n\}\cong M_{(1,0,-n)}=S^{[n]}\,.\]

Note that the deformation theory of $M_{\alpha}$ is smooth (see e.g. \cite[Proposition 2.2]{EGL}), so under the isomorphism above the virtual fundamental class $[M_\alpha]$ coincides with the usual fundamental class $[S^{[n]}]$.

We claim that if $S_{\beta}^{[0,n]}$ is not empty then the map $f$ is a virtual projective bundle as described in the previous section. For this we need to show that if $F=I_Z\otimes L_\beta\in M_\alpha$ then $H^2(F)=0$. If $S_{\beta}^{[0,n]}$ is not empty then there must exist a divisor $E\subseteq S$ in class $\beta$. Considering the long exact sequence on cohomology obtained from
\[0\to \CO_S\to \CO_S(E)\cong L_\beta\to \CO_E(E)\to 0\]
and using that $H^2(\CO_S)=0$ it follows that $H^2(L_\beta)=0$. Then the long exact sequence on cohomology associated to 
\[0\to I_Z\otimes L_\beta\to L_\beta\to \CO_Z\otimes L_\beta\to 0\]
shows that $H^2(I_Z\otimes L_\beta)=0$.

\begin{proposition}\label{prop: nestedhilb}
The nested Hilbert scheme $S_\beta^{[0,n]}$ satisfies the Virasoro constraints with either of the two descriptions as a pair moduli space, that is, 
\[ \Big[S_\beta^{[0,n]}\Big]^\vir_{(\CO, I_{\CZ}(\CE))}\in P_0^{\pa}\quad\textup{ and } \quad \Big[S_\beta^{[0,n]}\Big]^\vir_{(\CO, \CO_\CE(\CZ))}\in P_0^{\pa}\,.\]
\end{proposition}
\begin{proof}
We begin with the first statement. It was proven in \cite[Theorem 5]{moreira} that the Hilbert scheme $S^{[n]}\cong M_{(1,0, -n)}$ satisfies Virasoro constraints; see Remark \ref{rem: comparingnotation} and Proposition \ref{prop: normalizedvsinv} for a comparison between the formulation in loc. cit. and ours. By Lemma \ref{lem: twist}, it follows that Virasoro constraints hold for $M_\alpha$ for any $\alpha$ of rank~1. By Corollary \ref{cor:  MimpliesP} and the discussion preceeding this Proposition,
\[\Big[S_\beta^{[0,n]}\Big]^\vir_{(\CO, I_{\CZ}(\CE))}=[P_\alpha^\infty]^\vir_{(\CO,\BF)}=[P_\alpha^{0+}]^\vir_{(\CO,\BF)} \in P_0^{\pa}\,.\]
We now deduce the second statement from the first. The dual of $\CO_\CE(\CZ)$ in $K$-theory can be computed to be
\[\CO_\CE(\CZ)^\vee=-\CO_\CE(-\CZ)\otimes \CO_\CS(\CE)=-I_\CZ(\CE)+\CO_\CS\,.\]
In the first equality we used \cite[Example 3.41]{huyFM} and in the second we used
\[\CO_\CE(-\CZ)=\CO_\CE-\CO_\CZ=-(\CO_\CS(-\CE)-I_\CZ)\,.\]
As a consequence,
\[(\CO_\CE(\CZ)-\CO_\CS)^\vee=-I_\CZ(\CE)\,.\]

Define the involution $\bI\colon \BD^{X, \pa}\to \BD^{X, \pa}$ by
\begin{align*}
\bI\big(\ch_i^{\HH, \CF}(\gamma)\big)&=-(-1)^{i-p}\ch_i^{\HH, \CF-\CV}(\gamma)\,,\\
\bI\big(\ch_i^{\HH, \CF-\CV}(\gamma)\big)&=-(-1)^{i-p}\ch_i^{\HH, \CF}(\gamma)\,.
\end{align*}
By the previous computation of duals, we have
\[\xi_{(\CO, \CO_\CE(\CZ))}=\xi_{(\CO, I_\CZ(\CE))}\circ \bI\,.\]
We can see straight from the definition of the pair Virasoro operators $\bL_k^{\pa}$ (see Section \ref{sec: virasoropairs}) that
\[\bI\circ \bL_k^{\pa}=(-1)^k \bL_k^{\pa}\circ \bI\,.\]
With all these observations, the equivalence between the two statements becomes clear. Indeed,
\begin{align*}\int_{\big[S_\beta^{[0,n]}\big]^\vir_{(\CO, \CO_\CE(\CZ))}} \bL_k^{\pa}(D)&=\int_{\big[S_\beta^{[0,n]}\big]^\vir} \xi_{(\CO, \CO_\CE(\CZ))}(\bL_k^{\pa}(D))\\
&=\int_{\big[S_\beta^{[0,n]}\big]^\vir} \xi_{(\CO, I_\CZ(\CE))}( \bI(\bL_k^{\pa}(D)))\\
&=(-1)^k \int_{\big[S_\beta^{[0,n]}\big]^\vir} \xi_{(\CO, I_\CZ(\CE))}( \bL_k^{\pa}(\bI(D)))\\
&=(-1)^k \int_{\big[S_\beta^{[0,n]}\big]^\vir_{(\CO, I_\CZ(\CE))}}\bL_k^{\pa}(\bI(D))=0\,.
\end{align*}
This shows that $\Big[S_\beta^{[0,n]}\Big]^\vir_{(\CO, \CO_\CE(\CZ))}\in P_0^{\pa}$ as well and finishes the proof.
\qedhere
\end{proof}

\begin{appendices}
    \section{Joyce-Song wall-crossing and punctual Quot schemes}
\label{app:A}
In Sections \ref{sec: rankreduction} and \ref{sec:lowrank} we have explored equivalences between Virasoro constraints in different moduli spaces that ultimately led to the proof of Theorem \ref{thm: main}. Some of these equivalences are shown in the following diagram:
\begin{equation}
\label{Eq:pairsheafWC}
    \begin{tikzcd}
    \, & \arrow[ld,bend right=10, "\textup{JS}", swap] M_{n} &\, &\,& \, \arrow[ld,bend right=10, "\textup{JS}", swap] M_{\beta, n}\\
    C^{[n]}\arrow[ru, bend right=10, "\textup{RR}", swap]\arrow[rd, bend right=10, "\textup{PB}", swap]& \,&\, & \, S_\beta^{[0,n]} \arrow[ru, bend right=10, "\textup{RR}", swap]\arrow[rd, bend right=10, "\textup{PB}", swap] &\, \\
    \, & \Jac(C)  \arrow[lu,bend right=10, "\textup{PB}", swap]  &\, &\, & \, S^{[n]}\arrow[lu,bend right=10, "\textup{PB}", swap] 
    \end{tikzcd}
\end{equation}
where $M_n$ and $M_{\beta,n}$ denote the moduli of zero-dimensional and one-dimensional sheaves, respectively. The labels RR, JS and PB stand for rank reduction (Theorem \ref{thm: rankreduction}), Joyce-Song wall-crossing and projective bundle compatibility, respectively. The rank reduction argument was the main content of Section \ref{sec: rankreduction}. The projective bundle compatibility was explained in Section \ref{sec: virtualprojectivebundle} and used in Section \ref{sec: nestedhilbertscheme} to show that the Virasoro constraints hold for nested Hilbert schemes. This appendix concerns the remaining arrow: Joyce-Song wall-crossing expresses the virtual fundamental class of moduli of pairs in terms of the classes of moduli of sheaves. 

 A general formulation of Joyce-Song wall-crossing formula is proved in \cite[Theorem A.4]{Bo21.5}, and we recall it below, see \eqref{Eq: PJS}; in the case of moduli of vector bundles on curves it appears already in \cite[Theorem 2.8]{Bu22}. For the symmetric power of curves the consequence of \eqref{Eq: PJS}  is
\begin{equation}\label{eq: JSsymmetricpowers}
\big[C^{[n]}\big]_{(1,\CO_\CE)}=\sum_{\substack{ \underline{n}\vdash n}}\frac{1}{l!}\, \big[[M_{n_1}]^\inva, \big[\ldots, \big[[M_{n_l}]^\inva, e^{(1,0)}\big]\ldots\big]\big]\end{equation}
where $e^{(1,0)}\in V_\bullet^{\pa}$ is the class of a point $\{(\CO_X, 0)\}$ in the component $\CP_{(1,0)}$ and $M_n$ is the moduli space of 0-dimensional sheaves of length $n$. 

In particular, this formula together with Lemma \ref{Lem: Mnp}, which shows that $[M_{n}]^\inva\in \widecheck{P}_0$, gives an alternative proof of the Virasoro constraints for the symmetric power of a curve (Proposition \ref{prop: symmetricpowers}). More generally, the same approach proves the Virasoro constraints for punctual Quot schemes $\Quot_X(V,n)$ on curves and surfaces, which are of independent interest and for which we do not know a direct proof without wall-crossing techniques. We also remark that the first part of Proposition \ref{thm: projectivebundleformula} is a special case of the Joyce-Song wall-crossing.

\subsection{Joyce-Song pairs and wall-crossing}

\label{sec: joycesong}
Joyce-Song stable pairs are named after their first appearance in the work of Joyce and Song \cite[Section 12.1]{JS12}. Let $X$ be a curve or a surface, $\alpha\in C(X)$ and let $V$ be a fixed torsion-free sheaf on $X$ such that 
\begin{equation}\label{eq: JSextvanish}
\Ext^{\geq 2}(V,F)=0\,.
\end{equation}
Recall that $p_F(z)$ denotes the reduced Gieseker polynomial of a sheaf $F$ where we will also use the notation $p_{\alpha}(z):=p_F(z)$. Joyce-Song pairs consist of a sheaf $F$ and a morphism
$$
V \xrightarrow{s} F
$$
with the following stability condition:
\begin{enumerate}
    \item $F$ is Gieseker-semistable;
    \item $s\neq 0$; 
    \item there is no $0\neq G\subsetneq F$ with $p_F(G)=p_F(z)$ such that $\mathrm{im}(s)\subseteq G$.
\end{enumerate}

We denote the resulting moduli space by $P^{\JS}_{V,\alpha}$. It has an obstruction theory at each $[V\to F]\in P^{\JS}_{V,\alpha}$ given by
$$
\RHom([V\to F], F)\,,
$$
which is 2-term by \eqref{eq: JSextvanish}. Denote by
$[P^{\JS}_{V,\alpha}]^\vir$ the resulting virtual fundamental class.

When $V=\CO_X$, this stability condition is essentially the one in Proposition \ref{prop: limitt0}, except there stability is given in terms of $\mu$-stability and we now work with Gieseker stability because of a technicality explained in \cite[Rem. A.3]{Bo21.5}. This makes no difference in our applications, because we either work with sheaves supported in dimension $\leq 1$ or ideal sheaves, in which cases $P^{\JS}_{\CO_X, \alpha}$ coincides with the previously defined $P^{\JS}_{1,\alpha}$. Note also that Assumption \ref{ass: alphabig} implies the vanishing $\Ext^{\geq 2}(\CO_X,F)=0$. As explained in \cite{Bo21.5} (see Definition A.1 and the discussion afterwards), one can construct a 1-parameter family of stability conditions that connects the Joyce-Song stability to a stability condition in which the stable objects are $U\otimes V\to 0$ for some vector space $U$ and $0\to F$ for some Gieseker stable $F$. 

The author then used Joyce's general theory to show the following wall-crossing formula for any $X=C,S$ (cf.  \cite[Theorem A.4]{Bo21.5}):
 $$[P^{\JS}_{V,\alpha}]^\vir= \sum_{\substack{\underline{\alpha} \vdash \alpha\\ p_{\alpha_i}=p_{\alpha}}}\frac{1}{l!}\, \big[[M_{\alpha_1}]^\inva, \big[\ldots, \big[[M_{\alpha_l}]^\inva, e^{(\llbracket V\rrbracket,0)}\big]\ldots\big]\big]\,,$$
 where $e^{(\llbracket V\rrbracket,0)}\in V_\bullet^{\pa}$ is the class of a point in the component $\CP_{(\llbracket V\rrbracket ,0)}$. This formula holds in $\widecheck V_\bullet^{\pa}$, but using Lemma \ref{lem: liftvertexalgebra} and the same argument that appears after the Lemma, we can lift it to $V_\bullet^{\pa}$:
\begin{equation}
\label{Eq: PJS} [P^{\JS}_{V,\alpha}]^\vir_{(q^*V,\BF)}= \sum_{\substack{\underline{\alpha} \vdash \alpha\\ p_{\alpha_i}=p_{\alpha}}}\frac{1}{l!}\, \big[[M_{\alpha_1}]^\inva, \big[\ldots, \big[[M_{\alpha_l}]^\inva, e^{(\llbracket V\rrbracket,0)}\big]\ldots\big]\big]\,.
\end{equation}
We recall the argument: clearly $\ch_1^{\CV}(\pt)\cap$ annihilates both  $[P^{\JS}_{V,\alpha}]^\vir_{(q^*V,\BF)}$ and $e^{(\llbracket V\rrbracket,0)}$, so by Lemma \ref{lem: liftvertexalgebra} a) it annihilates both sides of \eqref{Eq: PJS} and by part b) of the same Lemma (together with the original wall-crossing formula in $\widecheck V_\bullet^{\pa}$), we conclude \eqref{Eq: PJS}.

The necessary assumptions for wall-crossing formulae to hold, similar to the ones alluded to in Assumption \ref{ass:WCpair}, were explicitly checked in \cite[Appendix]{Bo21.5}.

\begin{remark}\label{rem: JSprojectivebundle}
    Suppose that $\alpha$ is such that there are no Gieseker strictly semistable sheaves of type $\alpha$. Then condition $(3)$ in the definition of Joyce-Song pairs is vacuous and $P^{\JS}_{V,\alpha}$ admits a description as a virtual projective bundle (cf. Section \ref{sec: virtualprojectivebundle})
    \[P^{\JS}_{V,\alpha}=\BP_{M_\alpha}\big(Rp_\ast\big(q^\ast V^\vee\otimes \BG\big)\big)\,.\]
    
   Under this assumption, there are no decompositions $\alpha=\alpha_1+\ldots+\alpha_l$ with $p_{\alpha_i}=p_{\alpha}$, $M_{\alpha_i}\neq \emptyset$ and $l>1$. In this case, \eqref{Eq: PJS} becomes
    \[[P^{\JS}_{V,\alpha}]^\vir_{(q^*V,\BF)}=\big[[M_\alpha]^\inva, e^{(\llbracket V\rrbracket,0)}\big]\,.\]
When $V=\CO_X$ this is precisely the formula \eqref{eq: projectivebundlebracket}, for which we gave a direct proof. 
\end{remark}

\subsection{Application to Quot schemes}

One special case of Joyce-Song pairs are punctual Quot schemes, and we can use \eqref{Eq: PJS} to give a proof of the Virasoro constraints for them. We currently do not have a direct proof of these without wall-crossing techniques. 

For $X=C,S$ and a fixed torsion-free sheaf $V$ on $X$, the Quot scheme $\Quot_X(V,n)$ for $X=C,S$ parameterizes equivalence classes of surjective morphisms $V\to F$ from $V$ to a zero-dimensional sheaf $F$. When $V$ is a vector bundle, the virtual fundamental classes 
$$
\big[\Quot_X(V,n)\big]^\vir
$$
were constructed by Marian-Oprea-Pandharipande \cite[Lemma 1.1]{MOP1} (see Stark \cite[Proposition 5]{stark2} for a more detailed proof). It was remarked in \cite[Section 1.1]{Bo21.5} that the same obstruction theory given at each $[V\to F]\in \Quot_X(V,n)$ by
$$
\RHom([V\to F], F)
$$
is perfect of tor-amplitude $[-1,0]$ whenever $V$ is more generally torsion-free. To apply \eqref{Eq: PJS}, we use the identifications of moduli spaces and virtual fundamental classes
$$
P^{\JS}_{V,n} = \Quot_X(V,n)\,,\qquad [P^{\JS}_{V,n}]^\vir = \big[\Quot_X(V,n)\big]^\vir
$$
following immediately from their descriptions above; note that we are writing $n$ for the $K$-theory class of the sheaf $\CO_{\pt}^{\oplus n}$.
To prove the Virasoro constraints for Quot schemes, we first show that the invariant classes counting zero-dimensional sheaves satisfy them. Note that $M_n$ has virtual dimension 1, which makes a direct proof accessible. 

\begin{lemma}
\label{Lem: Mnp}
    For any $n>0$ and any $X=C,S$, the class $[M_{n}]^\inva$ is primary, i.e.,
$$
[M_n]^\inva\in \widecheck{P}_0\,.
$$
\end{lemma}
\begin{proof}
By Proposition \ref{prop: primaryomegabracket}, it is sufficient to prove that 
$$
\int_{[M_{n}]^{\overline{\inva}}} (\bL_k - \delta_{k,0})(D) =0\quad \textnormal{for all}\quad k\geq 0,\ D\in \BD^X\,,
$$
where $[M_{n}]^{\overline{\inva}}\in V_\bullet$ is the unique lift of $[M_n]^\inva\in \widecheck V_\bullet$ satisfying $\ch^\HH_1(1)\cap [M_{n}]^{\overline{\inva}}$. Since $\bL_0$ acts as the multiplication by the conformal degree 1, the case $k=0$ follows. On the other hand $\bR_{1}(D)$ annihilates $[M_{n}]^{\overline{\inva}}$ by degree reasons. In conclusion, it suffices to prove that $\int_{[M_{n}]^{\overline{\inva}}} \bT_{1}=0$. 

Recall the definition of $\bT_1$ from Section \ref{sec: Virarosorops} leading to
\begin{align*}
    \bT_1=\sum_s (-1)^{\dim(X)}\Big[(-1)^{p_s^L}+(-1)^{p_s^R}\Big]\,\ch_0^\HH(\gamma_s^L)\ch_1^\HH(\gamma_s^R)
\end{align*}
where $\Delta_*(\td(X))=\sum_s \gamma_s^L\otimes \gamma_s^R$. Therefore it suffices to consider the Kunneth components satisfying $|p_s^L|=|p_s^R|$ which is further restricted by $p_s^L+p_s^R\geq \dim(X)$. On the other hand, $\ch_0^\HH(-)$ has a property that (after realization) 
$$\ch_0^\HH(\gamma_s^L)=\begin{cases}
\int_X \gamma_s^L\cup n\pt &\textnormal{if}\ \  p_s^L=q_s^L=0\,,\\
0&\textnormal{if}\ \ p_s^L=q_s^L>0 \ \textnormal{ or } \ p_s^L>q_s^R\,.
\end{cases}
$$
These vanishing properties are enough to prove that $\int_{[M_{n}]^{\overline{\inva}}} \bT_{1}=0$ when $X=C$. 

When $X=S$, we additionally need that 
\begin{equation}\label{eq: extra vanishing}
    \ch_0^\HH(\gamma^{1,2})\cap [M_{n}]^{\overline{\inva}} = \ch_1^\HH(\gamma^{2,\bullet})\cap [M_{n}]^{\overline{\inva}}=0
\end{equation}
for all $\gamma^{1,2}\in H^{1,2}(X)$ and $\gamma^{2,\bullet}\in H^{2,\bullet}(X)$. This follows from \cite[Lemma 4.2]{Bo21.5}, but we make the argument used to prove it explicit in terms of descendents. Note that both $\ch_0^\HH(\gamma^{1,2})$ and $\ch_1^\HH(\gamma^{2,\bullet})$ use the first Chern character $\ch_1(\CF)$ of the universal complex over $\CM_{n\pt}\times X$.\footnote{Recall that $\CM_{n\pt}$ is the stack of all perfect complexes with class $n\pt$.} By the construction of invariant classes, they lie in the image of the pushforward map 
$$\iota_*:H_\bullet(\mathcal{N}_{n\pt})\rightarrow H_\bullet(\CM_{n\pt})\,,
$$
where $\iota:\mathcal{N}_{n\pt}=\mathcal{N}_{(0,n\pt)}\hookrightarrow \CM_{n\pt}$ denotes the open immersion from the stack of zero-dimensional coherent sheaves of length $n$. Then the extra vanishings \eqref{eq: extra vanishing} follow from the geometric fact that $(\iota\times \id_X)^*\mathcal{F}$ is the universal zero-dimensional sheaf on $\mathcal{N}_{n\pt}\times X$ hence $(\iota\times \id_X)^*\ch_1(\mathcal{F})=0$. 
\end{proof}
We now conclude the precise version of Theorem \ref{thm: punctualquot}. 
Because of
$$
\Quot_C(\CO_C,n)=C^{[n]}\,,
$$
Proposition \ref{prop: symmetricpowers} is a consequence of this more general result.
\begin{theorem}
\label{prop: quot}
If $X=C$ or $X=S$ with $h^{2,0}(S)=0$, punctual Quot schemes $\Quot_X(V,n)$ satisfy pair Virasoro constraints, i.e.,
$$\big[\Quot_X(V,n)\big]^\vir_{(p^*V,\BF)}\in P_0^{\pa}\,.
$$
\end{theorem}
\begin{proof}
As a corollary of \eqref{Eq: PJS}, we obtain the wall-crossing formula for Quot scheme involving invariant classes of zero-dimensional sheaves:
$$
    \big[\Quot_X(V,n)\big]^\vir_{(q^*V,\BF)} = \sum_{\begin{subarray}a \underline{n} \vdash n\end{subarray}}\frac{1}{l!}\, \big[[M_{n_1}]^\inva, \big[\ldots, \big[[M_{n_l}]^\inva, e^{(\llbracket V\rrbracket,0)}\big]\ldots\big]\big]\,. 
$$
Using Lemma \ref{Lem: Mnp} and Proposition \ref{prop: wallcrossingcompatibility2}, we conclude the result.
\end{proof}

    \end{appendices}

\bibliographystyle{mybstfile.bst}
\bibliography{refs.bib} 
\end{document}